\documentclass[11pt]{article}

\usepackage[margin=1in]{geometry}
\usepackage{amsmath,amssymb,amsthm,mathtools,bm,mathrsfs,euscript}
\usepackage[Symbolsmallscale]{upgreek}
\usepackage{enumitem,microtype,booktabs,array}
\usepackage[colorlinks=true,linkcolor=blue,citecolor=blue,urlcolor=blue]{hyperref}
\usepackage{aliascnt}
\usepackage[nameinlink,capitalise]{cleveref}
\allowdisplaybreaks
\numberwithin{equation}{section}
\renewcommand{\le}{\leqslant}
\renewcommand{\leq}{\leqslant}
\renewcommand{\ge}{\geqslant}
\renewcommand{\geq}{\geqslant}

\newtheorem{theorem}{Theorem}[section]
\newtheorem*{informalthm}{Informal theorem}
\newaliascnt{proposition}{theorem}
\newtheorem{proposition}[proposition]{Proposition}
\aliascntresetthe{proposition}
\newaliascnt{lemma}{theorem}
\newtheorem{lemma}[lemma]{Lemma}
\aliascntresetthe{lemma}
\newaliascnt{corollary}{theorem}

\aliascntresetthe{corollary}
\theoremstyle{definition}
\newaliascnt{definition}{theorem}

\aliascntresetthe{definition}
\newaliascnt{algorithm}{theorem}
\newtheorem{algorithm}[algorithm]{Algorithm}
\aliascntresetthe{algorithm}
\newaliascnt{assumption}{theorem}

\aliascntresetthe{assumption}
\newtheorem*{assumption*}{Assumption}
\theoremstyle{remark}
\newaliascnt{remark}{theorem}
\newtheorem{remark}[remark]{Remark}
\aliascntresetthe{remark}

\crefname{algorithm}{Algorithm}{Algorithms}
\Crefname{algorithm}{Algorithm}{Algorithms}

\newcommand{\R}{\mathbb R}
\newcommand{\E}{\mathbb E}

\newcommand{\Id}{I}
\newcommand{\cN}{\mathsf N}
\newcommand{\dd}{\mathrm d}
\newcommand{\deq}{\coloneqq}
\newcommand{\eps}{\varepsilon}
\newcommand{\grad}{\nabla}
\newcommand{\Hess}{\nabla^2}
\newcommand{\norm}[1]{\lVert #1\rVert}
\newcommand{\TV}{\mathsf{TV}}
\newcommand\msf[1]{\mathsf{#1}}
\DeclareMathOperator\KL{\msf{KL}}
\newcommand\Ren{\mathsf R}
\newcommand\mmid{\mathbin{\|}}
\newcommand{\W}{W_2}
\DeclareMathOperator{\law}{Law}
\newcommand{\prox}{\operatorname{prox}}
\newcommand{\polylog}{\operatorname{polylog}}
\newcommand{\wtO}{\widetilde O}
\newcommand{\one}{\mathbf 1}

\newcommand{\BPS}{\mathrm{BPS}}
\newcommand{\ip}[2]{\langle #1,#2\rangle}
\newcommand{\HS}{\mathrm{HS}}
\newcommand{\op}{\mathrm{op}}
\newcommand{\T}{\mathsf T}
\newcommand{\LHam}{\mathcal L_{\mathsf H}}
\newcommand{\Lip}{\operatorname{Lip}}
\newcommand{\Ebias}[1][]{\mathcal E_{{\rm bias}\if\relax\detokenize{#1}\relax\else,#1\fi}}
\newcommand{\Evar}[1][]{\mathcal E_{{\rm var}\if\relax\detokenize{#1}\relax\else,#1\fi}}
\newcommand{\Edisc}[1][]{\mathcal E_{{\rm disc}\if\relax\detokenize{#1}\relax\else,#1\fi}}

\usepackage{xcolor}

\title{Smoothed Picard Hamiltonian Monte Carlo}
\author{
 Fan Chen\thanks{Department of Electrical Engineering and Computer Science,
 Massachusetts Institute of Technology.
 Email: \href{mailto:fanchen@mit.edu}{\texttt{fanchen@mit.edu}}.}
 \and
 Sinho Chewi\thanks{Department of Statistics and Data Science, Yale University.
 Email: \href{mailto:sinho.chewi@yale.edu}{\texttt{sinho.chewi@yale.edu}}.}
 \and
 Jianfeng Lu\thanks{Department of Mathematics, Duke University.
 Email: \href{mailto:jianfeng@math.duke.edu}{\texttt{jianfeng@math.duke.edu}}.}
 \and
 Matthew S. Zhang\thanks{Department of Mathematics, Massachusetts Institute of Technology.
 Email: \href{mailto:shuns436@mit.edu}{\texttt{shuns436@mit.edu}}.}
}
\date{August 2026}

\begin{document}
\maketitle

\begin{abstract}
    We develop a new low-accuracy sampler, called \emph{smoothed Picard Hamiltonian Monte Carlo}, which combines Gaussian smoothing, Picard iteration, and higher-order discretization.
    For a log-concave target $\pi \propto \exp(-V)$ in dimension $d$ satisfying $0 \prec \alpha I \preceq \nabla^2 V \preceq \beta I$, with condition number $\kappa \deq \beta/\alpha$, smoothed Picard HMC returns a sample with $\sqrt \alpha\,W_2(\cdot,\pi) \le \varepsilon$ using $\widetilde O(\kappa^2 + \kappa^{7/6} d^{1/6}/\varepsilon^{1/3})$ gradient queries.
    We also prove stronger $W_q$ bounds, and then develop an algorithmic framework, the recursive warm start generator, to upgrade these $W_q$ bounds to stronger divergence guarantees.
    This produces a warm start for the proximal bouncy particle sampler, introduced in a companion work, leading to a high-accuracy log-concave sampler with complexity $\widetilde O((\kappa^{7/6} d^{1/6} + \kappa^{1/2} d^{1/4})\polylog(1/\varepsilon))$.
\end{abstract}

\tableofcontents

\section{Introduction}
\label{sec:introduction}

We study the problem of sampling from a target probability density $\pi \propto \exp(-V)$ over $\R^d$, where
\(V\in C^2(\R^d)\) satisfies, for \(0<\alpha\le\beta\),
\begin{equation}
  \alpha\Id\preceq\Hess V(x)\preceq\beta\Id\,,
  \qquad x\in\R^d\,,
  \qquad \kappa\deq \frac\beta\alpha\,.
  \label{eq:curvature}
\end{equation}
We design a new sampler that combines Gaussian smoothing of the density with a high-order
discretization of Hamiltonian dynamics via Picard iteration and Chebyshev--Lobatto quadrature.  This combination gives a stable unadjusted Hamiltonian step even
though the original potential is assumed to have only two derivatives.

Our first result is a low-accuracy sampler with Wasserstein guarantee. \begin{theorem}[$W_2$ guarantee]
\label{thm:w2-main}
Assume \eqref{eq:curvature}, let \(0<\eps\le1/2\), and suppose that we are given a point
\(x_{\rm ref}\) with
\(\norm{\grad V(x_{\rm ref})}\le\sqrt{\alpha d}\).
Then, the smoothed Picard HMC algorithm
(\cref{alg:w2-complete}) returns a sample from a probability law $\widehat \pi$ satisfying
\begin{equation*}
 \sqrt\alpha\,W_2(\widehat\pi,\pi)\le\eps\,.
\end{equation*}
The expected number of gradient and proximal queries is at most
\begin{equation*}
 O\Bigl(
   \Bigl\{\kappa^2
   +\frac{\kappa^{7/6}d^{1/6}}{\eps^{1/3}}\Bigr\}
   \log^4\frac{\kappa d}{\eps}
 \Bigr)\,.
\end{equation*}
\end{theorem}

We find the dimension dependence in this result to be quite surprising.
Prior to this work, under~\eqref{eq:curvature}, the best dimension dependence was $d^{1/3}$.
This was achieved with the randomized midpoint discretization of the
underdamped Langevin, with rate
$\widetilde O(\kappa d^{1/3}/\varepsilon^{2/3})$ in $W_2$ \cite{SL19}, and also the shifted ODE method~\cite{FLO21}.
With the development of the shifted composition framework~\cite{ACZ26} and accelerated entropic hypocoercivity~\cite{Lu26, LL26}, the rate was improved to $\widetilde O(\kappa^{5/6} d^{1/3}/\varepsilon^{2/3})$ in $\KL$. The closely related Poisson
midpoint discretization of underdamped Langevin achieved an incomparable rate of roughly $\widetilde O(\kappa^{7/6} d^{1/3}/\varepsilon^{1/3} + \kappa^{11/8} d^{3/8}/\varepsilon^{1/4})$ in $W_2$~\cite{SN26}.  More recently, exact simulation of the underdamped Langevin
diffusion via FORS achieved
$\widetilde O(\kappa^{2/3}d^{1/3}\polylog(1/\eps))$ complexity
\cite{CCRZ26}.
In contrast,~\cref{thm:w2-main} achieves dimension dependence $d^{1/6}$.

\smallskip
In a companion work~\cite{PBPS26}, we introduce a new high-accuracy sampler, called the proximal bouncy particle sampler or \emph{proximal BPS}, with $\widetilde O(\kappa^{1/2} d^{1/4}\polylog(1/\varepsilon))$ complexity from a warm start. One key application of smoothed Picard HMC is to provide an implementable warm start for that sampler.

Toward this end, we develop a novel algorithmic framework which we call the recursive warm start generator. It is designed as a wrapper around smoothed samplers---that is, samplers which target the Gaussian smoothed distribution $\pi * \mathsf N(0,\eta I)$---and it upgrades transport metric guarantees into stronger divergence guarantees. Namely, with negligible overhead, it upgrades $W_2$ bounds into $\KL$ bounds; $W_q$ bounds into mixed TV and R\'enyi bounds; and $W_{\psi_2}$ bounds into genuine R\'enyi bounds. We apply this to smoothed Picard HMC by extending the arguments for~\cref{thm:w2-main} to produce stronger $W_q$ guarantees, leading to the following warm start result.

\begin{theorem}[Warm start]
\label{thm:adaptive-gibbs-warm}
Assume \eqref{eq:curvature} and that we are given \(x_{\rm ref}\) satisfying
\(\norm{\grad V(x_{\rm ref})}\le\sqrt{\alpha d}\).  For every
\(0<\delta<1/2\), the algorithm of
\cref{ssec:recursive_overview} returns a sample from a law
\(\widehat\pi_\delta\) with the following guarantee.
There exists a law \(\widehat\pi_\delta^\dagger\) satisfying
\begin{equation*}
 \TV(\widehat\pi_\delta,\widehat\pi_\delta^\dagger)\le\delta\,,
 \qquad
 \Ren_2(\widehat\pi_\delta^\dagger \mmid \pi)\le1\,.
\end{equation*}
Moreover, the expected number of gradient queries is at most
\begin{equation*}
    O\Bigl(
      \kappa^{7/6}d^{1/6}
      \log^7\frac{\kappa d}{\delta}
    \Bigr)\,.
\end{equation*}
\end{theorem}

In \cref{thm:picard-warm-start}, we prove a stronger statement that relaxes the assumption of strong log-concavity to a log-Sobolev inequality, and provides a KL divergence guarantee.
Note that the recursive warm start generator also removes the $\kappa^2$ term from the complexity.

Finally, by using this warm start result to initialize proximal BPS~\cite{PBPS26}, we obtain a new high-accuracy sampler which enjoys the following guarantee.

\begin{theorem}[Warm start + proximal BPS]
\label{thm:main-synthesis}
Assume \eqref{eq:curvature} and suppose that we are given a point $x_{\rm ref}$ with
\(\norm{\grad V(x_{\rm ref})}\le\sqrt{\alpha d}\).  For every
\(0<\varepsilon<1/4\), the warm start
algorithm followed by the proximal BPS algorithm returns a sample from a law
\(\widehat\pi\) satisfying
\begin{equation*}
 \TV(\widehat\pi,\pi)\le\varepsilon\,.
\end{equation*}
The expected number of gradient evaluations is
\begin{equation*}
 O\Bigl(
     \kappa^{7/6}d^{1/6}\log^7\frac{\kappa d}{\varepsilon}
     +\kappa^{1/2}d^{1/4}\log^{11/4}\frac{\kappa d}{\varepsilon}
 \Bigr)\,.
\end{equation*}
\end{theorem}

Prior to this work, the best high-accuracy sampler was exact diffusion simulation~\cite{CCRZ26}, which achieved a complexity of $\widetilde O(\kappa^{2/3} d^{1/3}\polylog(1/\varepsilon))$. Hence, our new high-accuracy sampler has state-of-the-art dependence on the dimension, but not on the condition number.

\begin{remark}[Further consequences]
The proximal sampler reduction \cite{LST21,CCSW22} converts general
log-smooth sampling problems into well-conditioned log-concave subproblems.
By invoking this framework, we could also state corollaries in other settings, such as when $\pi$ is only assumed to satisfy a Poincar\'e inequality.
Since this reduction is standard, we omit these extensions for brevity.
\end{remark}

We give an overview of the main new ideas---both algorithmic and analytic---introduced in this work in the technical overview (\cref{sec:technical-overview}).

\paragraph{Related work.}
This work advances the study of the complexity of log-concave sampling; an
introduction to the subject and an overview of existing approaches can be
found in~\cite{Chewi26Book}.

Our use of Gaussian smoothing is reminiscent of score-based diffusion
models~\cite{Soh+15NonEquib,SonErm19Gradients,HoJaiAbb20DDPM,Song+21SDE} and stochastic localization~\cite{Eldan13ThinShell},
which also rely on score functions for Gaussian convolutions of a target distribution.
Here, we do not learn score
functions from data: we use queries to the unsmoothed score to construct
nearly unbiased estimators of the smoothed score, and control their bias
and fluctuations in the sampling analysis.

Hamiltonian Monte Carlo was introduced as hybrid Monte Carlo in lattice
field theory~\cite{DKPR87}; see~\cite{N11} for its development in statistics.
Early analyses quantified its favorable dimension dependence through
asymptotic scaling results for product targets~\cite{BPRSS13}.
Subsequent works established non-asymptotic convergence guarantees for
ideal HMC~\cite{LSV18,BEZ20,MS21,CheVem22HMC}.
More recent work has focused on accelerated convergence of randomized
HMC and related piecewise deterministic processes, including convergence
in relative entropy~\cite{LuWan22PDMP,MitSriWanWib26RHMC,MonWan26EntropicPDMP}.

A substantial literature studies numerical discretization, both for
unadjusted HMC~\cite{MV18,MS19,BE23,BRM25,GLBMM24,BMW26,ZhaAltChe26HMC}
and for Metropolized HMC~\cite{CDWY20,LeeSheTia20MALA,CheGatJia26MHMC,ZhaAltChe26HMC}.
Representative results achieve $d^{1/4}$ dimension dependence under
quantitative third derivative bounds, such as structured or
Frobenius-Lipschitz Hessian assumptions~\cite{MV18,CheGatJia26MHMC,ZhaAltChe26HMC}.
Higher-order integrators and polynomial collocation yield still better
dimension dependence under stronger smoothness and structural
assumptions~\cite{MS17,LSV18}.
Our departure is to improve the dimension dependence \emph{without} imposing
higher-order smoothness on the original potential: Gaussian smoothing
provides the regularity needed for accurate integration, while we control
the error from estimating the smoothed score using queries to the
unsmoothed score.

The recursive warm start framework builds heavily on the proximal
sampler~\cite{LST21,CCSW22}.
The main analytic mechanism for warm start generation is regularization
along the heat flow, leading to reverse transport inequalities.  Such inequalities are
dual formulations of dimension-free Harnack inequalities, originating in
Wang's work~\cite{Wang1997LSINonCompact,Wang04Equiv}; see also
\cite{BakGenLed15HarnackOT,AltChe25SCI} for the transport formulation.
Related regularization arguments have been used to generate algorithmic
warm starts with underdamped Langevin~\cite{AC24} and HMC~\cite{ZhaAltChe26HMC},
and to establish KL divergence guarantees via the shifted composition
framework~\cite{AltChe26SCIII,ACZ26}.
Coupling-based regularization has also yielded total variation
bounds~\cite{BE23} and divergence guarantees for unadjusted
HMC~\cite{BMW26}.  These works analyze regularization for particular
sampling dynamics.  Our framework instead applies heat flow regularization
to the output of any smoothed sampler, yielding a general route to KL, mixed TV and R\'enyi, or genuine
R\'enyi warm starts.

\paragraph{AI usage.}
The main ideas underlying the algorithm and its analysis emerged through extensive interactions with GPT-5.6 Sol. The initial goal was to improve the dimensional dependence beyond $d^{1/3}$. In response, GPT proposed considering the smoothed potential $\widehat{V}_\eta \deq \E_{G\sim \mathsf N(0,\Id)} V(\cdot +\sqrt{\eta}\, G)$, which led to a sampler with dimensional dependence $d^{1/4}$. We subsequently suggested the more natural smoothing adopted in this paper and, through further refinements of the analysis and extensive interactions between the authors and GPT, obtained the improved dimensional dependence and the algorithms presented here. The authors independently checked the arguments, prepared the manuscript, and take full responsibility for its contents.

\section{Preliminaries}
\label{sec:setting}

Except in the generic warm start result \cref{thm:recursive-warm-generator}
and its consequence \cref{thm:picard-warm-start}, throughout the remainder of the paper
\(V\in C^2(\R^d)\) and \eqref{eq:curvature} holds pointwise.
For most of the paper, we normalize $\beta = 1$, so that $\kappa^{-1} I \preceq \nabla^2 V \preceq I$; this can be achieved by rescaling.

For simplicity of exposition, we initially allow ourselves access to proximal queries to $V$, in addition to queries to $\nabla V$.
For step size \(a>0\), one proximal query returns
\begin{equation}
 \prox_{aV}(y)
 \deq \operatorname*{argmin}_{x\in\R^d}
 {\Bigl\{V(x)+\frac{\norm{x-y}^2}{2a}\Bigr\}}\,.
 \label{eq:prox-def}
\end{equation}
In this paper, we only ever need to call the proximal oracle for a $\beta$-smooth potential with step size at most $1/(2\beta)$.
In this regime, the computation of $\prox_{aV}$ is a well-conditioned optimization problem and can be easily implemented via standard routines.
In \cref{app:gradient-only-prox}, we remove the use of proximal queries altogether.

The sampling analogue of the proximal operator is the restricted Gaussian oracle (RGO), which is the distribution
\begin{equation*}
 R_{a,y}(\dd x)
 \propto
 \exp\Bigl\{-V(x)-\frac{\norm{x-y}^2}{2a}\Bigr\}\,\dd x\,.
\end{equation*}
In \cref{sec:fors-terminal}, we show how to approximately implement it in $\widetilde O(1+{(a^2 d)}^{1/3})$ queries, based on the exact diffusion sampler of~\cite{CCRZ26}.

We assume that an admissible reference point is supplied:
\begin{equation}
  \norm{\grad V(x_{\rm ref})}\le\sqrt{\alpha d}\,.
  \label{eq:reference}
\end{equation}
This reference point could be obtained via an optimization routine.

The total variation distance is
\(\TV(P,Q)\deq\sup_A|P(A)-Q(A)|\).
We also use
\begin{equation*}
 \KL(P \mmid Q)\deq\int\log\bigl(\frac{\dd P}{\dd Q}\bigr)\,\dd P\,,
 \qquad
 \W^2(P,Q)\deq\inf_{(X,Y)}\E\norm{X-Y}^2\,,
\end{equation*}
with the usual value \(+\infty\) for KL without absolute continuity; the
infimum in the definition of \(\W\) is over all couplings of \(P\) and \(Q\).
For \(q>1\), define R\'enyi divergence by
\begin{equation*}
  \Ren_q(P \mmid Q)
  \deq \frac1{q-1}\log\int
       \bigl(\frac{\dd P}{\dd Q}\bigr)^q\,\dd Q\,,
\end{equation*}
with value \(+\infty\) when \(P\not\ll Q\); this is the standard convention
of \cite{EH14}.

For matrices and tensors, \(\norm{\cdot}_{\op}\) and
\(\norm{\cdot}_{\HS}\) denote the operator and Hilbert--Schmidt (or Frobenius) norms,
respectively.  For an $r$-tensor $A$ on $\R^d$, this means
\[
 \norm A_{\HS}^2
 =\sum_{j_1,\ldots,j_r\in[d]}
   \bigl|A[e_{j_1},\ldots,e_{j_r}]\bigr|^2\,,
\]
where $(e_j)_{j\in[d]}$ is any orthonormal basis of $\R^d$.

The notation \(\wtO(B)\) means \(B\) times a fixed universal power of a
logarithm of the dimension, condition number, inverse accuracy, and any
explicit confidence parameter.  Throughout, \(C,c>0\) denote universal
constants that may change from line to line.

\section{Technical overview}
\label{sec:technical-overview}

Recall our convention that $\beta =1$, so that $\kappa^{-1} I\preceq\Hess V \preceq I$.

\subsection{The smoothing trick}\label{ssec:smoothing}

For a smoothing variance
\(\eta>0\), write
\begin{equation*}
 \pi_\eta\deq \pi*\cN(0,\eta\Id)\,,
 \qquad
 \Pi_\eta\deq \pi_\eta\otimes\cN(0,\Id)\,.
\end{equation*}
The \textbf{key algorithmic idea} of this paper is to pick a smoothing parameter
\(\eta>0\) and to sample from the smoothed distribution \(\pi_\eta\)
instead of the original target \(\pi\).  If
\(\law(Y)\approx\pi_\eta\), then we can recover a sample from \(\pi\) via
the RGO\@: for \(X\sim R_{\eta,Y}\), we have \(\law(X)\approx\pi\).

The smoothed
potential \(V_\eta\), defined by
\(e^{-V_\eta}\deq e^{-V}*\cN(0,\eta\Id)\), is given by
\begin{equation}
 \grad V_\eta(y)
 =\E\bigl[\grad V(X)\bigm\vert X+\sqrt\eta\,G=y\bigr]\,,
 \label{eq:overview-smoothed-score}
\end{equation}
where $G \sim \cN(0, \Id)$ and $X \sim \pi$ are independent.
The smoothed potential admits quantitative higher-order derivative bounds, allowing us to apply higher-order integrators.
The catch is that~\eqref{eq:overview-smoothed-score} is given as a conditional expectation, and is therefore not explicitly computable.
Write \(g_\eta\deq\grad V_\eta\).

Therefore, we replace evaluations of \(\nabla V_\eta\) with
\begin{equation}
 \widehat g_\eta(y;G)
 \deq \grad V\bigl(\prox_{\eta V}(y)+\sqrt\eta\,G\bigr)\,,
 \qquad G\sim\cN(0,\Id)\,.
 \label{eq:w2-exact-center-probe}
\end{equation}
The intuition is that
\(\nabla V_\eta(y)=\int\nabla V\,\dd R_{\eta,y}\), and
\(\law(\prox_{\eta V}(y)+\sqrt\eta\,G)\approx R_{\eta,y}\), so that
\(\widehat g_\eta(y;G)\) is almost an unbiased estimator of
\(\nabla V_\eta(y)\).

\subsection{The higher-order integrator and the \texorpdfstring{\(W_2\)}{W2} guarantee}

\paragraph{Picard iteration.}
The numerical scheme we consider below is a discretization of an unadjusted generalized Hamiltonian sampler
with partial Ornstein--Uhlenbeck (OU) momentum refreshment, following the terminology
of \cite{GLBMM24}.
Between momentum refreshments, we follow the Hamiltonian dynamics
\begin{align*}
    \dot X_t = P_t\,, \qquad \dot P_t = -\nabla V_\eta(X_t)\,.
\end{align*}
The dynamics can be written in integral form as
\begin{equation}
 X_t
 =X_0 +tP_0
 -\int_0^t(t-s)\,g_\eta(X_s)\,\dd s\,, \qquad P_t = P_0 - \int_0^t g_\eta(X_s)\,\dd s\,.
 \label{eq:overview-hamiltonian-volterra}
\end{equation}

Since the integral cannot generally be computed explicitly, we simulate it with Picard iteration. Let ${(X_t^{[0]}, P_t^{[0]})}_{t\ge 0}$ be a computable initial guess for the trajectory.
The ideal Picard iteration maps this to an updated trajectory ${(X_t^{[1]}, P_t^{[1]})}_{t\ge 0}$ by applying~\eqref{eq:overview-hamiltonian-volterra} but with the non-linear term $g_\eta$ evaluated at the preceding trajectory.
That is,
\begin{align}\label{eq:ideal_picard}
    X_t^{[1]}
 \deq X_0 +tP_0
 -\int_0^t(t-s)\,g_\eta(X_s^{[0]})\,\dd s\,, \qquad P_t^{[1]} \deq P_0 - \int_0^t g_\eta(X_s^{[0]})\,\dd s\,.
\end{align}
A particularly simple choice for the initial trajectory is
\(X_t^{[0]}\deq X_0+tP_0\).

This describes one Picard iteration, and in principle one can iterate further.
For the main results of this paper, it suffices to iterate up to depth $2$.
Higher depth iteration is treated in~\cref{app:depth-K-picard}.

The ideal Picard iteration requires exact integration, which is not directly implementable.
To develop an algorithm, we apply numerical integration, or quadrature.

\paragraph{Chebyshev--Lobatto quadrature.}
Fix an integer \(J\ge2\).  The Chebyshev--Lobatto nodes on the interval \([0,h]\) are defined to be
\begin{equation*}
    t_j\deq \frac h2\,\bigl\{1-\cos\bigl({\textstyle \frac{j-1}{J-1}}\,\uppi\bigr)\bigr\}\,,
 \qquad j=1,\ldots,J\,.
\end{equation*}
Associated with the nodes are certain interpolating polynomials $\ell_j$, $j=1,\dotsc,J$,
characterized by
\(\ell_j(t_i)=\mathbf 1_{\{i=j\}}\).  Then, for every vector-valued function
\(f\) defined at the nodes, let
\begin{equation}
  (\mathcal I_{J,h}f)(t)\deq \sum_{j=1}^{J}f(t_j)\,\ell_j(t)
  \label{eq:interpolation-operator}
\end{equation}
be its Lagrange interpolant; this is an approximation to $f(t)$ itself.
Background material on Chebyshev{--}Lobatto quadrature is collected in \cref{app:chebyshev-lobatto}.

To obtain a numerical scheme, we replace $g_\eta$ in~\eqref{eq:ideal_picard} with its interpolant.
This produces the iterations
\begin{align*}
    X_t^{[1]}
    &= X_0 + tP_0 - \int_0^t (t-s)\,\bigl(\mathcal I_{J,h}[g_\eta \circ X^{[0]}]\bigr)(s)\,\dd s\,,
    &P_t^{[1]}
    &= P_0 -\int_0^t \mathcal I_{J,h}[g_\eta \circ X^{[0]}](s)\,\dd s\,, \\[0.25em]
    X_t^{[2]}
    &= X_0 + tP_0 - \int_0^t (t-s)\,\bigl(\mathcal I_{J,h}[g_\eta \circ X^{[1]}]\bigr)(s)\,\dd s\,,
    &P_t^{[2]}
    &= P_0 -\int_0^t \mathcal I_{J,h}[g_\eta \circ X^{[1]}](s)\,\dd s\,.
\end{align*}
Note that we only ever need to evaluate these paths at the nodes $t_j$.
We can write the iterations in terms of the precomputed quantities
\begin{align}\label{eq:rhmc-chebyshev-data}
 \omega_j&\deq\int_0^h\ell_j(t)\,\dd t\,,
 &\omega_{i,j}&\deq\int_0^{t_i}(t_i-t)\,\ell_j(t)\,\dd t\,,
 &1\le i,j\le J\,.
\end{align}
Finally, we replace \(g_\eta\) by the stochastic-gradient approximation
\(\widehat g_\eta\) from \eqref{eq:w2-exact-center-probe}.  Thus, throughout
this subsection, the algorithm has exact access to \(\prox_{\eta V}\).
We replace these proximal queries by gradient queries in
\cref{app:gradient-only-prox}.

This leads to the following algorithm.

\begin{algorithm}[Smoothed Picard HMC]
\label{alg:rhmc-prefix}
Initialize
\(X_{\rm init}\deq x_{\rm ref}\) and
\(P_{\rm init}\sim\cN(0,\Id)\).  Given
\((X_{\rm init},P_{\rm init})\), one
phase draws independent standard Gaussians
\(\zeta_0,\zeta_1,(G_j^{[0]},G_j^{[1]})_{j=1}^{J}\), and performs the following steps.
\begin{enumerate}
    \item Partially refresh the momentum and create the starting trajectory.

        Set $P_0 \deq e^{-h/2}\,P_{\rm init} + \sqrt{1-e^{-h}}\,\zeta_0$ and $X_{t_j}^{[0]} \deq X_{\rm init} + t_j P_0$.
    \item Perform the first Picard update:
\begin{align*}
            X_{t_i}^{[1]}&\deq X_{\rm init} +t_iP_0
   -\sum_{j=1}^{J}\omega_{i,j}\,
      \widehat g_\eta
      (X_{t_j}^{[0]};G_j^{[0]})\,.
        \end{align*}
    \item Perform the second Picard update:
\begin{align*}
            X_h^{[2]} &\deq X_{\rm init} +hP_0-\sum_{j=1}^{J}\omega_{J,j}\,
      \widehat g_\eta
      (X_{t_j}^{[1]};G_j^{[1]})\,, \qquad
      P_h^{[2]} \deq P_0 -\sum_{j=1}^{J}\omega_j\,
      \widehat g_\eta
      (X_{t_j}^{[1]};G_j^{[1]})\,.
        \end{align*}
    \item Partially refresh the momentum and initialize the next phase.

        Set $X_{\rm init} \deq X_h^{[2]}$ and $P_{\rm init} \deq e^{-h/2}\,P_h^{[2]} + \sqrt{1-e^{-h}}\,\zeta_1$.
\end{enumerate}
After the prescribed number of phases, return \( Y\deq X_{\rm init}\).
\end{algorithm}

Finally, recall from~\cref{ssec:smoothing} that once we have an approximate sample from the smoothed distribution $\pi_\eta$, we can recover an approximate sample from $\pi$ by sampling from the RGO\@.
We implement this last step via FORS, leading to the full algorithm.

\begin{algorithm}[Smoothed Picard HMC with a terminal FORS step]
\label{alg:w2-complete}
Given the final position \(Y\) of~\cref{alg:rhmc-prefix}, independently run the terminal
FORS routine of \cref{thm:fors-implementation} with target \(R_{\eta,Y}\).
Return
its output \(X_{\rm out}\).
\end{algorithm}

The exact guarantee is \cref{thm:w2-main}, and it is proven in~\cref{sec:w2-proof}.
We next describe some of the ideas of the analysis at a high level in order to explain the source of the $d^{1/6}$ scaling.
For the sake of exposition, we ignore the dependence on $\kappa$ and $\varepsilon$.

\paragraph{Analysis.}
The analysis is based on the local error framework~\cite{MilTre21StochNum}, which produces multi-step $W_2$ bounds via computations of ``local'' or one-step errors.
The key is to decompose the one-step error into ``deterministic'' and ``stochastic'' components.
Over the course of $\asymp h^{-1}$ iterations, the former is expected to accrue with a prefactor $h^{-1}$, whereas the latter only accrues with a prefactor $h^{-1/2}$ due to cancellations.
In our analysis, the ``deterministic'' component consists of two terms: the \emph{bias} of our stochastic gradient estimator, since $\E\widehat g_\eta(\cdot; G) \ne g_\eta$, and the \emph{discretization error} of the Picard integrator.
The ``stochastic'' component consists of the fluctuations, or \emph{variance}, produced by the stochastic gradient estimator.
We handle these terms individually.
\begin{itemize}
    \item \textbf{Discretization error.}
        Here, the error further consists of two parts: the quadrature error, and the error of the Picard iteration itself.
        They are controlled in~\cref{prop:rhmc-proof-deterministic-defect-l2}.
        \begin{itemize}
            \item \emph{Quadrature error.} It is classical that the error of approximating an integral via numerical quadrature depends strongly on the smoothness of the integrand.
                In our setting, this amounts to the smoothness of $V_\eta$.
                By studying Gaussian cumulants, we present a refined estimate showing that the size of $\norm{D^k V_\eta}_{\HS}$ is of order at most $\sqrt d\,\eta^{-(k-2)/2}$, for $k\ge 2$ (\cref{lem:rhmc-proof-potential-gradient-derivatives}).
                From this bound, we are able to show that the quadrature error is bounded by a term involving $(h/\sqrt \eta)^J$, where $J$ is the number of Chebyshev--Lobatto nodes.
                By taking polylogarithmic $J$, this error is made negligible, provided $h < \sqrt \eta$.
                This shows that the natural scale is to take $h \asymp \sqrt \eta$, and we henceforth impose this choice in order to facilitate the interpretation of the subsequent bounds.
            \item \emph{Picard error.}
                The error of Picard iteration decreases exponentially with the depth.
                We show that already for the depth-$2$ scheme, this error is of order $\sqrt d\,h^5$, which is negligible compared to the other error terms.
        \end{itemize}
    \item \textbf{Bias.} Deterministically, we have $\norm{\E\widehat g_\eta(\cdot; G) - g_\eta} \le \eta^{3/2}\sqrt d$ (\cref{lem:rhmc-proof-stochastic-gradient}).
        Propagating this through the dynamics, it is straightforward to show that the bias error is of order $\sqrt d\,(h\eta^{3/2} + \eta^{1/2} h^3) \asymp \sqrt d\,h^4$ (\cref{prop:rhmc-proof-conditional-mean-error}).
    \item \textbf{Variance.} As shown in~\cite{PW26}, sharp analyses of stochastic gradient samplers should leverage the fact that \emph{centered} perturbations of a standard Gaussian are closer in $W_2$ to the Gaussian than a na\"{\i}ve coupling argument would predict; in fact, in this case, the $W_2$ distance scales with the square of the size of the perturbation.
        We refer to this type of result as a second-order $W_2$ estimate.
        We extend the convolution bound of~\cite{PW26} to a more general setting in~\cref{lem:rhmc-proof-random-map-l2}, based on a novel heat flow interpolation argument.
        We use this to prove a bound on the variance error term of order $\sqrt d\,(\eta h^{3/2} + \eta^{1/2} h^3) \lesssim \sqrt d\,h^{7/2}$ (\cref{prop:rhmc-proof-random-error-l2}).
\end{itemize}

After accumulating these errors through the local error framework, we find that the multi-step error is of order $\sqrt d\,h^3$, leading to the scaling $h \asymp d^{-1/6}$.

Finally, since smoothed Picard HMC step targets the smoothed distribution $\pi_\eta$, we need a terminal step in which we sample from the RGO $R_\eta$.
As shown in~\cref{sec:fors-terminal}, when implemented with the exact diffusion sampler of~\cite{CCRZ26}, the cost of this step is $\widetilde O(1+(\eta^2 d)^{1/3})$.
At our scale $\eta \asymp d^{-1/3}$, this terminal step only costs $\widetilde O(d^{1/9})$ queries, so it is not dominant.

\subsection{Higher moment control}

For our warm start construction, we need a stronger $W_q$ guarantee, $q\ge 2$, for smoothed Picard HMC.
We establish the following result.
\begin{informalthm}[see \cref{thm:rhmc-module} for precise statement]
Suppose that $\kappa^{-1} I\preceq\Hess V\preceq I$ and we are given a reference point
$x_{\rm ref}$ satisfying $\norm{\grad V(x_{\rm ref})}\le\sqrt{d/\kappa}$.
For any $q\ge2$, $0<\eps\le1$, and fixed $c_0>0$, gradient-only smoothed
Picard HMC can choose a smoothing variance $0<\eta\le c_0$ and return a law
$\widehat\pi_\eta$ satisfying $W_q(\widehat\pi_\eta,\pi_\eta)\le\eps$, using
\begin{equation*}
 O\Bigl(\Bigl\{\kappa^2+
     \frac{\kappa^{4/3}\,(d+q)^{1/6}}{\eps^{1/3}}\Bigr\}\,
 \Bigl\{q+\log\frac{e\kappa d}{\eps}\Bigr\}^{6}\Bigr)
\end{equation*}
gradient queries in expectation, with an implicit constant depending only on
$c_0$.
\end{informalthm}

This is taken up in \cref{sec:higher-moment-analysis} and follows the same outline as the analysis in \cref{sec:w2-proof}, except that we need several $L^q$ extensions of the ingredients from \cref{sec:w2-proof} which can be somewhat subtle.
Note that this theorem only guarantees closeness to the smoothed distribution $\pi_\eta$, but this is all that is required for our warm start generator described in \cref{ssec:recursive_overview}.

We highlight in particular that \cref{lem:rhmc-proof-centered-kernel-perturbation} establishes an $L^q$ extension of the classical local error framework, which could be useful for future works.

Note that when we upgrade the $W_2$ guarantee in \cref{thm:w2-main} to a $W_q$ guarantee, we pay a price which is polynomial in $q$.
However, for our warm start application, we only need to take $q$ to be logarithmic in the problem parameters, so this turns out to be harmless.

\subsection{Recursive warm start generation}\label{ssec:recursive_overview}

Next, we develop a powerful recursive framework for warm start generation.
Here, we suppose we have access to a sampling subroutine which we call a \emph{smoothed sampler}.
Given a strongly convex and smooth potential \(V\), normalized to satisfy $\kappa^{-1} I \preceq \nabla^2 V \preceq I$, as well as an admissible reference point $x_{\rm ref}$ with $\norm{\nabla V(x_{\rm ref})} \le \sqrt{d/\kappa}$, the smoothed sampler returns a smoothing level $0 \le \eta \ll 1$ and a sample which is $\varepsilon$-close in law to $\pi_V * \cN(0, \eta I)$, where $\pi_V \propto \exp(-V)$.

We remark that if $V$ satisfies $\alpha I \preceq \nabla^2 V \preceq \beta I$, then we can apply the smoothed sampler to $V(\cdot/\sqrt\beta)$, and then rescale the output by $\beta^{-1/2}$.
After this rescaling, the reported smoothing level $\eta$ corresponds to the law of the rescaled sample being close to $\pi_V * \cN(0, \beta^{-1} \eta I)$.

We consider three types of guarantees for the smoothed sampler: the output law is close to the smoothed distribution in $W_2$, in $W_p$ (for $p\ge 2$), or in the sub-Gaussian Orlicz{--}Wasserstein metric $W_{\psi_2}$.
In these three cases, by taking advantage of regularization along the heat flow, we develop a recursive sampler that upgrades the guarantee to $\KL$, to mixed $\TV$ and $\Ren_q$ (as in the guarantee of \cref{thm:adaptive-gibbs-warm}), and to $\Ren_q$, respectively.

For $V\in C^2(\R^d)$, $A > 0$, and $u\in\R^d$, consider the RGO potential and distribution
\begin{align*}
    V_{A,u}
    &\deq V + \frac{\norm{\cdot - u}^2}{2A}\,, \qquad R^V_{A,u} \propto \exp(-V_{A,u})\,.
\end{align*}
The following recursive sampler aims to draw an approximate sample from the RGO distribution $R^V_{A,u}$.
Note that when $A = \infty$, this just corresponds to $\pi \propto \exp(-V)$. The sampling task is easier for smaller $A$ as $V_{A, u}$ then has a smaller condition number. The trick of the following algorithm is to recursively reduce the RGO sampling task to one with a smaller $A$, until the terminal level is reached.

\begin{algorithm}[Recursive RGO sampler]
\label{alg:generic-warm-overview}
Fix a terminal level $\underline A > 0$.
\begin{itemize}
    \item If $A \le \underline A$, then simply apply the localized FORS routine in~\cref{thm:fors-implementation} directly to $R_{A,u}^V$.
    \item Otherwise, proceed via the following steps.
        \begin{itemize}
            \item Apply the smoothed sampler to $V_{A,u}$, and let $\eta > 0$ be the reported smoothing level, and let $Z$ be the rescaled output.
                Thus, $\law(Z) \approx R_{A,u}^V * \cN(0, \beta_A^{-1} \eta I)$, where $\beta_A$ is the smoothness of $V_{A,u}$.
            \item Pick $\tau > 0$, draw an independent
\(G\sim\cN(0,\Id)\), and form
\begin{equation*}
  Y\deq Z+\sqrt{\frac\tau{\beta_A}}\,G\,,
  \qquad
  a\deq\frac{\eta+\tau}{\beta_A}\,.
\end{equation*}
Thus, $\law(Y) \approx R^V_{A,u} * \cN(0, aI)$.
\item We wish to obtain an approximate sample from $R^V_{A,u}$. It suffices to sample from $R^{V_{A,u}}_{a,Y}$.
    Indeed, by the definition of the RGO, for any potential $V$, we have the identity $X \sim \pi_V$ if $Y \sim \pi_V * \cN(0,aI)$ and $X \sim R^V_{a,Y}$.

    In~\cref{lem:rgo-calculus}, we justify the following closure property of the RGO\@: $R_{a,Y}^{V_{A,u}} = R^V_{A^+, u^+}$, where the updated parameters are
\begin{equation*}
  \frac1{A^+}\deq\frac1A+\frac1a\,,
  \qquad
  u^+\deq A^+\,\Bigl(\frac uA+\frac Ya\Bigr)\,.
\end{equation*}
Thus, we call the recursive RGO sampler on \(R^V_{A^+,u^+}\), and
return its output.
        \end{itemize}
\end{itemize}
\end{algorithm}

At this point, it may be unclear why we add an additional Gaussian to the output of the smoothed sampler.
In fact, this is the key step which enables the use of regularization along the heat flow---specifically, in the form of reverse transport inequalities---to upgrade the smoothed sampler guarantee from Wasserstein to stronger distances.
The guarantee for \cref{alg:generic-warm-overview} is given in \cref{thm:recursive-rgo-sampler}.

We can apply \cref{alg:generic-warm-overview} to sample from $\pi$ by taking $A = \infty$.
However, if the iteration complexity of the smoothed sampler has poor dependence on the condition number, then it is better to wrap the entire recursive scheme within the proximal sampler~\cite{LST21, CCSW22}.

\begin{algorithm}[Proximal sampler wrapper]
\label{alg:outer-warm-overview}
    Given an initial sample \(X_0\), an outer variance \(a_{\rm out}>0\),
    and a number \(N\) of outer steps, iterate as follows.  At step
    \(n\in\{0,\ldots,N-1\}\), draw
    \(Y_n=X_n+\cN(0,a_{\rm out} I)\), and set \(X_{n+1}\) to be the
    output of \cref{alg:generic-warm-overview} with target
    \(R^V_{a_{\rm out},Y_n}\).  Return the final iterate \(X_N\).
\end{algorithm}

This wrapper ensures that we only ever call the smoothed sampler on well-conditioned targets, leading to an overall improved condition number dependence.
This aligns with the original use of the proximal sampler in~\cite{LST21}.
The full guarantee for the recursive warm start generator is given in
\cref{thm:recursive-warm-generator}.
We emphasize that the guarantee is fully generic and applies to any smoothed sampler implementing our interface.

Instantiating the recursive warm start generator with smoothed Picard HMC as a subroutine then establishes our warm start guarantee in \cref{thm:adaptive-gibbs-warm}.

\section{Proof of the \texorpdfstring{$W_2$}{W\_2} guarantee}
\label{sec:w2-proof}

In this section, our goal is to prove the $W_2$ guarantee (\cref{thm:w2-main}). However, since we will later need to obtain stronger moment estimates for~\cref{thm:adaptive-gibbs-warm}, we will first develop the part of the analysis that is common to both results.

\subsection{Setup and smoothing estimates}
\label{sec:rhmc-proof}

For a positive-definite matrix \(M\), write
\[
 W_{q,M}(P,Q)\deq
 \inf_{X\sim P,\,Y\sim Q}\bigl(\E\norm{X-Y}_M^q\bigr)^{1/q}\,,
 \qquad
 \norm{x}_M\deq(x^\T Mx)^{1/2}\,.
\]
Recall that
\[
 \pi_\eta\deq \pi * \cN(0,\eta\Id)\,, \qquad
 \Pi_\eta(\dd x\,\dd p)\deq \pi_\eta(\dd x)\,\cN(0,\Id)(\dd p)\,.
\]
Let \(X\mid Y=y\) be drawn from $R_{\eta,y}$.
Recalling the definition~\eqref{eq:prox-def} of the proximal operator, throughout this section we write
\[
 x_y^+\deq\prox_{\eta V}(y)\,.
\]
It is easy to check the identities
\begin{equation}
 g_\eta(y)= \nabla V_\eta(y) =\eta^{-1}\,(y-\E[X\mid Y=y])\,,
 \qquad
 \Hess V_\eta(y)=\eta^{-1}\Id-\eta^{-2}\operatorname{Cov}(X\mid Y=y)\,.
 \label{eq:rhmc-proof-potential-gradient-covariance}
\end{equation}

\begin{lemma}[Hessian bounds for the smoothed potential]
\label{lem:rhmc-proof-smoothed-curvature}
For \(0<\eta\le1\),
\begin{equation*}
 \frac1{\kappa+\eta}\,\Id
 \preceq \Hess V_\eta
 \preceq \frac1{1+\eta}\,\Id
 \preceq\Id\,.
\end{equation*}
In particular \(g_\eta\) is $1$-Lipschitz.
\end{lemma}

\begin{proof}
The RGO potential has Hessian between
\((\kappa^{-1}+\eta^{-1})\,\Id\) and \((1+\eta^{-1})\,\Id\).  The
 Brascamp--Lieb covariance inequality \cite{BL76} gives
\[
 \operatorname{Cov}(X\mid Y=y)
 \preceq(\kappa^{-1}+\eta^{-1})^{-1}\Id\,.
\]
For the reverse bound, the
Cram\'er--Rao inequality (see, e.g.,~\cite[Theorem 3.5.8]{Chewi26Book}) gives
\[
 \operatorname{Cov}(X\mid Y=y)
 \succeq (1+\eta^{-1})^{-1}\Id\,.
\]
Substitute both bounds in
\eqref{eq:rhmc-proof-potential-gradient-covariance}.
\end{proof}

The following is the key estimate that decomposes the stochastic gradient error into a bias term and a stochastic error term.

\begin{lemma}[Stochastic gradient estimates]
\label{lem:rhmc-proof-stochastic-gradient}
Write
\begin{equation*}
 \widehat g_\eta(y;G)\deq g_\eta(y)+b_\eta(y)+\xi_\eta(y;G)\,,
 \qquad \E\xi_\eta(y;G)=0\,.
\end{equation*}
Then, for all \( u,y,y',G,G'\),
\begin{align}
 \norm{b_\eta(y)}&\le\eta^{3/2}\sqrt d\,,\label{eq:rhmc-proof-stochastic-gradient-bias}\\
 \E e^{\ip{u}{\xi_\eta(y;G)}}&\le e^{\eta\,\norm u^2/2}\,,\label{eq:rhmc-proof-stochastic-gradient-subg}\\
 \norm{\widehat g_\eta(y;G)-\widehat g_\eta(y';G)}
 &\le\norm{y-y'}\,,\label{eq:rhmc-proof-stochastic-gradient-y-lip}\\
 \norm{\widehat g_\eta(y;G)-\widehat g_\eta(y;G')}
 &\le\sqrt\eta\,\norm{G-G'}\,.\label{eq:rhmc-proof-stochastic-gradient-g-lip}
\end{align}
\end{lemma}

\begin{proof}
Proximal optimality gives
\(\grad V(x_y^+)=(y-x_y^+)/\eta\).  Define
\[
    \rho_y(u)\deq V(x_y^++\sqrt\eta u)-V(x_y^+)
             -\sqrt\eta\,\ip{\grad V(x_y^+)}u\,.
\]
Then \(0\preceq\Hess\rho_y\preceq\eta\Id\), and under $R_{\eta,y}$,
\((X-x_y^+)/\sqrt\eta\) has density
proportional to \( u\mapsto e^{-\norm u^2/2-\rho_y(u)}\).  Call this law \(\mathsf r_y\).
We have
\begin{equation}
    W_2(\mathsf r_y,\cN(0,\Id))\le\sqrt{\E_{\mathsf r_y}[\norm{\nabla \rho_y(U)}^2]}
    \le \eta\sqrt{\E_{\mathsf r_y}\norm U^2}\le\eta\sqrt d\,.
 \label{eq:rhmc-proof-r-gaussian}
\end{equation}
Indeed, the first inequality combines Talagrand's inequality for the standard
Gaussian with the Gaussian log-Sobolev inequality; the second uses that
\(\nabla\rho_y(0)=0\) and \(\nabla\rho_y\) is \(\eta\)-Lipschitz; and the last
follows by integration by parts since $\mathsf r_y$ is $1$-strongly
log-concave with its mode at $0$.

Moreover,
\[
 g_\eta(y)=\grad V(x_y^+)+\eta^{-1/2}\,\E_{\mathsf r_y}\grad\rho_y(U)\,,
 \qquad
 \widehat g_\eta(y;G)
 =\grad V(x_y^+)+\eta^{-1/2}\,\grad\rho_y(G)\,.
\]
The gradient of \(\rho_y\) is \(\eta\)-Lipschitz, so
\eqref{eq:rhmc-proof-r-gaussian} gives \eqref{eq:rhmc-proof-stochastic-gradient-bias}.  Since $\widehat g_\eta(y;\cdot)$ is \(\sqrt\eta\)-Lipschitz, the standard Gaussian concentration inequality
gives \eqref{eq:rhmc-proof-stochastic-gradient-subg}, and the same observation gives
\eqref{eq:rhmc-proof-stochastic-gradient-g-lip}.  Finally, \(y\mapsto x_y^+\) is non-expansive and
\(\grad V\) is $1$-Lipschitz, proving \eqref{eq:rhmc-proof-stochastic-gradient-y-lip}.
\end{proof}

We will also need high-order regularity of the smoothed potential.  The key
point is a near-Gaussian cancellation, which is stronger than the generic moment bound for a
strongly log-concave law. The Gaussian cumulant estimates used in the proof
are collected in \cref{app:gaussian-analysis}.

\begin{lemma}[Derivatives of the smoothed potential]
\label{lem:rhmc-proof-potential-gradient-derivatives}
For every integer \(k\ge2\),
\begin{align}
    \sup_{y\in\R^d}\,\norm{D^kV_\eta(y)}_{\HS}
 &\le (k-1)!\,\sqrt d\,\eta^{-(k-2)/2}\,,
 \label{eq:rhmc-proof-derivative-full}\\
 \sup_{y\in\R^d,\,\norm v=1}\,\norm{D^kV_\eta(y)[v,\cdot^{k-1}]}_{\HS}
 &\le (k-1)!\,\eta^{-(k-2)/2}\,.
 \label{eq:rhmc-proof-derivative-slot}
\end{align}
Consequently,
\begin{equation}
 \norm{\Hess V_\eta(y)-\Hess V_\eta(y')}_{\HS}
 \le 2\eta^{-1/2}\,\norm{y-y'}\,.
 \label{eq:rhmc-proof-hessian-F-lipschitz}
\end{equation}
\end{lemma}

\begin{proof}
Let \(\mathsf r_y\) be the normalized RGO from the proof of
\cref{lem:rhmc-proof-stochastic-gradient}.  Its potential has Hessian between
\(\Id\) and \((1+\eta)\Id\).  Let \(T_y\) be the optimal transport
from the standard Gaussian to \(\mathsf r_y\).  Caffarelli's contraction theorem
\cite{C00}, applied in both directions, gives
\begin{equation}
 (1+\eta)^{-1/2}\Id\preceq DT_y\preceq\Id\,,
 \qquad \norm{DT_y-\Id}_{\op}\le\eta/2\,,
 \label{eq:rhmc-proof-transport-close}
\end{equation}
almost everywhere.

By \eqref{eq:smoothed-potential-cumulant-identity}, for \(k\ge3\),
\[
 D^kV_\eta(y)=-\eta^{-k}\operatorname{Cum}_k(X\mid Y=y)\,.
\]
Under the RGO,
\(X=x_y^++\sqrt\eta\,T_y(G)\), $G \sim \cN(0,I)$.  Translation invariance and homogeneity of
cumulants, followed by \eqref{eq:rhmc-proof-transport-close} and
\cref{lem:rhmc-proof-near-identity-cumulant} with \(\varepsilon_T\le\eta/2\), prove
\eqref{eq:rhmc-proof-derivative-full}--\eqref{eq:rhmc-proof-derivative-slot}.
The case \(k=2\) is \cref{lem:rhmc-proof-smoothed-curvature}; integrating the
\(k=3\) estimate~\eqref{eq:rhmc-proof-derivative-slot} along a segment gives
\eqref{eq:rhmc-proof-hessian-F-lipschitz}.
\end{proof}

\subsection{Local error analysis}

The $W_2$ proof will be based on the local error framework.

\begin{lemma}[Local error framework]
\label{lem:rhmc-proof-local-error}
Let \(P,\widehat P\) be Markov kernels on \(\R^m\).  Suppose that, for every
\(x,y\in\R^m\), there is a coupling
\(X\sim\delta_xP\), \(Y\sim\delta_yP\) such that
\begin{equation}
 \norm{X-Y}_{L^2}
 \le\rho\,\norm{x-y}\,,
 \qquad
 \norm{(X-x)-(Y-y)}_{L^2}
 \le\chi\,\norm{x-y}\,.
 \label{eq:rhmc-proof-local-error-coupling}
\end{equation}
Suppose also that, writing \(\widehat X\sim\delta_x\widehat P\),
\begin{equation*}
 W_2(\delta_x\widehat P,\delta_xP)\le\mathcal E_{\rm strong}\,,
 \qquad
 \norm{\E\widehat X-\E X}\le\mathcal E_{\rm weak}\,.
\end{equation*}
Assume \(0<\rho<1\).  Then, for all probability laws \(\mu,\nu\) on
\(\R^m\),
\begin{equation*}
 W_2^2(\mu\widehat P,\nu P)
 \le\rho W_2^2(\mu,\nu)
 +\mathcal E_{\rm strong}^2
 +\frac{(\mathcal E_{\rm weak}
          +\chi\mathcal E_{\rm strong})^2}{\rho\,(1-\rho)}\,.
\end{equation*}
\end{lemma}
\begin{proof}
This is the standard local error theorem
\cite[Lemma~5.1.2]{Chewi26Book}.
\end{proof}

We introduce three kernels and associated mappings.
\begin{itemize}
    \item Let \(\widehat K\) denote the one-phase
kernel defined by \cref{alg:rhmc-prefix} with
stochastic gradients \(\widehat g_\eta\).  We can express the output of the phase as
\(\widehat\Phi_z(\zeta,G)\) for Gaussian variables $\zeta$, $G$.
\item Define
\[
    \bar\Phi_z(\zeta)\deq\E[\widehat\Phi_z(\zeta,G)\mid \zeta]\,,
 \qquad
 \bar K(z,\cdot)\deq(\bar\Phi_z)_\#\cN(0,\Id_{2d})\,.
\]
\item Finally, let \(K\) be the exact gradient numerical kernel obtained by
replacing every stochastic gradient $\widehat g_\eta$ by \(g_\eta\).  Write \(\Phi_z\) for the corresponding map, so that
\[
 K(z,\cdot)\deq(\Phi_z)_\#\cN(0,\Id_{2d})\,.
\]
\end{itemize}

It is standard that contraction for HMC is only witnessed after a change of coordinates. Set
\begin{equation*}
 M_\kappa\deq
 \begin{bmatrix}\frac1{2\kappa}+\frac12&\frac12\\[1mm]\frac12&1\end{bmatrix}
 \otimes\Id_d\,.
\end{equation*}
Since $\kappa \geq 1$, this defines a metric which is uniformly equivalent to the Euclidean norm; such twisted metrics are standard in the analysis of kinetic dynamics; see, e.g.,~\cite{CCBJ18, GLBMM24}.
We apply the lemma after the linear change of coordinates
\(z\mapsto M_\kappa^{1/2}z\), with \(P\deq K\) and
\(\widehat P\deq\widehat K\).
The structure of the proof is given by the following key theorem.

\begin{theorem}[Local error analysis in \(W_2\)]
\label{prop:rhmc-proof-global-recurrence-l2}
Let \(J\ge2\), and assume that \(h\le c/\kappa\).
Then, for every phase-space law \(\widehat\mu\),
\begin{equation*}
 W_{2,M_\kappa}^2(\widehat\mu \widehat K,\Pi_\eta)
 \le(1-ch/\kappa)\,W_{2,M_\kappa}^2(\widehat\mu,\Pi_\eta)
 +C\Evar^2
 +\frac{C\kappa}{h}\,(\Ebias^2+\Edisc^2)\,,
\end{equation*}
where we define
\begin{align*}
 \Ebias
 &\deq\sup_{z,\zeta\in\R^{2d}}\,
   \norm{\bar\Phi_z(\zeta)-\Phi_z(\zeta)}\,, \quad
 \Evar
 \deq\sup_{z\in\R^{2d}}
   W_2(\delta_z\widehat K,\delta_z\bar K)\,,\quad
 \Edisc
 \deq W_2(\Pi_\eta K,\Pi_\eta)\,.
\end{align*}
Consequently,
\begin{align}
 W_{2,M_\kappa}^2(\widehat\mu \widehat K^N,\Pi_\eta)
 &\le e^{-cNh/\kappa}\,W_{2,M_\kappa}^2(\widehat\mu,\Pi_\eta)
 +\frac{C\kappa}{h}\,\Evar^2
 +\frac{C\kappa^2}{h^2}\,(\Ebias^2+\Edisc^2)\,.
 \label{eq:rhmc-proof-global-recurrence-l2}
\end{align}
\end{theorem}
\begin{proof}
Fix \(z,z'\in\R^{2d}\) and set
\(Z\deq\Phi_z(\zeta)\), \(Z'\deq\Phi_{z'}(\zeta)\).
Since \(J\ge2\) and
\(h\le c/\kappa\), the assumption on \(h\) implies
the hypotheses of
\cref{prop:rhmc-proof-contraction,lem:rhmc-proof-increment-coupling} below.
Those results verify \eqref{eq:rhmc-proof-local-error-coupling} with
\((Z,Z')\) in place of \((X,Y)\).  We can bound
\begin{align*}
 \mathcal E_{\rm weak}
 &\deq
 \sup_{z\in\R^{2d}}\,
 \norm{\E \widehat\Phi_z(\zeta,G)-\E\Phi_z(\zeta)}_{M_\kappa}
 \le
 \sup_{z\in\R^{2d}}
 \E\norm{\bar\Phi_z(\zeta)-\Phi_z(\zeta)}_{M_\kappa}
 \le C\Ebias\,,
\end{align*}
and
\begin{align*}
 \mathcal E_{\rm strong}
 &\deq\sup_{z\in\R^{2d}}
   W_{2,M_\kappa}(\delta_z\widehat K,\delta_zK)
 \le\sup_{z\in\R^{2d}}\,
 \{
 W_{2,M_\kappa}(\delta_z\widehat K,\delta_z\bar K)
 +W_{2,M_\kappa}(\delta_z\bar K,\delta_zK)
 \}\\
 &\le C\,(\Evar+\Ebias)\,,
\end{align*}
since the $M_\kappa$ norm is equivalent to the standard Euclidean norm.
Thus, the hypotheses of
\cref{lem:rhmc-proof-local-error} hold with
\[
 \rho\deq1-\frac{ch}{\kappa}\,,
 \qquad
 \chi\deq Ch\,,
 \qquad
 \mathcal E_{\rm strong}\le C\,(\Evar+\Ebias)\,,
 \qquad
 \mathcal E_{\rm weak}\le C\Ebias\,.
\]
We obtain
\[
 W_{2,M_\kappa}^2(\widehat\mu\widehat K,\Pi_\eta K)
 \le\Bigl(1-\frac{ch}{\kappa}\Bigr)\,W_{2,M_\kappa}^2(\widehat\mu,\Pi_\eta)
 +C\Evar^2+\frac{C\kappa}{h}\,\Ebias^2\,.
\]
By Young's inequality, for $\lambda > 0$,
\begin{align*}
    W_{2,M_\kappa}^2(\widehat\mu\widehat K, \Pi_\eta)
    &\le (1+\lambda)\,W_{2,M_\kappa}^2(\widehat\mu\widehat K, \Pi_\eta K) + (1+\lambda^{-1})\,W_{2,M_\kappa}^2(\Pi_\eta K, \Pi_\eta)\,.
\end{align*}
Uniform norm equivalence gives
\(W_{2,M_\kappa}(\Pi_\eta K,\Pi_\eta)\le C\Edisc\).
Choose \( \lambda \asymp h/\kappa\).  After reducing
the universal constant in the contraction factor \(ch/\kappa\), this proves the one-step
bound; summing the recurrence proves the global bound.
\end{proof}

The three errors \(\Ebias\), \(\Evar\), and \(\Edisc\) correspond,
respectively, to the bias and centered fluctuation of the stochastic gradient
phase and to the discretization error of the numerical kernel.  Thus, the
main task is to bound these errors.

\subsection{Properties of the exact-gradient kernel}
\label{sec:rhmc-proof-exact-kernel}

Write \(a\deq e^{-h/2}\),  and for a symmetric matrix \(H\),
let
\[
 A_H\deq \begin{bmatrix}0&\Id\\-H&-\Id\end{bmatrix}\,.
\]
We use the integrated weight bounds in \cref{prop:chebyshev-lobatto-bounds}.

\begin{lemma}[First-order difference expansion]
\label{lem:rhmc-proof-difference-expansion}
Couple draws \(Z_1\sim\delta_zK\) and \(Z_1'\sim\delta_{z'}K\) from
phase-space inputs
\(z\deq (x_{\rm init},p_{\rm init})\) and
\(z'\deq (x'_{\rm init},p'_{\rm init})\) using the same Gaussian variables.
Set
\[
 \Delta_0\deq z-z'\,,
 \qquad
 \Delta_1\deq Z_1-Z_1'\,.
\]
If \(h^2\le1\), then
\[
 \Delta_1=(\Id+hA_H)\Delta_0+\mathcal R_h\Delta_0\,,
 \qquad
 \norm{\mathcal R_h}_{\op}\le Ch^2\,,
\]
for a symmetric \(H\) with
\(\frac1{2\kappa}\,\Id\preceq H\preceq\Id\).
\end{lemma}
\begin{proof}
To verify this, write
\(\Delta x\deq x_{\rm init}-x'_{\rm init}\) and
\(\Delta p\deq p_{\rm init}-p'_{\rm init}\).
The difference of gradients can be written as
\[
 g_\eta(x)-g_\eta(x')=H(x,x')\,(x-x')\,,
 \quad
 H(x,x')\deq \int_0^1\Hess V_\eta(x'+t\,(x-x'))\,\dd t\,.
\]
By \cref{lem:rhmc-proof-smoothed-curvature},
\(\frac1{2\kappa}\,\Id\preceq H \preceq \Id\).

The first OU half-step changes \(\Delta p\) to
\(a\,\Delta p\).
Thus,
\[
 \Delta X_{t_j}^{[0]}=\Delta x+at_j\,\Delta p\,.
\]
Writing $H_j^{[0]} \deq H(X_{t_j}^{[0]},X_{t_j}^{\prime[0]})$,
the first Picard update gives
\[
 \Delta X_{t_i}^{[1]}
 =\Delta x+at_i\,\Delta p
 \underbrace{-\sum_{j\in[J]}\omega_{i,j}H_j^{[0]}
 \,(\Delta x+at_j\,\Delta p)}_{\eqqcolon E_i}\,.
\]
The bounds \(\norm{H_j^{[0]}}_{\op}\le1\), \(a\le1\), \(t_j\le h\), and
\(\sum_{j\in[J]}|\omega_{i,j}|\le h^2\) (\cref{prop:chebyshev-lobatto-bounds}) imply
\[
 \norm{E_i}
 \le\sum_{j\in[J]}|\omega_{i,j}|\,
       \bigl(\norm{\Delta x}+at_j\,\norm{\Delta p}\bigr)
 \le h^2
       \,\bigl(\norm{\Delta x}+h\,\norm{\Delta p}\bigr)\,.
\]
Then, set
\[
 H_j\deq H\bigl(X_{t_j}^{[1]},X_{t_j}^{\prime[1]}\bigr)\,,
 \qquad
 H\deq h^{-1}\sum_{j\in[J]}\omega_jH_j\,.
\]
Positivity and \(\sum_{j\in[J]}\omega_j=h\) imply
\(\frac1{2\kappa}\,\Id\preceq H\preceq\Id\).  The exact endpoint differences are
\begin{align*}
 \Delta X_{\rm init}^+
 &=\Delta x+ah\,\Delta p
   -\sum_{j\in[J]}\omega_{J,j}H_j\,\Delta X_{t_j}^{[1]}\,,
 \qquad \Delta P_{\rm init}^+
 = a^2\,\Delta p-a\sum_{j\in[J]}\omega_jH_j\,\Delta X_{t_j}^{[1]}\,.
\end{align*}
Subtract \(\Delta x+h\,\Delta p\) from the first line and
\(\Delta p-hH\,\Delta x-h\,\Delta p\) from the second.  Use
\(a=1-h/2+O(h^2)\), \(a^2=1-h+O(h^2)\),
\eqref{eq:rhmc-chebyshev-bounds}, and the preceding bound on \(E_i\).  Every remaining
term is bounded by
\(Ch^2\,\norm{(\Delta x,\Delta p)}\), which proves the claimed
expansion.
\end{proof}

\begin{proposition}[Contraction]
\label{prop:rhmc-proof-contraction}
If \(h\le c/\kappa\), then, for all \(z,z',\zeta\in\R^{2d}\),
\begin{equation}
 \norm{\Phi_z(\zeta)-\Phi_{z'}(\zeta)}_{M_\kappa}
 \le\Bigl(1-\frac{ch}{\kappa}\Bigr)\,\norm{z-z'}_{M_\kappa}\,.
 \label{eq:rhmc-proof-pointwise-contraction}
\end{equation}
\end{proposition}
\begin{proof}
Let \(q\) be a unit
eigenvector of \(H\) with eigenvalue
\(\lambda\in[\frac1{2\kappa},1]\).  The two-dimensional phase-space subspace $\{(u q,v q):u,v\in\R\}$
is invariant under \(A_H\).  In the coordinates \((u,v)\), the restrictions
of \(A_H\) and \(M_\kappa\) to this subspace are represented by
\[
 A_H(uq,vq)=\bigl(vq,-(\lambda u+v)q\bigr)\,,
 \qquad
 A_\lambda\deq\begin{bmatrix}0&1\\-\lambda&-1\end{bmatrix}\,,
 \qquad
 M_\kappa^{(2)}\deq
 \begin{bmatrix}\frac1{2\kappa}+\frac12&\frac12\\[1mm]\frac12&1\end{bmatrix}\,.
\]
A direct calculation gives
\begin{equation}
 A_\lambda^\T M_\kappa^{(2)}+M_\kappa^{(2)}A_\lambda
 =-\begin{bmatrix}
     \lambda&\lambda-\frac1{2\kappa}\\
     \lambda-\frac1{2\kappa}&1
    \end{bmatrix}
 \preceq-\frac1{8\kappa}\,M_\kappa^{(2)}\,.
 \label{eq:rhmc-proof-LMI}
\end{equation}
Indeed, the positive matrix after the minus sign has trace at most \(2\) and
determinant
\[
 \lambda-\Bigl(\lambda-\frac1{2\kappa}\Bigr)^2
 \ge\frac1{2\kappa}\,.
\]
Its smallest eigenvalue is therefore at least \(1/(4\kappa)\).  Since
\(M_\kappa^{(2)}\preceq2\Id\), this proves \eqref{eq:rhmc-proof-LMI}.
Decomposing along an orthonormal eigenbasis of \(H\) yields the full
phase-space inequality
\begin{align}\label{eq:rhmc-proof-LMI-2}
 A_H^\T M_\kappa+M_\kappa A_H
 \preceq-\frac1{8\kappa}\,M_\kappa\,.
 \end{align}
Apply \cref{lem:rhmc-proof-difference-expansion}.  Uniform equivalence of
the Euclidean and \(M_\kappa\) norms gives
\[
 \norm{\mathcal R_h\Delta_0}_{M_\kappa}
 \le Ch^2\,\norm{\Delta_0}_{M_\kappa}\,.
\]
Expanding the square and
applying \eqref{eq:rhmc-proof-LMI-2} gives
\[
 \norm{\Delta_1}_{M_\kappa}^2
 \le\{1-ch/\kappa+Ch^2\}\,
      \norm{\Delta_0}_{M_\kappa}^2\,.
\]
The step size restriction absorbs the last term.  Taking square roots and changing
the universal constant gives \eqref{eq:rhmc-proof-pointwise-contraction}.
\end{proof}

\begin{lemma}[Synchronous increment coupling]
\label{lem:rhmc-proof-increment-coupling}
If \(h\le c/\kappa\), then, for all \(z,z',\zeta\in\R^{2d}\),
\begin{equation}
 \norm{\{\Phi_z(\zeta)-z\}-\{\Phi_{z'}(\zeta)-z'\}}_{M_\kappa}
 \le Ch\,\norm{z-z'}_{M_\kappa}\,.
 \label{eq:rhmc-proof-increment-coupling}
\end{equation}
\end{lemma}
\begin{proof}
Subtract \(\Delta_0\) from the expansion in
\cref{lem:rhmc-proof-difference-expansion}.  Uniform equivalence of the
Euclidean and \(M_\kappa\) norms, together with
\(\norm H_{\op}\le1\), gives
\(\norm{A_H\Delta_0}_{M_\kappa}
\le C\,\norm{\Delta_0}_{M_\kappa}\).
Moreover,
\(h^2\le h\) under the step size assumption.  Hence
\(\norm{\Delta_1-\Delta_0}_{M_\kappa}
\le Ch\,\norm{\Delta_0}_{M_\kappa}\), which is
\eqref{eq:rhmc-proof-increment-coupling}.
\end{proof}

\subsection{Local error bounds}
\label{sec:rhmc-proof-local-bounds}

\subsubsection{Bias bound}
\label{sec:rhmc-proof-bias-bound}

We first control the displacement between the conditional mean and
exact gradient kernels.  Put
\[
 \bar g_\eta(y)\deq\E\widehat g_\eta(y;G)
 =g_\eta(y)+b_\eta(y)\,.
\]

\begin{proposition}[Bias bound]
\label{prop:rhmc-proof-conditional-mean-error}
If \(h\le1\), then
\begin{equation}
 \Ebias
 \le C\sqrt d\,
 \bigl(h\eta^{3/2}+ \eta^{1/2} h^3/\sqrt J\bigr)\,.
 \label{eq:rhmc-proof-conditional-mean-error}
\end{equation}
\end{proposition}
\begin{proof}
Fix \(z,\zeta\in\R^{2d}\) and write
\[
 B\deq\sup_{y\in\R^d}\,\norm{b_\eta(y)}
 \le\eta^{3/2}\sqrt d\,,
\]
by \eqref{eq:rhmc-proof-stochastic-gradient-bias}.  To distinguish the two
first Picard updates, let \(X_{t_i}^{[1]}\) denote the iterate obtained with
the exact gradient \(g_\eta\), and let \(\widehat X_{t_i}^{[1]}\) denote
the iterate obtained with \(\widehat g_\eta\).  Taking the expectation over the first
stochastic gradient block,
\begin{align*}
 \bar X_{t_i}^{[1]}-X_{t_i}^{[1]}
 &\deq \E[\widehat X_{t_i}^{[1]}\mid\zeta]-X_{t_i}^{[1]}
 =-\sum_{j\in[J]}\omega_{i,j}
     b_\eta(X_{t_j}^{[0]})\,.
\end{align*}
Consequently, \eqref{eq:rhmc-chebyshev-bounds} gives the
estimate
\begin{equation}
 \max_{j\in[J]}\,\norm{\bar X_{t_j}^{[1]}-X_{t_j}^{[1]}}
 \le h^2\,B\,.
 \label{eq:rhmc-proof-mean-node-bias}
\end{equation}
We also need to control the fluctuations:
\[
 \widehat X_{t_i}^{[1]}-\bar X_{t_i}^{[1]}
 =-\sum_{j\in[J]}\omega_{i,j}
   \xi_\eta(X_{t_j}^{[0]};G_j^{[0]})\,.
\]
The variables in the sum are independent and centered, and
\eqref{eq:rhmc-proof-stochastic-gradient-subg} implies
\(\E\norm{\xi_\eta(y;G)}^2\le\eta d\).  Therefore,
\begin{align}
 \E[\norm{\widehat X_{t_i}^{[1]}
                 -\bar X_{t_i}^{[1]}}^2\mid\zeta]
 &=\sum_{j\in[J]}\omega_{i,j}^2\,
   \E[\norm{\xi_\eta(X_{t_j}^{[0]};G_j^{[0]})}^2\mid\zeta]
 \le \frac{5\uppi^2h^4}{128\,(J-1)}\,\eta d
 \le \frac{Ch^4\eta d}{J}\,.
 \label{eq:rhmc-proof-centered-node-fluctuation}
\end{align}

It remains to propagate these errors through the endpoint update.
The endpoint stochastic gradient block is independent of the first block,
so for each \(j\),
\[
 \E[\widehat g_\eta
       (\widehat X_{t_j}^{[1]};G_j^{[1]})\mid\zeta]
 =\E[\bar g_\eta(\widehat X_{t_j}^{[1]})\mid\zeta]\,.
\]
Since \(\widehat g_\eta(\,\cdot\,;G)\) is \(1\)-Lipschitz for every \(G\),
so is \(\bar g_\eta\).  Using
\(\bar g_\eta=g_\eta+b_\eta\), followed by
\eqref{eq:rhmc-proof-mean-node-bias},
\eqref{eq:rhmc-proof-centered-node-fluctuation}, and Jensen's inequality,
we obtain
\begin{align*}
 \norm{\E[\bar g_\eta(\widehat X_{t_j}^{[1]})\mid\zeta]
             -g_\eta(X_{t_j}^{[1]})}
 &\le
 \E[\norm{\widehat X_{t_j}^{[1]}
                  -\bar X_{t_j}^{[1]}}\mid\zeta]
 +\norm{\bar X_{t_j}^{[1]}-X_{t_j}^{[1]}}
 +\norm{b_\eta(X_{t_j}^{[1]})}\\
 &\le
 Ch^2\sqrt{\eta d/J}
 +(1+h^2)\,B\,.
\end{align*}

The position and momentum components of
\(\bar\Phi_z(\zeta)-\Phi_z(\zeta)\) are obtained by weighting these
gradient differences by \(\omega_{J,j}\) and \(a\omega_j\), respectively.
Since
\[
 \sum_{j\in[J]}|\omega_{J,j}|
 \le h^2\,,
 \qquad
 a\sum_{j\in[J]}\omega_j=ah\le h\,,
\]
we conclude that
\begin{align*}
 \norm{\bar\Phi_z(\zeta)-\Phi_z(\zeta)}
 &\le (h+h^2)\,
       \bigl\{
        Ch^2\sqrt{\eta d/J}
        +(1+h^2)\,B
       \bigr\}\\
 &\le C h\,\bigl(B+Ch^2\sqrt{\eta d/J}\bigr)\\
 &\le C\sqrt d\,
 \bigl(\eta^{3/2} h+ \eta^{1/2} h^3/\sqrt J\bigr)\,.
\end{align*}
The bound is uniform in \(z\) and \(\zeta\), proving
\eqref{eq:rhmc-proof-conditional-mean-error}.
\end{proof}

\subsubsection{Variance bound}
\label{sec:rhmc-proof-variance-bound}

Fix a phase-space input \(z\deq(x_{\rm init},p_{\rm init})\), and put
\(\sigma\deq(1-e^{-h})^{1/2}\).  Collect the OU
variables as
\(\zeta\deq(\zeta_0,\zeta_1)\sim\cN(0,\Id_{2d})\) and all
Gaussian variables used by the stochastic gradients as \(G_{\rm all}\).
Define
\[
 \mathcal G\deq h^{-1}\sum_{j\in[J]}\omega_j
   \widehat g_\eta(X_{t_j}^{[1]};G_j^{[1]})\,,
 \qquad
 \mathcal G'\deq h^{-2}\sum_{j\in[J]}\omega_{J,j}
   \widehat g_\eta(X_{t_j}^{[1]};G_j^{[1]})\,.
\]
Direct substitution in \cref{alg:rhmc-prefix} gives
\begin{equation*}
    \begin{bmatrix}X_{\rm init}^+\\[0.25em] P_{\rm init}^+\end{bmatrix}
 =u+B_h\{\zeta+F(\zeta,G_{\rm all})\}\,,
\end{equation*}
where
\begin{align*}
 u&\deq\begin{bmatrix}x_{\rm init}+ah\, p_{\rm init}\\a^2\,p_{\rm init}\end{bmatrix}\,,\qquad
 B_h\deq\sigma\begin{bmatrix}h\Id&0\\a\Id&\Id\end{bmatrix}\,,\qquad
 F\deq\begin{bmatrix}-h\mathcal G'/\sigma\\
              ah\,(\mathcal G'-\mathcal G)/\sigma\end{bmatrix}\,.
\end{align*}
Moreover, \(h/\sigma\asymp\sqrt h\) and
\(\norm{M_\kappa^{1/2}B_h}_{\op}\le C\sqrt h\).
The purpose of this decomposition is to apply the refined second-order $W_2$ estimate in~\cref{lem:rhmc-proof-random-map-l2} below.
First, we record some Lipschitz estimates.

\begin{lemma}[Lipschitz bounds]
\label{lem:rhmc-proof-phase-lipschitz}
If \(h^2\le1\), then, uniformly in the input,
\begin{align}
 \norm{F(\zeta,G)-F(\zeta',G)}
 &\le Ch^2\,\norm{\zeta-\zeta'}\,,
 \label{eq:rhmc-proof-f-zeta-lip}\\
 \norm{F(\zeta,G)-F(\zeta,G')}
 &\le C\sqrt{\eta h/J}\,\norm{G-G'}\,.
 \notag
\end{align}
\end{lemma}
\begin{proof}
Consider first two OU inputs \(\zeta\deq(\zeta_0,\zeta_1)\) and
\(\zeta'\deq(\zeta_0',\zeta_1')\), with the stochastic gradient variables
fixed.  The map \(F\) does not depend on \(\zeta_1\).  Moreover,
\[
 \norm{X_{t_j}^{[0]}(\zeta_0)
       -X_{t_j}^{[0]}(\zeta_0')}
 =\sigma t_j\,\norm{\zeta_0-\zeta_0'}
 \le \sigma h\,\norm{\zeta_0-\zeta_0'}\,.
\]
Using \eqref{eq:rhmc-proof-stochastic-gradient-y-lip} and then \eqref{eq:rhmc-chebyshev-bounds} gives
\begin{align*}
 &\max_{j\in[J]}\,
 \norm{X_{t_j}^{[1]}(\zeta_0,G^{[0]})
       -X_{t_j}^{[1]}(\zeta_0',G^{[0]})}
 \le \sigma h\,\norm{\zeta_0-\zeta_0'}
 +h^2\,
      \sigma h\,\norm{\zeta_0-\zeta_0'}
 \le 2\sigma h\,\norm{\zeta_0-\zeta_0'}\,.
\end{align*}
Since \(\omega_j\ge0\), \(\sum_{j\in[J]}\omega_j=h\), and
\(\sum_{j\in[J]}|\omega_{J,j}|\le h^2\), it follows that
\[
 \norm{\Delta\mathcal G}
 \le C\sigma h\,\norm{\Delta\zeta_0}\,,
 \qquad
 \norm{\Delta\mathcal G'}
 \le C \sigma h\,\norm{\Delta\zeta_0}\,.
\]
Multiplying by the factor \(h/\sigma\) in the definition of \(F\) proves
\eqref{eq:rhmc-proof-f-zeta-lip}.

Now fix $\zeta$ and compare two collections
$G=(G^{[0]},G^{[1]})$ and $G'=(G^{\prime[0]},G^{\prime[1]})$.
Cauchy--Schwarz, \eqref{eq:rhmc-proof-stochastic-gradient-g-lip}, and
\cref{lem:integrated-weights-l2} give
\[
 \max_{j\in[J]}\,\norm{\Delta X_{t_j}^{[1]}}
 \le Ch^2\sqrt{\eta/J}\,\norm{\Delta G^{[0]}}\,.
\]
Use the Chebyshev--Lobatto weight bounds.
Since the normalizations of $\mathcal G$ and $\mathcal G'$ are $h^{-1}$
and $h^{-2}$, respectively, both satisfy
\[
 \norm{\Delta\mathcal G}\vee\norm{\Delta\mathcal G'}
 \le C\sqrt{\eta/J}\,\bigl\{
 h^2\,\norm{\Delta G^{[0]}}+\norm{\Delta G^{[1]}}\bigr\}\,.
\]
Multiplying by $h/\sigma\le C\sqrt h$ proves the second bound.
\end{proof}

The key to our analysis is the following refined bound.

\begin{lemma}[Second-order \(W_2\) estimate]
\label{lem:rhmc-proof-random-map-l2}
Let \(\EuScript H\) be a Hilbert space, let \(Z\sim\cN(0,\Id_m)\) and
\(G\sim\cN(0,\Id_n)\) be independent, and let
\(f:\R^m\times\R^n\to\R^m\) satisfy
\[
 \norm{f(z,g)-f(z',g)}\le A\,\norm{z-z'}\,,
 \qquad
 \norm{f(z,g)-f(z,g')}\le B\,\norm{g-g'}\,.
\]
Assume that \(f(z,g)-f(z,g')\) belongs to a fixed subspace
\(E\subseteq\R^m\), and let \(\Pi_E\) denote the orthogonal projection
onto \(E\).  If \(L:\R^m\to\EuScript H\) is linear and
\(\bar f(Z)\deq\E[f(Z,G)\mid Z]\), then
\begin{equation}
 W_2\bigl(\law\{L(Z+f(Z,G))\},
           \law\{L(Z+\bar f(Z))\}\bigr)
 \le \sqrt m\,\Bigl\{
 \frac12\,\norm{L\Pi_E}_{\op}\,B^2
 +5\,\norm L_{\op}\,AB
 \Bigr\}\,.
 \label{eq:rhmc-proof-random-map-l2-result}
\end{equation}
\end{lemma}

It is instructive to compare this to a more trivial bound.
The obvious coupling, together with the Gaussian Poincar\'e inequality applied conditionally on $Z$ and over the subspace $E$, yields
\begin{align*}
    \E[\norm{L(Z+f(Z,G)) - L(Z+\bar f(Z))}^2]
    &= \E[\norm{L\Pi_E (f(Z,G) - \bar f(Z))}^2] \\
    &\le \norm{L\Pi_E}_{\op}^2\,\E[\norm{f(Z,G) - \bar f(Z)}^2]
    \le m\,\norm{L\Pi_E}_{\op}^2\,B^2\,.
\end{align*}
This yields
\begin{equation*}
 W_2\bigl(\law\{L(Z+f(Z,G))\},
           \law\{L(Z+\bar f(Z))\}\bigr)
 \le \sqrt m\,
\norm{L\Pi_E}_{\op}\,B\,.
\end{equation*}
This argument does not use any assumption on the distribution of $Z$.
In contrast,~\cref{lem:rhmc-proof-random-map-l2} uses the fact that $Z$ is Gaussian, together with the fact that the difference $f(Z,G) - \bar f(Z)$ is conditionally centered, to obtain a second-order bound which depends quadratically on $A$ and $B$.
The proof, given in
\cref{app:second-order-w2-estimate}, is based on heat flow interpolation and Gaussian integration by parts.
We remark that for the purpose of this paper, we only need the case $E = \R^m$.

\begin{remark}
Related second-order estimates are known for centered perturbations
independent of the Gaussian variable being perturbed: see
\cite[Lemma~1]{SN26} for bounded perturbations on a one-dimensional
subspace, \cite[Theorems~2.1 and~2.7]{PW26} for more general Gaussian
convolution inequalities, and the Ingster--Suslina second moment method \cite{IS03} for divergence bounds.
A complementary heat flow perspective is given by
\cite[Theorem~2.1]{CNW22}: the first unmatched moment determines the
asymptotic decay of \(W_2\) under Gaussian smoothing.
These independent perturbation results do not apply directly when
\(f(Z,G)-\bar f(Z)\) depends on the Gaussian \(Z\) providing the
smoothing; the \(AB\) term in \cref{lem:rhmc-proof-random-map-l2}
accounts for this dependence.
\end{remark}

With this in hand, we can now bound the variance term.

\begin{proposition}[Variance bound]
\label{prop:rhmc-proof-random-error-l2}
Under \(h^2\le1\),
\begin{equation*}
 \Evar
 \le C\sqrt d\,
 \bigl(\eta h^{3/2}/J+\eta^{1/2} h^3/\sqrt J\bigr)\,.
\end{equation*}
\end{proposition}
\begin{proof}
Apply \cref{lem:rhmc-proof-random-map-l2} with
\[
 m=2d\,,\qquad Z=\zeta\,,\qquad G=G_{\rm all}\,,\qquad
 f=F\,,\qquad L=B_h\,,\qquad E=\R^{2d}\,.
\]
By \cref{lem:rhmc-proof-phase-lipschitz}, we may take
\[
 A=Ch^2
 \qquad\text{and}\qquad
 B=C\sqrt{\eta h/J}\,.
\]
Since \(\norm{B_h}_{\op}\le C\sqrt h\),
\cref{eq:rhmc-proof-random-map-l2-result} yields
\begin{align*}
 \Evar
 &\le C\sqrt d\,\norm{B_h}_{\op}\,(B^2+AB)
 \le C\sqrt d\,
 (\eta h^{3/2}/J+\eta^{1/2} h^3/\sqrt J)\,. \qedhere
\end{align*}
\end{proof}

\subsubsection{Discretization error bound}
\label{sec:rhmc-proof-discretization-bound}

Consider the Hamiltonian generator.
\[
 \LHam\deq \ip p{\grad_x}-\ip{g_\eta(x)}{\grad_p}
\]
For a position vector field \(W=(W_i)_{i\in[d]}\), put
\(\delta_{\pi_\eta}W\deq\ip{g_\eta}W-\operatorname{div}_xW\).  Integration by
parts under \(e^{-V_\eta}\) gives the identity
\begin{equation}
 \norm{\delta_{\pi_\eta}W}_{L^2(\pi_\eta)}^2
 \le\norm W_{L^2(\pi_\eta;\HS)}^2
 +\norm{D_xW}_{L^2(\pi_\eta;\HS)}^2\,.
 \label{eq:rhmc-proof-witten-l2}
\end{equation}
Indeed, expanding the square and integrating the two cross terms produces a
contraction of the derivative arrays and the quadratic form generated by
\(\Hess V_\eta\); Cauchy--Schwarz and
\(0\preceq\Hess V_\eta\preceq\Id\) give the displayed bound.

The following lemma relies on Gaussian polynomial bounds collected in
\cref{lem:gaussian-polynomial-estimates}.  In particular, if a
Hilbert-valued momentum polynomial \(F(p)\) and a polynomial vector field
\(W(p)\) have degree at most \(r\), then
\begin{equation}
 \norm{D_pF}_{L^2(\gamma;\HS)}\le\sqrt r\,\norm F_{L^2(\gamma)}\,,
 \qquad
 \norm{\delta_\gamma W}_{L^2(\gamma)}
 \le\sqrt{r+1}\,\norm W_{L^2(\gamma;\HS)}\,.
 \label{eq:rhmc-proof-momentum-l2}
\end{equation}

\begin{lemma}[Stationary iterated derivatives in \(L^2\)]
\label{lem:rhmc-proof-iterated-derivative-l2}
For every integer \(k\ge0\),
\begin{equation*}
 \norm{\LHam^kg_\eta}_{L^2(\Pi_\eta)}
 \le C^{k+1}(k!)^{3/2}\sqrt d\,\eta^{-k/2}\,.
\end{equation*}
\end{lemma}

\begin{proof}
For a tensor-valued function \(F\), define the weighted Sobolev norm
\[
 \norm F_{\mathsf W^{N,2}}
 \deq\sum_{j=0}^N\frac{\eta^{j/2}}{j!}\,
       \norm{D_x^jF}_{L^2(\Pi_\eta;\HS)}\,.
\]
The weight ratio gives
\(\norm{D_xF}_{\mathsf W^{N-1,2}}
 \le N\eta^{-1/2}\,\norm F_{\mathsf W^{N,2}}\).  We next establish the
corresponding bound for \(\delta_{\pi_\eta}\).  For \(0\le j\le N-1\), the
notation
\[
 \operatorname{contr}(D_x^\ell g_\eta,D_x^{j-\ell}W)
 \deq\operatorname{Sym}\Bigl(
   \sum_{a\in[d]}D_x^\ell(g_\eta)_a
   \otimes D_x^{j-\ell}W_a
 \Bigr)
\]
denotes contraction over the position index \(a\), followed by
normalized symmetrization over the \(j\) derivative indices.  The
higher-order Leibniz rule \cite[Exercise~1.2.13]{N06} gives
\[
 D_x^j(\delta_{\pi_\eta}W)
 =\delta_{\pi_\eta}(D_x^jW)
  +\sum_{\ell=1}^j\binom j\ell
    \operatorname{contr}(D_x^\ell g_\eta,D_x^{j-\ell}W)\,.
\]
By \eqref{eq:rhmc-proof-witten-l2}, the weighted sum of the first terms is
bounded by
\begin{align*}
 \sum_{j=0}^{N-1}\frac{\eta^{j/2}}{j!}
   \,\norm{\delta_{\pi_\eta}(D_x^jW)}_{L^2(\Pi_\eta)}
 &\le
 \sum_{j=0}^{N-1}\frac{\eta^{j/2}}{j!}
 \,\bigl(
   \norm{D_x^jW}_{L^2(\Pi_\eta;\HS)}
   +\norm{D_x^{j+1}W}_{L^2(\Pi_\eta;\HS)}
 \bigr)\\
 &\le \frac{2N}{\sqrt\eta}\,
   \norm W_{\mathsf W^{N,2}}\,.
\end{align*}
For the remaining terms, \eqref{eq:rhmc-proof-derivative-slot} implies
\[
 \norm{\operatorname{contr}
   (D_x^\ell g_\eta,D_x^{j-\ell}W)}_{\HS}
 \le \ell!\,\eta^{-(\ell-1)/2}\,
      \norm{D_x^{j-\ell}W}_{\HS}\,.
\]
Relative to the weight of \(D_x^{j-\ell}W\), the coefficient of the term
with \(\ell\) derivatives on \(g_\eta\) is
\[
 \frac{\eta^{j/2}}{j!}\,\binom j\ell\,\ell!\,
 \eta^{-(\ell-1)/2}
 \bigg/\frac{\eta^{(j-\ell)/2}}{(j-\ell)!}
 =\sqrt\eta\,.
\]
After setting \(s=j-\ell\), each fixed position derivative \(D_x^sW\)
appears for at most \(N-1-s\le N\) choices of \(\ell\).  Thus, the
commutator contribution is at most
\(N\sqrt\eta\,\norm W_{\mathsf W^{N,2}}\).
Combining the two contributions, using \(0<\eta\le1\) and \(N\ge1\), yields
\[
 \norm{\delta_{\pi_\eta}W}_{\mathsf W^{N-1,2}}
 \le \frac{CN}{\sqrt\eta}\,\norm W_{\mathsf W^{N,2}}\,.
\]
It remains to control the Hamiltonian generator. One can check that
\begin{equation}
 \LHam F=\delta_\gamma(D_xF)-\delta_{\pi_\eta}(D_pF)\,.
 \label{eq:Lham-factorization}
\end{equation}
Suppose that \(F\) has momentum
degree at most \(r\).  Since \(D_x\) commutes with \(\delta_\gamma\), the
Gaussian divergence estimate in \eqref{eq:rhmc-proof-momentum-l2} gives
\begin{align*}
 \norm{\delta_\gamma(D_xF)}_{\mathsf W^{N-1,2}}
 &\le\sqrt{r+1}\,\norm{D_xF}_{\mathsf W^{N-1,2}}
 \le\frac{N\sqrt{r+1}}{\sqrt\eta}\,
       \norm F_{\mathsf W^{N,2}}\,.
\end{align*}
For the second term, the Sobolev divergence estimate above and the Gaussian
gradient estimate in \eqref{eq:rhmc-proof-momentum-l2} yield
\begin{align*}
 \norm{\delta_{\pi_\eta}(D_pF)}_{\mathsf W^{N-1,2}}
 &\le\frac{CN}{\sqrt\eta}\,
       \norm{D_pF}_{\mathsf W^{N,2}}
 \le\frac{CN\sqrt r}{\sqrt\eta}\,
       \norm F_{\mathsf W^{N,2}}\,.
\end{align*}
Combining the last two displays proves
\begin{align}\label{eq:Lham_bd}
 \norm{\LHam F}_{\mathsf W^{N-1,2}}
 \le\frac{CN\sqrt{r+1}}{\sqrt\eta}\,
      \norm F_{\mathsf W^{N,2}}\,.
\end{align}
Finally, integration by parts gives
\(\norm{g_\eta}_{L^2(\pi_\eta)}^2=\E_{\pi_\eta}\Delta V_\eta\le d\), while
\eqref{eq:rhmc-proof-derivative-full} gives
\[
 \norm{g_\eta}_{\mathsf W^{N,2}}
 \le \sqrt d+N\sqrt{\eta d}
 \le (N+1)\sqrt d\,.
\]
Since
\(\LHam^jg_\eta\) has momentum degree at most \(j\), iterating the
bound~\eqref{eq:Lham_bd} for \(j=0,\ldots,k-1\) produces
\(k!\sqrt{k!}\).  The extra factor \(k+1\) from the initialization is
absorbed into \(C^{k+1}\), proving the claim.
\end{proof}

Let \((X_t,P_t)_{t\ge 0}\) be the exact Hamiltonian flow and put
\(g(t)\deq g_\eta(X_t)\) when the initial state has law \(\Pi_\eta\).  Invariance
and \cref{lem:rhmc-proof-iterated-derivative-l2} give
\begin{equation}
 \sup_{0\le t\le h}\,\norm{g^{(k)}(t)}_{L^2}
 \le C^{k+1}(k!)^{3/2}\sqrt d\,\eta^{-k/2}\,.
 \label{eq:rhmc-proof-gradient-derivatives-l2}
\end{equation}

\begin{proposition}[Discretization error bound]
\label{prop:rhmc-proof-deterministic-defect-l2}
If \(h^2\le1\), then
\begin{equation}
 \Edisc
 \le C\sqrt d\,\Bigl[
 B_{J,2} h\,\Bigl(\frac h{\sqrt\eta}\Bigr)^J
 +h^5
 \Bigr]\,.
 \label{eq:rhmc-proof-det-error-l2}
\end{equation}
Here, one may take
\(B_{J,2}=C_0^{J+1}\sqrt{J!}\), where \(C_0\) is universal.
\end{proposition}
\begin{proof}
For an \(L^2(\Pi_\eta;\R^d)\)-valued curve \(f\) on \([0,h]\), write
\[
 \norm f_{\infty,2}
 \deq\sup_{0\le t\le h}\,\norm{f(t)}_{L^2}\,.
\]
In particular, we can identify the position coordinates $X$, $X^{[0]}$, $X^{[1]}$ of the Hamiltonian flow and Picard iterations, respectively, as well as $g$, as $L^2(\Pi_\eta;\R^d)$-valued curves.

Apply the Banach-valued interpolation remainder in
\cref{lem:chebyshev-banach-remainder} to the curve
\(g:[0,h]\to L^2(\Pi_\eta;\R^d)\).  Together with
\eqref{eq:rhmc-proof-gradient-derivatives-l2}, it gives
\begin{align*}
 \mathcal E_{J,2}
 &\deq\sup_{0\le t\le h}\,
 \norm{g(t)-\mathcal I_{J,h}g(t)}_{L^2}
 \le\frac{h^J}{J!}\sup_{0\le t\le h}\,
       \norm{g^{(J)}(t)}_{L^2}\\
 &\le C^{J+1}\sqrt{J!}\,\sqrt d\,
       \Bigl(\frac h{\sqrt\eta}\Bigr)^J
 \le B_{J,2}\sqrt d\,\Bigl(\frac h{\sqrt\eta}\Bigr)^J\,.
\end{align*}
Since
the exact flow is stationary under \(\Pi_\eta\), integration by parts gives
\[
 \norm{g_\eta\circ X}_{\infty,2}^2
 =\E_{\pi_\eta}\norm{g_\eta}^2
 =\E_{\pi_\eta}\Delta V_\eta
 \le d\,.
\]
Therefore, the integral equation for $X$ shows that\begin{equation}
 \norm{X^{[0]}-X}_{\infty,2}
 \le\sqrt d\,\frac{h^2}{2}\,.
 \label{eq:w2-ballistic-exact-error}
\end{equation}
To control the first Picard iteration, insert the interpolant of the force along
the exact trajectory:
\begin{align*}
 X_t^{[1]}-X_t
 &=\int_0^t(t-s)\,
   \bigl\{g_\eta(X_s)-\mathcal I_{J,h}[g_\eta\circ X^{[0]}](s)\bigr\}
   \,\dd s\\
 &=\int_0^t(t-s)\,
   \bigl\{g_\eta(X_s)-\mathcal I_{J,h}[g_\eta\circ X](s)\bigr\}
   \,\dd s
+
   \int_0^t(t-s)\,
   \mathcal I_{J,h}[g_\eta\circ X-g_\eta\circ X^{[0]}](s)\,\dd s\,.
\end{align*}
The proof of \cref{lem:integrated-weights-l2} applies to
$k=(t-\cdot)_+$ for every $t\in[0,h]$, so its consequence is
uniform in $t$.  Since $g_\eta$ is $1$-Lipschitz,
\eqref{eq:w2-ballistic-exact-error} gives
\[
 \norm{X^{[1]}-X}_{\infty,2}
 \le \frac{h^2}{2}\,\mathcal E_{J,2}
      +h^2\,\norm{X^{[0]}-X}_{\infty,2}
 \le \frac{h^2}{2}\,\mathcal E_{J,2}+\frac{\sqrt d\,h^4}{2}\,.
\]
Set $(\widehat X_h,\widehat P_h)=(X_h^{[2]},P_h^{[2]})$.
At the momentum endpoint, positivity and the sum of the weights give
\begin{align*}
 \norm{\widehat P_h-P_h}_{L^2}
 &\le h\mathcal E_{J,2}+h\,\norm{X^{[1]}-X}_{\infty,2}\\
 &\le (1+h^2/2)\,h\mathcal E_{J,2}+\sqrt d\,h^5/2
 \le C\mathcal E_{J,2} h+C\sqrt d\,h^5\,.
\end{align*}
Similarly,
\begin{align*}
 \norm{\widehat X_h-X_h}_{L^2}
 &\le \frac{h^2}{2}\,\mathcal E_{J,2}
       +h^2\,\norm{X^{[1]}-X}_{\infty,2}\\
 &\le \frac{h^2}{2}\,(1+h^2)\,\mathcal E_{J,2}+\sqrt d\,h^6/2
 \le C\mathcal E_{J,2}h^2 +C\sqrt d\,h^6\,.
\end{align*}
In the last lines, we used \(h^2\le1\).

Finally, start both the numerical phase and an exact Hamiltonian phase from
\(\Pi_\eta\), and use the same Gaussian variables in the two OU refreshes.
The first refresh preserves \(\Pi_\eta\), so the preceding stationary bounds
apply.  The exact Hamiltonian flow and the final refresh also preserve
\(\Pi_\eta\).  In the final refresh the common Gaussian increment cancels,
while the momentum difference is multiplied by \(e^{-h/2}\).  Consequently,
using also \(h\le1\), which follows from \(h^2\le1\) and
\(h>0\),
\begin{align*}
 \Edisc
 &\le\norm{\widehat X_h-X_h}_{L^2}
      +\norm{\widehat P_h-P_h}_{L^2}
 \le C \mathcal E_{J,2}\,h+C\sqrt d\,h^5
 \le C\sqrt d\,\Bigl[
   B_{J,2}h\,\Bigl(\frac h{\sqrt\eta}\Bigr)^J
   +h^5
 \Bigr]\,.
\end{align*}
This proves \eqref{eq:rhmc-proof-det-error-l2}.
\end{proof}

\subsection{Proof of
  \texorpdfstring{\hyperref[thm:w2-main]{Theorem~\ref*{thm:w2-main}}}
                 {Theorem 1.1}}
\label{sec:rhmc-proof-main-theorem}

We record the following standard initialization lemma.

\begin{lemma}[Initialization]
\label{lem:rhmc-proof-initial-radius}
Let \(\widehat\mu_0\deq\delta_{x_{\rm ref}}\otimes\cN(0,\Id)\).  For every
\(q\ge2\),
\begin{equation*}
 W_{q,M_\kappa}(\widehat\mu_0,\Pi_\eta)
 \le C\sqrt{\kappa\,(d+q)}\,.
\end{equation*}
\end{lemma}

\begin{proof}
Let \(x_\star\) minimize \(V\).  Strong monotonicity and the assumption
on \(\norm{\grad V(x_{\rm ref})}\) give
\[
 \norm{x_{\rm ref}-x_\star}
 \le\kappa\,\norm{\grad V(x_{\rm ref})}
 \le\sqrt{\kappa d}\,.
\]
If \(X\sim\pi\), concentration from strong log-concavity
\cite[Chapter~5]{BGL14} and integration by parts give
\(\norm{X-x_\star}_{L^q}\le C\sqrt{\kappa\,(d+q)}\).  A draw from $\pi_\eta$ is given by \(X+\sqrt\eta\, G\), and
\(\sqrt\eta\,\norm G_{L^q}\le C\sqrt{\kappa\,(d+q)}\) because
\(\eta\le1\le\kappa\).  Coupling the initial and stationary momenta
identically, the claim follows from the triangle inequality and \(M_\kappa\preceq C\Id\).
\end{proof}

The following lemma handles the terminal step.

\begin{lemma}[Terminal RGO step in \(W_2\)]
\label{lem:w2-rgo-lift}
Suppose \(\kappa^{-1}\Id\preceq\Hess V\preceq\Id\).  For every
\(\eta>0\), write
\(\pi_\eta\deq\pi*\cN(0,\eta\Id)\) and
\(R_\eta(y,\dd x)\deq R_{\eta,y}(\dd x)\).  Then, for every
\(y,y'\in\R^d\),
\begin{equation*}
 W_2(R_{\eta,y},R_{\eta,y'})
 \le\frac{\kappa}{\kappa+\eta}\,\norm{y-y'}
 \le\norm{y-y'}\,.
\end{equation*}
Moreover, \(\pi_\eta R_\eta=\pi\).  Consequently, for any center laws
\(\mu,\mu'\),
\begin{equation*}
 W_2(\mu R_\eta,\mu'R_\eta)
 \le W_2(\mu,\mu')\,.
\end{equation*}
\end{lemma}

\begin{proof}
These are the standard invariance and contraction properties of the
restricted Gaussian oracle; see \cite{LST21,CCSW22}.
\end{proof}

The proof of~\cref{thm:w2-main} is completed by
plugging the bias, variance, and
discretization error bounds from \cref{sec:rhmc-proof-local-bounds} into
\eqref{eq:rhmc-proof-global-recurrence-l2} and performing some bookkeeping.

\begin{proof}[Proof of \cref{thm:w2-main}]
    Throughout, we assume access to exact proximal queries; this assumption can be removed via the arguments in in \cref{app:gradient-only-prox}. After the change of variables \(z=\sqrt\beta\,x\), it suffices to work
under \(\beta=1\) and prove \(W_2\le\sqrt\kappa\,\eps\).
Choose \begin{align*}
 \mathfrak L&\deq \log\frac{C\kappa d}{\eps}\,,
 \qquad
 J\deq \lceil C\mathfrak L^2\rceil\,,
\qquad
 h\deq \frac c{\mathfrak L}
 \min\Bigl\{\frac1\kappa,
   \frac{\eps^{1/3}}{\kappa^{1/6}d^{1/6}}\Bigr\}\,,
 \qquad
 N\deq \Bigl\lceil\frac{C\kappa\mathfrak L}{h}\Bigr\rceil\,,
\end{align*}
where \(c,C>0\) are universal constants; the following assertions hold after adjusting these constants appropriately.
Here, $\mathfrak L$ is introduced to streamline the bookkeeping of logarithmic terms.

Let \(\theta>0\),
to be chosen below, and choose \(\eta>0\) so that
\(h^2/\eta=\theta\).
The explicit choice of \(B_{J,2}\) in
\cref{prop:rhmc-proof-deterministic-defect-l2} gives
\begin{align}
 B_{J,2}\theta^{J/2}
 &=C_0^{J+1}\sqrt{J!}\,\theta^{J/2}
 \le C_0^{J+1}(J\theta)^{J/2}
 \le e^{-c_1J}
 \le h^4\,,
 \label{eq:w2-interpolation-small}
\end{align}
provided that \(\theta\le c\mathfrak L^{-2}\). We henceforth choose \(\theta\deq c/\mathfrak L^2\).
Then, \(J\ge2\) and
\(h\le c/\kappa\),
so the restrictions in
\cref{prop:rhmc-proof-deterministic-defect-l2,prop:rhmc-proof-contraction}
hold, and moreover $\eta \le 1$.

Take square roots in \eqref{eq:rhmc-proof-global-recurrence-l2} and substitute
the local bounds from
\cref{prop:rhmc-proof-random-error-l2,prop:rhmc-proof-conditional-mean-error,prop:rhmc-proof-deterministic-defect-l2}.
Together with \cref{lem:rhmc-proof-initial-radius}, this gives
\begin{align*}
 W_{2,M_\kappa}(\widehat\mu_N^{\rm ph},\Pi_\eta)
 \le{}&C\sqrt{\kappa d}\,e^{-cNh/\kappa}
 +C\sqrt{\kappa d}\,
 (\eta h/J+\eta^{1/2} h^{5/2}/\sqrt J)
 \nonumber\\
 &\qquad{} +C\kappa\sqrt d\,
 (\eta^{3/2}+ \eta^{1/2} h^2/\sqrt J
       +h^4)\,,
\end{align*}
where \(\widehat\mu_N^{\rm ph}\) is the final phase-space law.  The Picard iteration error
term is included in the last monomial by
\eqref{eq:w2-interpolation-small}. Since
\(\eta=h^2/\theta\), the preceding bounds on
\(h,\eta\) show that the sum of the last five terms is at most \(\sqrt\kappa\,\eps/4\) for $c$ appropriately small.
This bookkeeping is uniform over \(0<\eps\le1\). The definition of \(N\) makes the initial
term at most \(\sqrt\kappa\,\eps/4\).

By \cref{lem:w2-rgo-lift},
\[
 W_2(\widehat \mu_NR_\eta,\pi)
 =W_2(\widehat \mu_NR_\eta,\pi_\eta R_\eta)
 \le W_2(\widehat \mu_N,\pi_\eta)
 \le CW_{2,M_\kappa}(\widehat\mu_N^{\rm ph},\Pi_\eta)
 \le\frac{\sqrt\kappa\,\eps}2\,,
\]
after adjusting universal constants.  Let \(\widehat R_{\eta,y}\) be the
conditional output law of the terminal FORS routine in
\cref{alg:w2-complete}, and set
\[
 \delta_{\rm term}
 \deq
 \min\Bigl\{\frac12,\frac{\kappa\eps^2}{8\eta}\Bigr\}\,.
\]
By \cref{thm:fors-implementation} with \(q=2\) and
\(\varepsilon^2\deq\min\{1/4,\delta_{\rm term}\}\), uniformly in \(y\),
we can ensure that
\[
 \KL(\widehat R_{\eta,y} \mmid R_{\eta,y})
 \leq\Ren_2(\widehat R_{\eta,y} \mmid R_{\eta,y})
 \leq\delta_{\rm term}\,.
\]
The RGO potential is
\((\kappa^{-1}+\eta^{-1})\)-strongly convex.
Talagrand's inequality therefore gives
\begin{equation*}
 W_2^2(\widehat R_{\eta,y},R_{\eta,y})
 \le \frac{2\delta_{\rm term}}{\kappa^{-1}+\eta^{-1}}
 \le 2\eta\delta_{\rm term}
 \le\frac{\kappa\eps^2}{4}\,.
\end{equation*}
The triangle inequality then yields
\(W_2(\widehat\mu_N\widehat R_\eta,\pi)\le\sqrt\kappa\,\eps\).

Each phase uses \(2J\) gradients and \(2J\) proximal calls. Substitution
of \(N\) gives
\begin{equation*}
 2JN
 \le C\,\Bigl\{\kappa^2
       +\frac{\kappa^{7/6}d^{1/6}}{\eps^{1/3}}\Bigr\}\,
       \mathfrak L^4
\end{equation*}
calls of each type.  By \cref{thm:fors-implementation}, the terminal routine
has expected cost at most $C\mathfrak L^{5/3}\,
(1+\eps^{4/9} d^{1/9}/\kappa^{2/9})$.
Any fixed power of
\(\mathfrak L\) is dominated by a positive power of
\(\kappa d/\eps\).  Hence this cost is dominated by the Picard
HMC bound for \(\kappa,d\ge1\) and \(0<\eps\le1\).  This proves the theorem.
\end{proof}

\section{\texorpdfstring{\(W_q\)}{W\_q} analysis}
\label{sec:higher-moment-analysis}

For our warm start application, we need to upgrade the $W_2$ guarantee to a $W_q$ guarantee, for $q\ge 2$.
We organize the proof in parallel with \cref{sec:w2-proof}, and along the way, we describe the new ingredients needed for this part of the analysis.

\subsection{\texorpdfstring{\(W_q\)}{W\_q} guarantee}
\label{sec:rhmc-higher-moment-statement}

\begin{theorem}[$W_q$ guarantee]
\label{thm:rhmc-module}
Fix any \(c_0>0\).  There is a constant \(C=C(c_0)>0\) with the following
property.  Suppose
\(\kappa\ge1\), \(V\in C^2(\R^d)\), and
\begin{equation*}
 \kappa^{-1}\Id\preceq\Hess V\preceq\Id\,,
 \qquad
 \norm{\grad V(x_{\rm ref})}\le\sqrt{d/\kappa}\,.
\end{equation*}
For every \(q\ge2\) and \(0<\eps\le1\), put
\begin{equation*}
 \mathfrak L_q\deq q+\log\frac{e\kappa d}{\eps}\,.
\end{equation*}
Then, there is a smoothing variance \(0<\eta\le c_0\) for which a gradient-only
implementation of Picard HMC returns a sample whose output law \(\widehat\pi_\eta\) satisfies
\begin{equation*}
 W_q(\widehat\pi_\eta,\pi_\eta)\le\eps\,.
\end{equation*}
Its expected number of gradient queries is at most
\begin{equation*}
 C\,\Bigl\{
   \kappa^2
   +\frac{\kappa^{4/3}\,(d+q)^{1/6}}{\eps^{1/3}}
 \Bigr\}\,
 \mathfrak L_q^{9/2}\,.
\end{equation*}
\end{theorem}

Note that unlike~\cref{thm:w2-main}, this guarantee targets the smoothed distribution $\pi_\eta$; that is, we omit the terminal FORS step.
This is because the terminal FORS guarantee does not immediately hold in $W_q$; moreover, the guarantee for sampling from $\pi_\eta$ will be enough for our eventual warm start application.
Indeed, the guarantee is intended to implement the smooth sampler \ref{ass:smoothed-wp} in \cref{sec:recursive-warm-start}.

\subsection{\texorpdfstring{$W_q$}{W\_q} local error analysis}
\label{sec:rhmc-higher-moment-local-error}

In the \(W_2\) proof, the local error framework (\cref{lem:rhmc-proof-local-error}) separates a weak error from a
strong error by expanding a square.  For \(q>2\), we replace
it with the following smoothness inequality.

\begin{lemma}[$L^q$ smoothness]
\label{lem:rhmc-proof-lr-tools}
Let \(q\ge2\).  If \(Z\) is a centered random vector in a Hilbert space and
\(a\) is deterministic, then
\begin{equation}
 (\E\norm{a+Z}^q)^{2/q}
 \le\norm a^2+(q-1)\,(\E\norm Z^q)^{2/q}\,.
 \label{eq:rhmc-proof-hilbert-smoothness}
\end{equation}
\end{lemma}

\begin{proof}
This follows from the smoothness of $L^q$; see \cite{P94}.  It also
follows directly by differentiating
\(t\mapsto(\E\norm{a+tZ}^q)^{2/q}\): centeredness gives zero derivative at
the origin and H\"older's inequality bounds the second derivative by
\(2\,(q-1)\,\norm Z_{L^q}^2\).
\end{proof}

\begin{lemma}[Local error framework in \(W_q\)]
\label{lem:rhmc-proof-centered-kernel-perturbation}
Let \(P,\widehat P\) be Markov kernels on \(\R^m\), and let \(q\ge2\).
Suppose that, for every \(x,y\in\R^m\), there is a coupling
\(X\sim\delta_xP\), \(Y\sim\delta_yP\) satisfying
\begin{equation*}
 \norm{\E(X-Y)}
 \le\rho\,\norm{x-y}\,,
 \qquad
 \norm{(X-Y)-\E(X-Y)}_{L^q}
 \le\chi_q\,\norm{x-y}\,.
\end{equation*}
Suppose also that, writing \(\widehat X\sim\delta_x\widehat P\) and
\(X\sim\delta_xP\),
\begin{equation*}
 W_q(\delta_x\widehat P,\delta_xP)
 \le\mathcal E_{\rm strong}\,,
 \qquad
 \norm{\E\widehat X-\E X}
 \le\mathcal E_{\rm weak}\,.
\end{equation*}
Assume $0<\rho<1$ and $4\,(q-1)\,\chi_q^2\le\rho\,(1-\rho)$.
Then, for all probability laws \(\mu,\nu\) on \(\R^m\),
\begin{equation*}
 W_q^2(\mu\widehat P,\nu P)
 \le\rho W_q^2(\mu,\nu)
    +\frac{16\,(q-1)}{1+\rho}\,\mathcal E_{\rm strong}^2
    +\frac2{1-\rho}\,\mathcal E_{\rm weak}^2\,.
\end{equation*}
\end{lemma}
\begin{proof}
Let \((Z,Z')\) be an optimal $W_q$ coupling of $\mu$ and $\nu$.  Given
\((Z,Z')\), let \(\widehat X\sim\widehat P(Z,\cdot)\) and \(X\sim P(Z,\cdot)\) be optimally coupled in $W_q$, and let $Y \sim P(Z',\cdot)$ be coupled to $X$ via the coupling in the theorem statement.
Write $\Delta \deq \widehat X - X$.
Conditionally on \((Z,Z')\), Jensen's inequality gives
\begin{equation*}
 \norm{\E[\Delta\mid Z,Z']}
 \le\mathcal E_{\rm weak}\,,
 \qquad
 \norm{\Delta-\E[\Delta\mid Z,Z']}_{L^q\mid Z,Z'}
 \le2\mathcal E_{\rm strong}\,.
\end{equation*}
By Minkowski's inequality and Young's inequality, for $\lambda > 0$,
\begin{align*}
    \norm{\widehat X - Y}_{L^q\mid Z,Z'}^2
    &\le (1+\lambda)\,\norm{\Delta - \E[\Delta \mid Z,Z'] + X-Y}_{L^q\mid Z,Z'}^2 + (1+\lambda^{-1})\,\norm{\E[\Delta \mid Z,Z']}_{L^q\mid Z,Z'}^2\,.
\end{align*}
For the first term, by~\eqref{eq:rhmc-proof-hilbert-smoothness},
\begin{align*}
    &\norm{\Delta-\E[\Delta\mid Z,Z']+X-Y}_{L^q\mid Z,Z'}^2 \\
    &\qquad \le \norm{\E[X-Y\mid Z,Z']}^2 + (q-1)\,\norm{\Delta-\E[\Delta \mid Z,Z']+X-Y -\E[X-Y\mid Z,Z']}_{L^q\mid Z,Z'}^2 \\
              &\qquad \le \rho^2\,\norm{Z-Z'}^2+
 (q-1)\,\bigl(2\mathcal E_{\rm strong}
             +\chi_q\,\norm{Z-Z'}\bigr)^2\,.
\end{align*}
Using \((u+v)^2\le2u^2+2v^2\) and the assumption on \(\chi_q\), the
right-hand side is at most
\[
 \frac{\rho\,(1+\rho)}2\,\norm{Z-Z'}^2
 +8\,(q-1)\,\mathcal E_{\rm strong}^2\,.
\]
Take the \(L^{q/2}\) norm over $(Z,Z')$.
Then, choosing $\lambda = \frac{1-\rho}{1+\rho}$,
\begin{align*}
    \norm{\widehat X-Y}_{L^q}^2
    &= \bigl\lVert \norm{\widehat X-Y}_{L^q\mid Z,Z'}^2 \bigr\rVert_{L^{q/2}} \\[0.25em]
    &\le \Bigl\lVert (1+\lambda)\,\Bigl( \frac{\rho\,(1+\rho)}{2}\,\norm{Z-Z'}^2 + 8\,(q-1)\,\mathcal E_{\rm strong}^2\Bigr) + (1+\lambda^{-1})\,\mathcal E_{\rm weak}^2 \Bigr\rVert_{L^{q/2}} \\[0.25em]
    &= \rho\,\norm{Z-Z'}_{L^q}^2 + \frac{16\,(q-1)}{1+\rho}\,\mathcal E_{\rm strong}^2 + \frac{2}{1-\rho}\,\mathcal E_{\rm weak}^2\,. \qedhere
\end{align*}
\end{proof}

Recall the kernels
\(\widehat K\), \(\bar K\), and \(K\) from \cref{sec:rhmc-proof}.  For
\(q\ge2\), define the centered variance and bias errors
\begin{align*}
 \Evar[q]
 &\deq\sup_{z\in\R^{2d}}
   W_q(\delta_z\widehat K,\delta_z\bar K)\,,
 &
 \Ebias
 &\deq\sup_{z,\zeta\in\R^{2d}}\,
   \norm{\bar\Phi_z(\zeta)-\Phi_z(\zeta)}\,.
\end{align*}
Take \(\widehat P\deq\widehat K\) and \(P\deq K\).  By the definition of
\(\bar\Phi_z\), the laws \(\delta_z\widehat K\) and
\(\delta_z\bar K\) have the same mean.  The triangle inequality yields
\begin{equation*}
 \mathcal E_{\rm strong}
 \le\Evar[q]+\Ebias\,,
 \qquad
 \mathcal E_{\rm weak}
 \le\Ebias\,.
\end{equation*}
Also, define
\begin{equation*}
 \Edisc[q]\deq W_q(\Pi_\eta K,\Pi_\eta)\,.
\end{equation*}
The analogue of \cref{prop:rhmc-proof-global-recurrence-l2} is the
following theorem.

\begin{theorem}[Local error analysis in \(W_q\)]
\label{prop:rhmc-proof-global-recurrence}
Let \(J\ge2\), and assume $h \le c/(\kappa q)$.
Then, for every phase-space law \(\widehat\mu\),
\begin{equation}
 W_{q,M_\kappa}^2(\widehat\mu\widehat K,\Pi_\eta)
 \le(1-ch/\kappa)\,W_{q,M_\kappa}^2(\widehat\mu,\Pi_\eta)
 +Cq \,\Evar[q]^2
 +\frac{C\kappa}{h}\,(\Ebias^2+\Edisc[q]^2)\,.
 \label{eq:rhmc-proof-one-step-recurrence}
\end{equation}
Consequently,
\begin{align}
 W_{q,M_\kappa}^2(\widehat\mu\widehat K^N,\Pi_\eta)
 &\le e^{-cNh/\kappa}\,W_{q,M_\kappa}^2(\widehat\mu,\Pi_\eta)
 +\frac{C\kappa q}{h}\,\Evar[q]^2
 +\frac{C\kappa^2}{h^2}\,(\Ebias^2+\Edisc[q]^2)\,.
 \label{eq:rhmc-proof-global-recurrence}
\end{align}
\end{theorem}

\begin{proof}
Repeat the proof of \cref{prop:rhmc-proof-global-recurrence-l2}, using
\cref{lem:rhmc-proof-centered-kernel-perturbation} in place of the \(W_2\)
local error framework; \cref{prop:rhmc-proof-contraction,lem:rhmc-proof-increment-coupling}
still apply and give \(\rho=1-ch/\kappa\) and \(\chi_q\le2Ch\).
\end{proof}

\subsection{\texorpdfstring{$W_q$}{W\_q} local error bounds}
\label{sec:rhmc-higher-moment-local-bounds}

The bias term $\Ebias$ is the same as before, and is controlled by \cref{prop:rhmc-proof-conditional-mean-error}.
Thus, we focus on controlling the other two terms.

\subsubsection{Variance bound}

For the variance bound, we must extend the second-order $W_2$ estimate in~\cref{lem:rhmc-proof-random-map-l2} to $W_q$.

\begin{lemma}[Second-order $W_q$ estimate]
\label{lem:rhmc-proof-random-map}
In the setting of~\cref{lem:rhmc-proof-random-map-l2}, for all $q\ge 2$, it holds that
\begin{equation}
 W_q\bigl(\law\{L(Z+f(Z,G))\}\,,
           \law\{L(Z+\bar f(Z))\}\bigr)
 \le C\sqrt m\,\bigl(
 \sqrt q\,\norm{L\Pi_E}_{\op}\,B^2
 +q\,\norm L_{\op}\,AB\bigr)\,.
 \label{eq:rhmc-proof-random-map-result}
\end{equation}
\end{lemma}

The proof, which is given in \cref{app:second-order-wasserstein-estimates}, relies on the same strategy as \cref{lem:rhmc-proof-random-map-l2}, but relies on a more involved Gaussian divergence inequality (\cref{lem:gaussian-sobolev-inequalities}).
Using this, we can control the variance term in $W_q$.

\begin{proposition}[Variance bound in $W_q$]
\label{prop:rhmc-proof-random-error}
If $h\le1$, then, for every \(q\ge2\),
\begin{equation*}
 \Evar[q]
 \le C\sqrt d\,
 \bigl(\eta h^{3/2} q^{1/2}/J
 +\eta^{1/2} h^3 q/\sqrt J\bigr)\,.
\end{equation*}
\end{proposition}
\begin{proof}
Apply \cref{lem:rhmc-proof-random-map} with the constants in
\cref{lem:rhmc-proof-phase-lipschitz}.
\end{proof}

\subsubsection{Discretization error bound}
\label{sec:rhmc-proof-deterministic}

Next, we need to replace the $L^2$ bound~\eqref{eq:rhmc-proof-witten-l2} with an $L^q$ version.
The following lemma reduces this to an $L^q$ bound for the Gaussian divergence operator, which in turn is handled in~\cref{lem:gaussian-sobolev-inequalities}.

\begin{lemma}[Divergence bound]
\label{lem:rhmc-proof-convolution-divergence}
Let \(X\sim\pi\), \(G\sim\cN(0,\Id)\), and
\(P\sim\cN(0,\Id)\) be independent, and set \(Y\deq X+\sqrt\eta\,G\).  For a smooth
tensor-valued position vector field \(W=(W_i)_{i\in[d]}\), we have the representation
\begin{equation}
 \delta_{\pi_\eta}W(Y,P)
 =\E\bigl[
     \eta^{-1/2}\,\delta_\gamma\{W(X+\sqrt\eta\,\cdot,P)\}(G)
   \bigm\vert Y, P\bigr]\,.
 \label{eq:rhmc-proof-convolution-divergence-identity}
\end{equation}
Consequently, for every \(q\ge2\),
\begin{equation}
 \norm{\delta_{\pi_\eta}W}_{L^q(\Pi_\eta)}
 \le C\sqrt{q/\eta}\,
       \norm W_{L^q(\Pi_\eta;\HS)}
   +Cq\,\norm{DW}_{L^q(\Pi_\eta;\HS)}\,.
 \label{eq:rhmc-proof-convolution-divergence-bound}
\end{equation}
The constant is independent of all tensor dimensions.
\end{lemma}
\begin{proof}
The identity in \eqref{eq:rhmc-proof-potential-gradient-covariance} is equivalent to $g_\eta(Y)=\eta^{-1/2}\,\E[G\mid Y]$.
The Gaussian divergence gives
\begin{align*}
    \delta_\gamma\{W(X+\sqrt \eta\,\cdot, P)\}(G)
    &= \langle G, W(Y,P)\rangle - \eta^{1/2} \operatorname{div}_x W(Y,P)\,.
\end{align*}
Here \(\operatorname{div}_xW\deq\sum_{i\in[d]}\partial_{x_i}W_i\) is the divergence in the position argument; the \(\eta^{1/2}\) factor comes from the chain rule.
Conditioning on \((Y,P)\) proves
\eqref{eq:rhmc-proof-convolution-divergence-identity}.  Conditional expectation is an
\(L^q\) contraction.  Apply Jensen's inequality and \eqref{eq:rhmc-proof-gaussian-div-q} to obtain~\eqref{eq:rhmc-proof-convolution-divergence-bound}.
\end{proof}

\begin{lemma}[Stationary iterated derivatives in $L^q$]
\label{lem:rhmc-proof-iterated-derivative}
For \(q\ge2\) and \(k\ge0\),
\begin{equation}
 \norm{\LHam^kg_\eta}_{L^q(\Pi_\eta)}
 \le C^{k+1}q^k\,(k!)^{3/2}\,\sqrt{d+q}\,\eta^{-k/2}\,.
 \label{eq:rhmc-proof-iterated-derivative}
\end{equation}
\end{lemma}
\begin{proof}
For a tensor-valued \(F\), define the weighted Sobolev norm
\begin{equation}
 \norm F_{\mathsf W^{N,q}}
 \deq \sum_{j=0}^N\frac{\eta^{j/2}}{j!}\,
       \norm{D_x^jF}_{L^q(\Pi_\eta;\HS)}\,.
 \label{eq:rhmc-proof-weighted-Snr}
\end{equation}
Tracing through the proof of~\cref{lem:rhmc-proof-iterated-derivative-l2}, but using~\eqref{eq:rhmc-proof-convolution-divergence-bound} in place of~\eqref{eq:rhmc-proof-witten-l2}, we obtain
\begin{equation}
 \norm{\delta_{\pi_\eta}W}_{\mathsf W^{N-1,q}}
 \le\frac{Cq N}{\sqrt\eta}\,\norm W_{\mathsf W^{N,q}}\,.
 \label{eq:rhmc-proof-weighted-divergence}
\end{equation}
Recall the factorization $\LHam F=\delta_\gamma(D_xF)-\delta_{\pi_\eta}(D_pF)$ from \eqref{eq:Lham-factorization}. If \(F\) is polynomial of
degree at most \(r\) in \(p\), then
\cref{lem:rhmc-proof-gaussian-polynomial}, \eqref{eq:rhmc-proof-weighted-Snr}, and
\eqref{eq:rhmc-proof-weighted-divergence} yield the
dimension-free bound
\begin{equation}
 \norm{\LHam F}_{\mathsf W^{N-1,q}}
 \le\frac{CN\,(\sqrt q+q\sqrt r)}{\sqrt\eta}\,
       \norm F_{\mathsf W^{N,q}}\,.
 \label{eq:rhmc-proof-A-one-step}
\end{equation}

It remains to initialize the induction.
From~\cite[Lemma 6.2.7]{Chewi26Book}, we have \(\norm{g_\eta}_{L^q(\pi_\eta)}\le C\sqrt{d+q}\).  For \(j\ge1\), the uniform
bound \eqref{eq:rhmc-proof-derivative-full} and the weights in
\eqref{eq:rhmc-proof-weighted-Snr} give
\begin{equation*}
 \norm{g_\eta}_{\mathsf W^{N,q}}
 \le C(N+1)\sqrt{d+q}\,.
\end{equation*}
Finally, \(\LHam^jg_\eta\) is polynomial of degree at most \(j\) in
momentum.  Apply \eqref{eq:rhmc-proof-A-one-step} \(k\) times, starting at \(N=k\) and
using \(r=j\) at step \(j\).  Since
\(\sqrt q+q\sqrt j\le2q\sqrt{j+1}\) for \(q\ge2\), the product of the
one-step factors is bounded by
\[
    \prod_{j=0}^{k-1} [(k-j)\,(\sqrt q+q\sqrt j)]
 \le (2q)^k\,(k!)^{3/2}\,.
\]
For \(k\ge1\), the initialization factor satisfies \(k+1\le2^k\) and
can therefore be absorbed into \(C^{k+1}\); the case \(k=0\) follows
directly from the initialization bound.  This proves
\eqref{eq:rhmc-proof-iterated-derivative}.
\end{proof}

\begin{proposition}[Discretization error bound]
\label{prop:rhmc-proof-deterministic-defect}
If \(h^2\le1\), then, for every \(q\ge2\),
\begin{equation*}
 \Edisc[q]
 \le C\sqrt{d+q}\,\Bigl[
 B_{J,q}h\,\Bigl(\frac h{\sqrt\eta}\Bigr)^J
 +h^5
 \Bigr]\,.
\end{equation*}
Here, one may take
\(B_{J,q}=C_0^{J+1}q^J\sqrt{J!}\), where \(C_0\) is universal.
\end{proposition}
\begin{proof}
Repeat the proof of \cref{prop:rhmc-proof-deterministic-defect-l2} with
the uniform-in-time \(L^q\) norm in place of the uniform-in-time \(L^2\)
norm.  The interpolation estimate is replaced by
\begin{align}
 \mathcal E_{J,q}
 &\deq\sup_{0\le t\le h}\,
   \norm{g(t)-\mathcal I_{J,h}g(t)}_{L^q(\Pi_\eta)}
\le\frac{h^J}{J!}\sup_{0\le t\le h}\,
       \norm{g^{(J)}(t)}_{L^q(\Pi_\eta)}
 \le C^{J+1}q^J\sqrt{J!}\,\sqrt{d+q}\,
       \Bigl(\frac h{\sqrt\eta}\Bigr)^J\,.
 \label{eq:rhmc-proof-interpolation-error}
\end{align}
Also, stationarity and \cite[Lemma~6.2.7]{Chewi26Book} give
\[
 \sup_{0\le t\le h}\,\norm{g_\eta(X_t)}_{L^q(\Pi_\eta)}
 =\norm{g_\eta}_{L^q(\pi_\eta)}
 \le C\sqrt{d+q}
\]
in place of the corresponding \(L^2\) bound by \(\sqrt d\).
\end{proof}

\subsection{Proof of
  \texorpdfstring{\hyperref[thm:rhmc-module]{Theorem~\ref*{thm:rhmc-module}}}
                 {Theorem 5.1}}
\label{sec:rhmc-higher-moment-main-theorem}

\begin{proof}[Proof of \cref{thm:rhmc-module}]
We first assume exact proximal queries.  The modifications needed to remove
proximal access are deferred to \cref{app:gradient-only-prox}.

Fix \(q\ge2\) and \(0<\eps\le1\), and let \(\mathfrak L_q\) be as in
the theorem.  Choose
\begin{equation*}
  h
  \deq c\mathfrak L_q^{-3/2}
  \min\Bigl\{\frac1\kappa,
    \Bigl(\frac{\eps}{\kappa\sqrt{d+q}}\Bigr)^{1/3}\Bigr\}\,,
  \qquad
  \eta\deq C_\eta\mathfrak L_q^3h^2\,,
\end{equation*}
where \(C_\eta\ge1\) is universal and \(c > 0\) will be chosen
sufficiently small.  Set
\begin{equation}
 J\deq\lceil C\mathfrak L_q\rceil\,,
 \qquad
 N\deq\lceil C\kappa\mathfrak L_q/h\rceil\,.
 \label{eq:rhmc-proof-parameter-choices}
\end{equation}
With appropriately large $C$ and small $c$, the definitions
give
\begin{equation*}
  e^{-C\mathfrak L_q}\le h
  \le\frac{c}{\kappa\mathfrak L_q^{3/2}}\,,
  \qquad
  h\mathfrak L_q\le\frac{c}{\kappa\mathfrak L_q^{1/2}}
  \le\frac c\kappa\,,
  \qquad
  \eta\le\frac{C_\eta c^2}{\kappa^2}
  \le\min\{c_0,1/2\}\,,
\end{equation*}
and \(\log(h^{-1})\le C\mathfrak L_q\).
Since
\(B_{J,q}=C_0^{J+1}q^J\sqrt{J!}\), we have
\begin{equation*}
 B_{J,q}C_\eta^{-J/2}\mathfrak L_q^{-3J/2}
 \le e^{-cJ}\,.
\end{equation*}
A sufficiently large $C$ therefore gives
\begin{equation*}
 B_{J,q}\,\Bigl(\frac h{\sqrt\eta}\Bigr)^J
 \le h^4\,.
\end{equation*}
Tedious bookkeeping shows that
\begin{equation}
 W_{q,M_\kappa}(\widehat\mu_N,\Pi_\eta)
 \le C\mathfrak L_q^{9/2}\kappa\sqrt{d+q}\,h^3\,,
 \label{eq:rhmc-proof-exact-phase-bound}
\end{equation}
Each phase makes \(2J\) gradient queries and
\(2J\) exact proximal queries.  Since
\(J\le C\mathfrak L_q\) and
\(N\le C\kappa\mathfrak L_q/h+1
\le C'\kappa\mathfrak L_q/h\), the numbers of direct gradient queries
and exact proximal queries are each at most
\(C\kappa\mathfrak L_q^2/h\).  Taking position marginals in
\eqref{eq:rhmc-proof-exact-phase-bound} and
the definition of \(h\) gives
\begin{equation*}
  W_q(\widehat\pi_\eta,\pi_\eta)
  \le C\mathfrak L_q^{9/2}\kappa\sqrt{d+q}\,h^3
  \le\eps\,.
\end{equation*}
The modifications in \cref{app:gradient-only-prox} remove the proximal
queries altogether with expected gradient cost at most
\(C\kappa\mathfrak L_q^3/h\).  Substitution of \(h\) gives the query
bound in the theorem.
\end{proof}

\section{Recursive warm start generator}
\label{sec:recursive-warm-start}

In this section, we describe the recursive warm start generator and its guarantee.
See \cref{ssec:recursive_overview} for the algorithm and an overview.

\subsection{Main result}

For probability laws \(P,Q\) on \(\R^d\), define the sub-Gaussian
Orlicz--Wasserstein distance
\begin{equation*}
  W_{\psi_2}(P,Q)
  \deq
  \inf_{(X,Y)}\inf\Bigl\{
    \eps>0:
    \E\exp\Bigl(\frac{\norm{X-Y}^2}{\eps^2}\Bigr)\le2
  \Bigr\}\,,
\end{equation*}
where the outer infimum ranges over all couplings of \(P\) and \(Q\).

The recursive warm start generator assumes access to a smoothed sampler as follows.

\begin{assumption*}[Smoothed sampler]
There is a sampler \(\mathfrak A\) which, provided any normalized potential \(U\) satisfying $\kappa_U^{-1}\Id\preceq\Hess U\preceq\Id$,
and an admissible reference point $x_{\rm ref}$ with $\norm{\nabla U(x_{\rm ref})} \le \sqrt{d/\kappa_U}$, implements one of the interfaces below. Here,
\(\pi_{U,\eta}\deq\pi_U*\cN(0,\eta\Id)\) denotes the smoothed distribution.
\begin{enumerate}[
  label=\(\mathsf{A\arabic*}\),
  ref=\(\mathsf{A\arabic*}\),
  leftmargin=2.8em
]
\item\label{ass:smoothed-w2}
Given \(\eps>0\), the sampler returns
\(0<\eta\le c_0\) and a sample
with law \(\widehat\pi_{U,\eta}\) such that
\begin{equation*}
  \kappa_U^{-1/2}\,W_2(\widehat\pi_{U,\eta},\pi_{U,\eta})\le\eps
\end{equation*}
using at most \(\mathsf C_2(\kappa_U,d,\eps)\) expected gradient queries.
\item\label{ass:smoothed-wp}
Given \(p\ge2\) and \(\eps>0\), it returns \(0<\eta\le c_0\) and a
sample with law \(\widehat\pi_{U,\eta}\) such that
\begin{equation*}
  \kappa_U^{-1/2}\,W_p(\widehat\pi_{U,\eta},\pi_{U,\eta})\le\eps
\end{equation*}
using at most \(\mathsf C_p(\kappa_U,d,\eps)\) expected gradient queries.
\item\label{ass:smoothed-psi2}
Given \(\eps>0\), it returns \(0<\eta\le c_0\) and a sample with law
\(\widehat\pi_{U,\eta}\) such that
\begin{equation*}
  \kappa_U^{-1/2}\,W_{\psi_2}(\widehat\pi_{U,\eta},\pi_{U,\eta})\le\eps
\end{equation*}
using at most \(\mathsf C_{\psi_2}(\kappa_U,d,\eps)\) expected gradient
queries.
\end{enumerate}
Assume that each complexity function is non-decreasing in its condition number
argument and non-increasing in its accuracy argument, and that
\(p\mapsto\mathsf C_p\) is non-decreasing.
\end{assumption*}

We will apply this interface directly to potentials whose smoothness is not normalized to
$1$.  Concretely, if
\(\alpha_U\Id\preceq\Hess U\preceq\beta_U\Id\), apply \(\mathfrak A\) to
the normalized potential \( U(\cdot/\sqrt{\beta_U})\), and then divide
its output by \(\sqrt{\beta_U}\). If \(\eta\) is the normalized smoothing level, then the law of the rescaled output
approximates $\pi_U*\cN(0,\beta_U^{-1}\eta\Id)$.
The corresponding scale-invariant guarantee is
\(\sqrt{\alpha_U}\,W_\bullet\le\eps\), and the query cost is unchanged.

In \cref{sec:w2-proof,sec:higher-moment-analysis}, we have shown that Picard HMC implements the first two interfaces above.
We have not shown that it implements the third, but we include this case for completeness because, as shown below, \ref{ass:smoothed-psi2} leads to a genuine R\'enyi warm start.

We now present the main guarantee.

\begin{theorem}[Recursive warm start generator]
\label{thm:recursive-warm-generator}
Suppose that \(\pi\propto\exp(-V)\) is \(\beta\)-log-smooth, in the
sense that \(\grad V\) is \(\beta\)-Lipschitz, and satisfies the
logarithmic Sobolev inequality with constant $\alpha^{-1}$.
Write \(\kappa\deq\beta/\alpha\), and
suppose that we are given a point \(x_{\rm ref}\) satisfying $V(x_{\rm ref})-\inf V\le d$.

Let the smoothed sampler \(\mathfrak A\) satisfy one of
\ref{ass:smoothed-w2}, \ref{ass:smoothed-wp}, and
\ref{ass:smoothed-psi2}.  Fix \(0<\delta, \Delta\le 1/2\).  For
\(q\ge2\), write
\begin{equation*}
  \mathfrak L_q
  \deq q+\log\frac{d\kappa q}{\Delta}\,.
\end{equation*}
Then the proximal sampler wrapper in \cref{alg:outer-warm-overview}, with
each inner RGO sampled by \cref{alg:generic-warm-overview}, has the
following guarantees.
\begin{enumerate}[label=\textup{(\roman*)},leftmargin=2.4em]
\item Under \ref{ass:smoothed-w2}, the generator returns a
sample with law \(\widehat\pi\) satisfying
\begin{equation*}
  \KL(\widehat\pi\mmid\pi)\le\Delta^2\,.
\end{equation*}
The expected total number of gradient queries is at most
\begin{equation*}
  C\kappa\mathfrak L_2^2\,
  \mathsf C_2\Bigl(
    6,d,\frac{c\Delta}{\sqrt{\kappa\mathfrak L_2}}
  \Bigr)
  +C\kappa\mathfrak L_2^2\,.
\end{equation*}
\item Fix \( 2 \le q < \infty \). Under \ref{ass:smoothed-wp}, the generator
returns a sample with law \(\widehat\pi\) for which there exists a law
\(\widehat\pi^\dagger\) satisfying
\begin{equation*}
  \TV(\widehat\pi,\widehat\pi^\dagger)\le\delta\,,
  \qquad
  \Ren_q(\widehat\pi^\dagger\mmid\pi)\le\Delta^2\,.
\end{equation*}
The expected total number of gradient queries is at most
\begin{equation*}
  C\kappa\mathfrak L_q^2\,
  \mathsf C_p\Bigl(
    6,d,
    \frac{c\Delta}{
      \sqrt{q\kappa\log(e q)\,\mathfrak L_q}}
  \Bigr)
  +C\kappa\mathfrak L_q\,
    \Bigl(
      \mathfrak L_q+\log\frac{C\kappa\mathfrak L_q}{\delta}
    \Bigr)\,,
\end{equation*}
where
\begin{equation*}
  p\deq
  2\vee\Bigl\lceil
    \log\frac{C\kappa\mathfrak L_q^2}{\delta}
  \Bigr\rceil\,.
\end{equation*}
\item Fix \( 2\le q <\infty \).  Under \ref{ass:smoothed-psi2}, the generator returns
a sample with law \(\widehat\pi\) satisfying
\begin{equation*}
  \Ren_q(\widehat\pi\mmid\pi)\le\Delta^2\,.
\end{equation*}
The expected total number of gradient queries is at most
\begin{equation*}
  C\kappa\mathfrak L_q^2\,
  \mathsf C_{\psi_2}\Bigl(
    6,d,
    c\min\Bigl\{
      \frac1q,
      \frac{\Delta}{
        \sqrt{q\kappa\log(e q)\,\mathfrak L_q}}
    \Bigr\}
  \Bigr)
  +C\kappa\mathfrak L_q^2\,.
\end{equation*}
\end{enumerate}
\end{theorem}

Although the statement of~\cref{thm:recursive-warm-generator} is cumbersome, it is quite powerful: it states that once we have a smoothed sampler with Wasserstein guarantees, we can immediately obtain a stronger guarantee with a negligible overhead.
Indeed, the cost of the warm start generation is essentially the cost of the original smoothed sampler, called with a slightly smaller accuracy \(\Delta/\sqrt{q \kappa}\) up to logarithmic factors.
Crucially, in case~\textup{(ii)}, we only need the smoothed sampler
guarantee in \(W_p\) for an order \(p\) which is logarithmic in
\(1/\delta\).

Due to the use of the outer proximal sampler wrapper, note that the smoothed sampler is only ever called on distributions with condition number bounded by a universal constant; above we take it to be $6$. This will ultimately improve the condition number dependence of our combined sampler.

\subsection{Technical tools}
\label{sec:warm-technical-tools}

\subsubsection{Regularization and reverse transport inequalities}
\label{sec:channels}

The key to the recursive warm start is the regularizing property of the heat flow, which enables proving reverse transport inequalities. We collect together the key inequalities here.

We begin with a simple lemma on truncation.

\begin{lemma}[Truncation]
\label{lem:wq-truncation}
Let \(P,Q\) be probability laws satisfying
\(W_p(P,Q)\le\eps_p\) for some \(p\ge2\).  For every \(0<\delta<1\),
there is a probability law \(P^\dagger\) such that
\begin{equation*}
  \TV(P,P^\dagger)\le\delta\,,
  \qquad
  W_\infty(P^\dagger,Q)
  \le\eps_p\delta^{-1/p}\,.
\end{equation*}
\end{lemma}

\begin{proof}
Choose an optimal \(W_p\) coupling \((Y,Y^\star)\) of $P$ and $Q$.  Define
\begin{equation*}
  Y^\dagger
  \deq
  \begin{cases}
    Y\,,&\norm{Y-Y^\star}\le\eps_p\delta^{-1/p}\,,
    \\
    Y^\star\,,&\norm{Y-Y^\star}>\eps_p\delta^{-1/p}\,.
  \end{cases}
\end{equation*}
Markov's inequality shows that \(Y\ne Y^\dagger\) with probability at most
\(\delta\), whereas
\(\norm{Y^\dagger-Y^\star}\le\eps_p\delta^{-1/p}\) almost surely.  Taking
\(P^\dagger\deq\law(Y^\dagger)\) proves the result.
\end{proof}

For \(\tau>0\), let \(\gamma_\tau\deq\cN(0,\tau\Id)\).
The following lemma shows that after applying the heat flow, we can bound stronger information-theoretic divergences in terms of transport distances.

\begin{lemma}[Reverse transport along the heat flow]
\label{lem:gaussian-channel-reverse-transport}
Let \(P,Q\) be probability laws on \(\R^d\).
\begin{enumerate}[label=\textup{(\roman*)},leftmargin=2.4em]
\item If \(P,Q\in\mathcal P_2(\R^d)\), then
\begin{equation*}
  \KL(P*\gamma_\tau \mmid Q*\gamma_\tau)
  \le\frac{\W^2(P,Q)}{2\tau}\,.
\end{equation*}
\item If $P,Q \in \mathcal P_p(\R^d)$, then for every
\(0<\delta<1\), there is a probability law \(P^\dagger\) such that
\begin{equation*}
  \TV(P,P^\dagger)\le\delta\,,
  \qquad \Ren_q(P^\dagger*\gamma_\tau \mmid Q*\gamma_\tau)
  \le\frac{qW_p^2(P,Q)}{2\delta^{2/p}\tau}\,, \qquad\text{for all}~q > 1\,.
\end{equation*}
\item Let \(q\ge2\) and suppose that
\(W_{\psi_2}(P,Q)\le\eps_{\psi_2}\). Then,
\begin{equation*}
  \frac{q\,(q-1)\,\eps_{\psi_2}^2}{2\tau}\le1 \implies
  \Ren_q(P*\gamma_\tau \mmid Q*\gamma_\tau)
  \le \frac{q\log 2}{2\tau}\,\eps_{\psi_2}^2\,.
\end{equation*}
\end{enumerate}
\end{lemma}

\begin{proof}
    Statement (i) follows from the joint convexity of the KL divergence.

For (ii), choose \(P^\dagger\) as in \cref{lem:wq-truncation} and use the fact that
\begin{equation*}
    \Ren_q(P^\dagger*\gamma_\tau \mmid Q*\gamma_\tau)
    \le \frac{q}{2\tau}\,W_\infty^2(P^\dagger, Q)\,.
\end{equation*}

Part~\textup{(iii)} is \cite[Remark~3.8]{AC24}, with the regularization
parameter there specialized to zero, with zero initial shift,
the source displacement parameter equal to \(W_{\psi_2}(P,Q)\), and
\(\sigma^2=t\).
\end{proof}

\subsubsection{RGO calculus}
\label{sec:one-stage-warm}
\label{sec:geometry}

For $U \in C^2(\R^d)$, $A \in (0,\infty]$, and $u \in\R^d$, define the RGO potential
\begin{equation*}
  U_{A,u}\deq U+\frac{\norm{\cdot-u}^2}{2A}\,,
\end{equation*}
and let $R^U_{A,u} \propto \exp(-U_{A,u})$ be the corresponding RGO\@.

In the lemma below, we note that the RGO of an RGO is itself an RGO for the original distribution, and we record how the parameters change. The proof is an elementary calculation and is omitted.

\begin{lemma}[RGO calculus]
\label{lem:rgo-calculus}
Suppose that \(U\) satisfies
\(\kappa^{-1}\Id\preceq\Hess U\preceq\Id\).
Then, $U_{A,u}$ is $\alpha_A$-strongly convex and $\beta_A$-smooth, where
\begin{equation*}
  \alpha_A\deq\kappa^{-1}+A^{-1}\,,
  \qquad
  \beta_A\deq1+A^{-1}\,,
  \qquad
  \kappa_A\deq\frac{\beta_A}{\alpha_A}\,.
\end{equation*}
Moreover,
\begin{equation*}
  R^{U_{A,u}}_{a,y}=R^U_{A^+,u^+}\,,
\end{equation*}
where
\begin{equation}
  \begin{aligned}
    \frac1{A^+}&\deq\frac1A+\frac1a\,,
    \qquad u^+\deq A^+\,\Bigl(\frac uA+\frac ya\Bigr)\,,
    \qquad \kappa_{A^+}
    =\frac{a \beta_A+1}{a \beta_A+\kappa_A}\,\kappa_A\,.
  \end{aligned}
  \label{eq:family-update}
\end{equation}
\end{lemma}

\subsection{Recursive RGO sampler}
\label{sec:adaptive}

The following result gives the formal guarantee for
\cref{alg:generic-warm-overview}.
This is essentially the same as \cref{thm:recursive-warm-generator}, but without the outer proximal sampler wrapper.

\begin{theorem}[Recursive RGO sampler]
\label{thm:recursive-rgo-sampler}
Assume \eqref{eq:curvature}, and suppose that the sampler \(\mathfrak A\)
satisfies one of
\ref{ass:smoothed-w2}, \ref{ass:smoothed-wp}, or
\ref{ass:smoothed-psi2}.  Fix $A_0\in(0,\infty]$, $u_0\in\R^d$, and correspondingly  a target \(R^V_{A_0,u_0}\) for which an admissible reference point $x_{\rm ref}$ is
supplied, let \(\kappa_0\deq\kappa_{A_0}\), and fix
\(0<\delta_0,\Delta_0\le1/2\).  For \(q\ge2\), write
\begin{equation*}
  \mathfrak L_q
  \deq q+\log\frac{\kappa_0dq}{\Delta_0}\,.
\end{equation*}
Then
\cref{alg:generic-warm-overview} has the following guarantees.
\begin{enumerate}[label=\textup{(\roman*)},leftmargin=2.4em]
\item Under \ref{ass:smoothed-w2}, it returns a sample with law
\(\widehat R^V_{A_0,u_0}\) satisfying
\begin{equation*}
  \KL(\widehat R^V_{A_0,u_0}\mmid R^V_{A_0,u_0})
  \le\Delta_0^2\,.
\end{equation*}
The expected total number of gradient queries is at most
\begin{equation*}
  C\mathfrak L_2\,
  \mathsf C_2\Bigl(
    2\kappa_0,d,\frac{c\Delta_0}{\sqrt{\mathfrak L_2}}
  \Bigr)
  +C\kappa_0\mathfrak L_2\,.
\end{equation*}
\item Fix \(q\ge2\).
Under \ref{ass:smoothed-wp}, it returns a sample with law
\(\widehat R^V_{A_0,u_0}\) for which there exists
\((\widehat R^V_{A_0,u_0})^\dagger\) satisfying
\begin{equation*}
  \TV\bigl(\widehat R^V_{A_0,u_0},
    (\widehat R^V_{A_0,u_0})^\dagger\bigr)
  \le\delta_0\,,
  \qquad
  \Ren_q\bigl((\widehat R^V_{A_0,u_0})^\dagger
    \bigm\Vert R^V_{A_0,u_0}\bigr)
  \le\Delta_0^2\,.
\end{equation*}
The expected total number of gradient queries is at most
\begin{equation*}
  C\mathfrak L_q\,
  \mathsf C_p\Bigl(
    2\kappa_0,d,\frac{c\Delta_0}{\sqrt{q\mathfrak L_q}}
  \Bigr)
  +C\kappa_0\mathfrak L_q\,,
\end{equation*}
where $p \deq 2 \vee \lceil \log(C\mathfrak L_q/\delta_0)\rceil$.
\item Fix \(q\ge2\).  Under \ref{ass:smoothed-psi2}, it returns a sample
with law \(\widehat R^V_{A_0,u_0}\) satisfying
\begin{equation*}
  \Ren_q(\widehat R^V_{A_0,u_0}\mmid R^V_{A_0,u_0})
  \le\Delta_0^2\,.
\end{equation*}
The expected total number of gradient queries is at most
\begin{equation*}
  C\mathfrak L_q\,
  \mathsf C_{\psi_2}\Bigl(
    2\kappa_0,d,
    c\min\Bigl\{\frac{1}{q},
      \frac{\Delta_0}{\sqrt{q\mathfrak L_q}}
    \Bigr\}
  \Bigr)
  +C\kappa_0\mathfrak L_q\,.
\end{equation*}
\end{enumerate}
\end{theorem}

Although \cref{alg:generic-warm-overview} is written recursively, it is convenient to analyze it iteratively.
Thus, we break down the algorithm into stages, so that the target at stage $j$ is \(V_{A_j,u_j}\).
Let \(\eta_j\)
be the reported smoothing level at stage $j$, and denote the rescaled output by
\(X_j\); its law approximates $R^V_{A_j,u_j}*
  \cN(0,\beta_{A_j}^{-1}\eta_j\Id)$
in a Wasserstein metric.  Draw an independent
\(G_j\sim\cN(0,\Id)\) and set
\begin{equation*}
  Y_j\deq X_j+\sqrt{\frac{\tau_j}{\beta_{A_j}}}\,G_j\,,
  \qquad
  a_j\deq\frac{\eta_j+\tau_j}{\beta_{A_j}}\,.
\end{equation*}
Thus, $\law(Y_j)$ is close to $R_{A_j,u_j}^V * \cN(0,a_j I)$.
The algorithm then updates \((A_{j+1},u_{j+1})\) by
\eqref{eq:family-update} with RGO variance \(a_j\) and center \(Y_j\), as
in \cref{alg:generic-warm-overview}.  At the terminal level \(J\),
it samples the remaining target \(R^V_{A_J,u_J}\) with the FORS subroutine.

At stage \(j\), write \(\kappa_j\deq\kappa_{A_j}\) and choose the heat flow
duration
\begin{equation}
  \tau_j
  \deq
  \begin{cases}
    \kappa_j\,,&\kappa_j\ge2\,,
    \\
    c_0\,,&\kappa_j<2\,.
  \end{cases}
  \label{eq:shield-choice}
\end{equation}
\begin{lemma}[Geometric progress of the recursive stages]
\label{lem:recursive-geometric-progress}
Suppose that \(c_0\in(0,1/4)\) and
\(0<\eta_j\le c_0\) at every stage.
\begin{enumerate}[label=\textup{(\roman*)},leftmargin=2.4em]
\item If \(\kappa_j\ge2\), then
\begin{equation}
  \frac12\,\kappa_j
  \le\kappa_{j+1}
  \le\frac45\,\kappa_j\,.
  \label{eq:high-condition-reduction}
\end{equation}
\item If \(\kappa_j<2\), then \(A_{j+1}\le2c_0\), and all later
well-conditioned stages obey
\begin{equation*}
  A_{j+1}\le\rho_0A_j\,,
  \qquad
  \rho_0\deq\frac{2c_0}{1+2c_0}<1\,.
\end{equation*}
\end{enumerate}
\end{lemma}

\begin{proof}
In the first case, by
\eqref{eq:shield-choice}, $\kappa_j
  =\tau_j
  \le\eta_j + \tau_j
  \le\kappa_j+c_0$.
Substitution in
\eqref{eq:family-update} gives
\eqref{eq:high-condition-reduction} after decreasing \(c_0\) if necessary.

In the well-conditioned regime,
\(c_0\le\eta_j+\tau_j\le2c_0\).
By \eqref{eq:family-update}, the definition of \(a_j\), and
\(A_j\beta_{A_j}=A_j+1\),
\begin{equation*}
  A_{j+1}
  =\frac{A_ja_j}{A_j+a_j}
  =\frac{A_j\,(\eta_j + \tau_j)}
    {A_j\beta_{A_j}+ \eta_j + \tau_j}
  =\frac{A_j\,(\eta_j+\tau_j)}
    {A_j+1+\eta_j + \tau_j}\,.
\end{equation*}
This yields
\begin{align*}
  A_{j+1}
  &\le \tau_j+\eta_j
  \le2c_0\,,
  \qquad
  \frac{A_{j+1}}{A_j}
  =\frac{\eta_j+\tau_j}{A_j+\eta_j+\tau_j+1}
  \le\frac{2c_0}{1+2c_0}\,.\qedhere
\end{align*}
\end{proof}

\begin{proof}[Proof of \cref{thm:recursive-rgo-sampler}]
\noindent\textbf{Recursive error propagation.}
Here and below, take \(q=2\) under \ref{ass:smoothed-w2}.  For sufficiently small and
large universal constants \(c,C>0\), respectively, set
\begin{equation}
  \underline A
  \deq\frac{c}{\sqrt{d\mathfrak L_q}+\mathfrak L_q}\,,
  \qquad
  J
  \deq
  \Bigl\lceil
    C\log\frac{e\kappa_0}{\underline A}
  \Bigr\rceil
  \le C\mathfrak L_q\,.
  \label{eq:stage-count}
\end{equation}
At every stage
\(j\in\{0,\ldots,J-1\}\), call the smoothed sampler with accuracy
\begin{align}
  \eps_{2,j}^2
  &\deq \frac{2\tau_j\Delta_0^2}{\kappa_j\,(J+1)}
  &&\text{under \ref{ass:smoothed-w2}}\,,
  \label{eq:center-request-kl}
  \\
  \eps_{p,j}^2
  &\deq \frac{2e^{-2}\tau_j\Delta_0^2}{q\kappa_j\,(J+1)}\,,
  \qquad
  p\deq 2 \vee \Bigl\lceil \log \frac{C\mathfrak L_q}{\delta_0} \Bigr\rceil
  &&\text{under \ref{ass:smoothed-wp}}\,,
  \label{eq:center-request-proxy}
  \\
  \eps_{\psi_2,j}^2
  &\deq \frac{c\tau_j}{\kappa_j}\min\Bigl\{\frac{1}{q\,(q-1)},
    \frac{\Delta_0^2}{q\,(J+1)}
  \Bigr\}
  &&\text{under \ref{ass:smoothed-psi2}}\,.
  \label{eq:center-request-renyi}
\end{align}
Fix a stage \(j<J\) and condition on the complete history up to that stage.
Write \(\widehat P_j\) for the conditional law of \(Y_j\), and
write \(P_j\deq R^V_{A_j,u_j}*\cN(0,a_j\Id)\)
for its ideal law.
Also, let $K_j(y,\cdot) \deq R^{V_{A_j,u_j}}_{a_j,y}$,
and let \(\widehat K_j(y,\cdot)\) be the conditional law returned by the
remaining recursive call after the center \(y\) is selected.
Thus, $R_j \deq R^V_{A_j,u_j} = P_j K_j$, and the
output law from level \(j\) is $\widehat R_j = \widehat P_j \widehat K_j$.
We prove by backward induction that, conditional on every history,
\begin{align*}
  \KL(\widehat R_j\mmid R_j)
  &\le\frac{J-j+1}{J+1}\,\Delta_0^2
  &&\text{under \ref{ass:smoothed-w2}}\,,
  \\
  \Ren_q(\widehat R_j^\dagger\mmid R_j)
  &\le\frac{J-j+1}{J+1}\,\Delta_0^2\,,
  \qquad
  \TV(\widehat R_j,\widehat R_j^\dagger)
  \le\frac{J-j}{J}\,\delta_0
  &&\text{under \ref{ass:smoothed-wp}}\,,
  \\
  \Ren_q(\widehat R_j\mmid R_j)
  &\le\frac{J-j+1}{J+1}\,\Delta_0^2
  &&\text{under \ref{ass:smoothed-psi2}}\,.
\end{align*}
The terminal construction below establishes the base case \(j=J\), with
\(\widehat R_J^\dagger=\widehat R_J\) under
\ref{ass:smoothed-wp}.

Suppose that these inequalities are valid at stage $j+1$.  Under
\ref{ass:smoothed-w2}, after rescaling the output, the smoothed sampler provides the guarantee $W_2^2 \le\kappa_j\eps_{2,j}^2/\beta_{A_j}$.
After convolving with
\(\cN(0,\beta_{A_j}^{-1}\tau_j\Id)\), part~\textup{(i)} of
\cref{lem:gaussian-channel-reverse-transport} and
\eqref{eq:center-request-kl} yield
\begin{equation*}
  \KL(\widehat P_j\mmid P_j)
  \le
  \frac{\kappa_j\eps_{2,j}^2/\beta_{A_j}}
    {2\tau_j/\beta_{A_j}}
  =\frac{\kappa_j\eps_{2,j}^2}{2\tau_j}
  =\frac{\Delta_0^2}{J+1}\,.
\end{equation*}
Under \ref{ass:smoothed-wp}, the corresponding guarantee is
\(W_p^2\le\kappa_j\eps_{p,j}^2/\beta_{A_j}\).  Apply
part~\textup{(ii)} of the same lemma with TV budget \(\delta_0/J\).  Since
\(p\ge\log(J/\delta_0)\), we have
\((J/\delta_0)^{2/p}\le e^2\).  Hence there exists \(\widehat P_j^\dagger\) with
\begin{equation*}
  \TV(\widehat P_j,\widehat P_j^\dagger)\le\frac{\delta_0}{J}\,,
  \qquad
  \Ren_q(\widehat P_j^\dagger\mmid P_j)
  \le
  \frac{qe^2\kappa_j\eps_{p,j}^2/\beta_{A_j}}
    {2\tau_j/\beta_{A_j}}
  =\frac{qe^2\kappa_j\eps_{p,j}^2}{2\tau_j}
  =\frac{\Delta_0^2}{J+1}\,.
\end{equation*}
Under \ref{ass:smoothed-psi2}, the corresponding guarantee is
\(W_{\psi_2}^2
\le\kappa_j\eps_{\psi_2,j}^2/\beta_{A_j}\).
For a sufficiently small universal \(c\),
\eqref{eq:center-request-renyi} verifies both
\begin{equation*}
  \frac{q(q-1)\kappa_j\eps_{\psi_2,j}^2/\beta_{A_j}}
    {2\tau_j/\beta_{A_j}}
  =\frac{q(q-1)\kappa_j\eps_{\psi_2,j}^2}{2\tau_j}
  \le1
  \qquad\text{and}\qquad
  \Ren_q(\widehat P_j\mmid P_j)
  \le\frac{\Delta_0^2}{J+1}\,.
\end{equation*}
We now verify the induction step.  Under \ref{ass:smoothed-w2}, the KL
chain rule and data processing give
\begin{align*}
  \KL(\widehat R_j\mmid R_j)
  &\le
  \KL(\widehat P_j\mmid P_j)
  +\int\KL\bigl(
    \widehat K_j(y,\cdot)\bigm\Vert K_j(y,\cdot)
  \bigr)\,\widehat P_j(\dd y)
  \\
  &\le
  \frac{\Delta_0^2}{J+1}
  +\frac{J-j}{J+1}\,\Delta_0^2
  =\frac{J-j+1}{J+1}\,\Delta_0^2\,.
\end{align*}
Under \ref{ass:smoothed-psi2}, the same argument applies except that we replace the KL chain rule with the R\'enyi composition rule.
Finally, under \ref{ass:smoothed-wp}, we take $\widehat R_j^\dagger \deq \widehat P_j^\dagger \widehat K_j^\dagger$, and we obtain the bound on $\Ren_q(\widehat R_j^\dagger \mmid R_j)$ as before.
Also, we maximally couple \(\widehat P_j\) and
\(\widehat P_j^\dagger\), and on the event that their draws agree, we use the inductive
coupling of $\widehat K_j^\dagger$ and $\widehat K_j$.  This gives
\begin{align*}
  \TV(\widehat R_j,\widehat R_j^\dagger)
  &\le
  \TV(\widehat P_j,\widehat P_j^\dagger)
  +\sup_{y\in\R^d}
    \TV\bigl(
      \widehat K_j(y,\cdot),\widehat K_j^\dagger(y,\cdot)
  \bigr)
  \le
  \frac{\delta_0}{J}
  +\frac{J-j-1}{J}\,\delta_0
  =\frac{J-j}{J}\,\delta_0\,.
\end{align*}
This completes the induction.

\medskip
\noindent\textbf{Reference point construction.}
We verify that every smoothed sampler call is supplied with an admissible
reference point and account for the work needed to construct these points.
Let \(U\) be the current normalized potential and \(Y\) the output of the
smoothed sampler, and write \(\tau\) for the current heat flow duration.
Under \ref{ass:smoothed-w2} and \ref{ass:smoothed-wp}, use \(W_2\le W_p\).
Under \ref{ass:smoothed-psi2}, Jensen's inequality in the definition of
\(W_{\psi_2}\) gives \(W_2\le\sqrt{\log2}\,W_{\psi_2}\).  Thus, in every
case, we can couple $Y$ to the ideal
\(Y^\star=X+\sqrt\eta\, G'\), where \(X\sim\pi_U\).
This yields, by $1$-smoothness of $U$,
\begin{equation*}
  \E\norm{\grad U(Y)}^2
  \le 2\,\bigl(\E\norm{\grad U(Y^\star)}^2 + \kappa_U \varepsilon_j^2\bigr)
  \le C\,(d+\eta d+\kappa_U\eps_j^2)\,,
\end{equation*}
where \(\eps_j\) denotes the accuracy at stage $j$.

In normalized coordinates, the next RGO potential is
\begin{equation*}
  U^+
  \deq U+\frac{\norm{\cdot-(Y+\sqrt\tau\, G)}^2}{2\,(\eta+\tau)}\,.
\end{equation*}
Since
\(\kappa_U^{-1}\Id\preceq\Hess U\preceq\Id\), the potential \(U^+\) is
\(\alpha^+\)-strongly convex and \(\beta^+\)-smooth, where
\begin{equation*}
  \alpha^+\deq\kappa_U^{-1}+(\eta+\tau)^{-1}\,,
  \qquad
  \beta^+\deq1+(\eta+\tau)^{-1}\,.
\end{equation*}
Its gradient at \(Y\) adds only
\(-\sqrt\tau\, G/(\eta+\tau)\), and hence
\begin{equation*}
  \E\norm{\grad U^+(Y)}^2
  \le C\,\Bigl(
    d+\kappa_U\eps_j^2+\frac{\tau d}{(\eta+\tau)^2}
  \Bigr)\,.
\end{equation*}
Gradient descent with step size \(1/\beta^+\) produces iterates $\{x_k\}_{k\ge 0}$ with
\begin{equation*}
  \norm{\grad U^+(x_k)}^2
  \le \kappa^+\,\Bigl(1-\frac1{\kappa^+}\Bigr)^k\,
    \norm{\grad U^+(Y)}^2\,,
  \qquad
  \kappa^+\deq\frac{\beta^+}{\alpha^+}\,.
\end{equation*}
Stopping when \(\norm{\grad U^+(x_k)}\le\sqrt{\alpha^+d}\) gives the admissible reference point for the next stage.  The requested radii make the
initial squared gradient polynomial in the stage parameters, so Jensen's
inequality bounds the expected work at this stage by
\(\wtO(\kappa^+)\).  The condition numbers in the
large conditioning regime form a decreasing geometric sequence, and all later
stages have universally bounded condition number.  Summing over the stages
in \eqref{eq:stage-count} bounds the expected cost of all inner reference
points by \(C\kappa_0\mathfrak L_q\).

\medskip
\noindent\textbf{Terminal stage.}
By \cref{lem:recursive-geometric-progress}(i), the regime
\(\kappa_j\ge2\) lasts for at most \(C\log(e\kappa_0)\) stages.  Thereafter,
part~\textup{(ii)} of the same lemma contracts \(A_j\) by the fixed factor
\(\rho_0<1\).  Thus \eqref{eq:stage-count} ensures that
\(A_J\le\underline A\).  Set
\(\eps_J\deq\Delta_0/\sqrt{J+1}\); this is the accuracy requested from the
terminal FORS sampler.

Starting from the retained reference point, gradient descent on the terminal
potential produces \(x_+\) satisfying $\norm{\grad V_{A_J,u_J}(x_+)}\le\sqrt{d/A_J}$.
The same descent estimate used for the inner reference points shows that
this terminal construction costs at most \(C\kappa_0\mathfrak L_q\)
expected gradient queries. Moreover,
\begin{equation*}
  \norm{u_J-A_J\grad V(x_+)-x_+}
  =A_J\,\norm{\grad V_{A_J,u_J}(x_+)}
  \le\sqrt{dA_J}\,,
\end{equation*}
which is exactly the condition in \cref{thm:fors-implementation} with
target \(R^V_{A_J,u_J}\).

Under \ref{ass:smoothed-w2}, let \(\widehat R_J\) be the output of FORS at
order $2$.  Under \ref{ass:smoothed-wp}, let
\(\widehat R_J\) be the output of FORS at order \(q\) and set
\(\widehat R_J^\dagger=\widehat R_J\).  Under
\ref{ass:smoothed-psi2}, let \(\widehat R_J\) be the output of FORS at order
\(q\).  The choice of $\varepsilon_J$ verifies the localized-scale condition in \cref{thm:fors-implementation}.  Hence every terminal call
contributes at most \(\Delta_0^2/(J+1)\) in the relevant divergence and has
constant expected gradient query cost, establishing the base cases of the
backward induction.
\end{proof}

\subsection{Outer proximal sampler}
\label{sec:warm-iteration}
\label{sec:adaptive-gibbs}

From \cref{thm:w2-main,thm:rhmc-module}, the smoothed Picard HMC sampler requires at least $\kappa^2$ queries, and if we use it as our smoothed sampler then this cost is inherited by the recursive RGO sampler.
Instead, the full warm start generator runs an outer proximal chain with step size
\((2\beta)^{-1}\) and applies the recursive RGO sampler only to its uniformly
well-conditioned RGO targets.  This is the construction described in
\cref{alg:outer-warm-overview}.

Initialize
\(X_0\sim\widehat\mu_0\deq
\cN(x_{\rm ref},(2\beta)^{-1}\Id)\).
An outer step draws
\(Y_n=X_n+\cN(0,(2\beta)^{-1} I)\), and then applies an approximate RGO
kernel at \(Y_n\).

\begin{proof}[Proof of \cref{thm:recursive-warm-generator}]
Under \ref{ass:smoothed-w2}, choose
\begin{equation*}
  \Delta_0^2\deq\frac{c\Delta^2}{\kappa}\,,
  \qquad
  N\deq\Bigl\lceil
    C\kappa\log\frac{ed\kappa}{\Delta}
  \Bigr\rceil\,.
\end{equation*}
Under \ref{ass:smoothed-wp} and \ref{ass:smoothed-psi2}, keep the
prescribed order \(q\ge2\), and choose
\begin{equation*}
  \Delta_0^2\deq
  \frac{c\Delta^2}{\kappa\log(e q)}\,,
  \qquad
  N_0\deq\Bigl\lceil
    C\kappa\log\frac{ed\kappa}{\Delta}
  \Bigr\rceil\,,
  \qquad N_1\deq\Bigl\lceil
    \frac{\log(q-1)}{\log(1+(2\kappa)^{-1})}
  \Bigr\rceil\,,
  \qquad
  N\deq N_0+N_1\,.
\end{equation*}

\noindent\textbf{Initialization.}
Write \(V_\star\deq\inf V\) and \(g=\grad V(x_{\rm ref})\).  The
logarithmic Sobolev inequality implies the Poincar\'e inequality, and hence
\(\operatorname{Cov}_\pi\preceq\alpha^{-1}\Id\).  The Gaussian
maximum entropy bound, together with
\(-\int\log \pi\,\dd\pi
\ge V_\star+\log\int e^{-V}\),
therefore gives
\begin{equation*}
  \log\int_{\R^d}e^{-V}+V_\star
  \le\frac d2\log\frac{2\uppi e}{\alpha}\,.
\end{equation*}
On the other hand, \(\beta\)-smoothness yields
\begin{equation*}
  V(x)\le V(x_{\rm ref})
    +\ip{g}{x-x_{\rm ref}}
    +\frac\beta2\,\norm{x-x_{\rm ref}}^2
\end{equation*}
and by the descent lemma,
\begin{equation*}
  \frac{\norm{g}^2}{2\beta}
  \le V(x_{\rm ref})-V_\star
  \le d\,.
\end{equation*}
Comparing the density of
\(\widehat\mu_0=\cN(x_{\rm ref},(2\beta)^{-1}\Id)\) with that of
\(\pi\), and maximizing the resulting quadratic upper bound, gives the
pointwise estimate
\begin{equation*}
  \log\frac{\dd\widehat\mu_0}{\dd\pi}
  \le\frac d2\log(2e\kappa)+2d
  \le C d\log(e\kappa)\,.
\end{equation*}
Consequently,
\(\Ren_q(\widehat\mu_0\mmid\pi)\le C d\log(e\kappa)\) for every
\(q\ge1\), with the convention \(\Ren_1=\KL\).

\medskip
\noindent\textbf{Approximate RGO kernels.}
For every outer RGO target, apply
\cref{thm:recursive-rgo-sampler} with the local accuracy \(\Delta_0\)
chosen above.  Under
\ref{ass:smoothed-wp}, use \(\delta_0=\delta/(2N)\).  Every outer RGO
has curvature between \(\beta\Id\) and \(3\beta\Id\).  Apply
\cref{thm:recursive-rgo-sampler} to sample from the RGO, taking \(A_0=\infty\).  Thus, its three explicit bounds
apply with \(\kappa_0\le3\), uniformly over the outer center and complete
history.

\medskip
\noindent\textbf{Outer contraction and error propagation.}
Let \(\widehat\mu_n\) denote the law of the $n$-th iterate $X_n$ of the proximal sampler,
and, under \ref{ass:smoothed-wp}, let \(\mu_n^\dagger\)
denote the comparison law.  Write
\(\mu_n^\sharp=\widehat\mu_n\) under \ref{ass:smoothed-w2} and
\ref{ass:smoothed-psi2}, and \(\mu_n^\sharp=\mu_n^\dagger\) under
\ref{ass:smoothed-wp}.

Under \ref{ass:smoothed-w2}, the forward half of the proximal sampler
contraction in \cite[Theorem~3 and Appendix~A.4]{CCSW22} and the KL chain rule give
\begin{equation*}
  \KL(\widehat\mu_{n+1}\mmid\pi)
  \le (1+(2\kappa)^{-1})^{-1}
    \KL(\widehat\mu_n\mmid\pi)
  +\Delta_0^2\,.
\end{equation*}
Consequently,
\begin{equation*}
  \KL(\widehat\mu_N\mmid\pi)
  \le (1+(2\kappa)^{-1})^{-N}
    \,C d\log(e\kappa)
  +C\kappa\Delta_0^2
  \le\Delta^2\,.
\end{equation*}

It remains to treat \ref{ass:smoothed-wp} and
\ref{ass:smoothed-psi2}.  During the first \(N_0\) steps, use the
R\'enyi contraction from \cite[Theorem~3 and
Appendix~A.4]{CCSW22} and R\'enyi composition. This yields
\begin{equation*}
  \Ren_2(\mu_{N_0}^\sharp\mmid\pi)
  \le (1+(2\kappa)^{-1})^{-N_0/2}\,
    C d\log(e\kappa)
  +C\kappa\Delta_0^2
  \le\frac{\Delta^2}{2}\,.
\end{equation*}
For the remaining \(N_1\) steps, we leverage hypercontractivity to obtain a sharper dependence on $q$. Define \(q_{N_0}\deq 2\) and
\begin{equation*}
  q_{n+1}-1
  \deq\min\{
    q-1,(1+(2\kappa)^{-1})\,(q_n-1)
  \}\,,
  \qquad n=N_0,\ldots,N-1\,.
\end{equation*}
Thus, \(q_N=q\).  Since \(\pi\) has log-Sobolev constant at most
\(\alpha^{-1}\), hypercontractivity of simultaneous heat flow
\cite[Theorem~1.2 and Corollary~1.3]{KV26} gives
\begin{equation*}
  \Ren_{q_{n+1}}(
    \mu_n^\sharp*\gamma_{(2\beta)^{-1}}
    \mmid \pi*\gamma_{(2\beta)^{-1}}
  )
  \le
  \Ren_{q_n}(\mu_n^\sharp\mmid\pi)\,.
\end{equation*}
Data processing through the exact backward RGO, followed by
R\'enyi composition with the approximate RGO kernel, therefore yields
\begin{equation*}
  \Ren_{q_{n+1}}(
    \mu_{n+1}^\sharp\mmid\pi
  )
  \le
  \Ren_{q_n}(\mu_n^\sharp\mmid\pi)
  +\Delta_0^2\,.
\end{equation*}
Iterating and using \(N_1\le C\kappa\log(e q)\), we obtain
\begin{equation*}
  \Ren_q(\mu_N^\sharp\mmid\pi)
  \le
  \Ren_2(\mu_{N_0}^\sharp\mmid\pi)
  +N_1\Delta_0^2
  \le\Delta^2\,.
\end{equation*}
Under \ref{ass:smoothed-wp}, the TV errors also sum to at most \(\delta/2\).

\medskip
\noindent\textbf{Outer reference points.}
Under \ref{ass:smoothed-w2} and \ref{ass:smoothed-psi2}, the preceding
bounds on the KL divergence, followed by
Talagrand's inequality and smoothness, bounds the second moment of
\(\grad V(Y_n)\). If we run gradient descent starting from \(Y_n\),
Jensen's inequality implies that we can find an admissible reference point to apply \cref{thm:recursive-rgo-sampler} using logarithmic expected work per
outer step.

Under \ref{ass:smoothed-wp}, the implementable chain need not have
controlled moments of all orders, so cap the outer reference gradient descent after
\begin{equation*}
  N_{\rm cap}\deq\Bigl\lceil
    \log_2\frac{C\kappa N}{\delta}
  \Bigr\rceil
\end{equation*}
steps. Markov's inequality shows that the
sum of the probabilities that we hit $N_{\rm cap}$ iterations is at most \(\delta/2\).  Together with TV bounds from the preceding argument, this gives
\(\TV(\widehat\mu_N,\mu_N^\dagger)\le\delta\).  Set
\(\widehat\pi\deq\widehat\mu_N\), and, under \ref{ass:smoothed-wp}, set
\(\widehat\pi^\dagger\deq\mu_N^\dagger\).

\medskip
\noindent\textbf{Query cost.}
The choice above satisfies \(N\le C\kappa\mathfrak L_2\) under
\ref{ass:smoothed-w2}, and \(N\le C\kappa\mathfrak L_q\) under
\ref{ass:smoothed-wp} and \ref{ass:smoothed-psi2}.
Since the capped outer reference construction uses at most
\begin{equation*}
  C\kappa\mathfrak L_q\,
  \Bigl(
    \mathfrak L_q+\log\frac{C\kappa\mathfrak L_q}{\delta}
  \Bigr)
\end{equation*}
queries, and in the other two cases the cost is at most the final displayed term, \cref{thm:recursive-rgo-sampler} and careful bookkeeping proves the result.
\end{proof}

\section{Picard HMC warm starts and the high-accuracy sampler}
\label{sec:picard-warm-high-accuracy}

\subsection{Picard HMC warm starts}
\label{sec:picard-warm-start}

We first state a strengthened form of \cref{thm:adaptive-gibbs-warm}.  In
particular, strong log-concavity is replaced by log-smoothness and a logarithmic
Sobolev inequality, and we also state KL divergence guarantee.

\begin{theorem}[Picard HMC warm starts]
\label{thm:picard-warm-start}
Suppose that \(\pi\propto\exp(-V)\) is \(\beta\)-log-smooth and satisfies
the logarithmic Sobolev inequality with constant \(\alpha^{-1}\).  Write
\(\kappa\deq\beta/\alpha\), and suppose that we are given a point
\(x_{\rm ref}\) satisfying $V(x_{\rm ref})-\inf V\le d$.
For every \(0<\Delta,\delta\le1/2\), Picard HMC within the recursive warm start
generator satisfies the following guarantees.
\begin{enumerate}[label=\textup{(\roman*)},leftmargin=2.4em]
\item It returns a sample with law \(\widehat\pi\) satisfying
\begin{equation*}
  \KL(\widehat\pi\mmid\pi)\le\Delta^2\,.
\end{equation*}
The expected number of gradient queries is at most
\begin{equation*}
  C\,\bigl(
    \kappa^{7/6}d^{1/6}\Delta^{-1/3}+\kappa
  \bigr)
  \log^7\frac{e\kappa d}{\Delta}\,.
\end{equation*}
\item It returns a sample with law \(\widehat\pi_\delta\)
for which there exists
\(\widehat\pi_\delta^\dagger\) satisfying
\begin{equation*}
  \TV(\widehat\pi_\delta,\widehat\pi_\delta^\dagger)\le\delta\,,
  \qquad
  \Ren_2(\widehat\pi_\delta^\dagger\mmid\pi)\le\Delta^2\,.
\end{equation*}
The expected number of gradient queries is at most
\begin{equation*}
  C\,\bigl(
    \kappa^{7/6}d^{1/6}\Delta^{-1/3}+\kappa
  \bigr)
  \log^7\frac{e\kappa d}{\Delta\delta}\,.
\end{equation*}
\end{enumerate}
\end{theorem}
\begin{proof}
We verify that Picard HMC implements \ref{ass:smoothed-w2} and
\ref{ass:smoothed-wp}, and then invoke
\cref{thm:recursive-warm-generator}.  Fix a normalized potential \(U\) with
condition number \(\kappa_U\), let \(p\ge2\), and put $\bar\eps\deq\min\{\eps,\kappa_U^{-1/2}\}$.
Apply \cref{thm:rhmc-module} with accuracy
\(\sqrt{\kappa_U}\,\bar\eps\).  It returns \(0<\eta\le c_0\) and a law
\(\widehat\pi_{U,\eta}\) satisfying
\begin{equation*}
  \kappa_U^{-1/2}\,
  W_p(\widehat\pi_{U,\eta},\pi_{U,\eta})
  \le\bar\eps\le\eps\,.
\end{equation*}
Thus Picard HMC implements \ref{ass:smoothed-wp} with complexity
\begin{equation*}
  \mathsf C_p(\kappa_U,d,\eps)
  \le C\,\bigl(
    \kappa_U^2
    +\kappa_U^{7/6}\,(d+p)^{1/6}\,\bar\eps^{-1/3}
  \bigr)\,
  \Bigl(
    p+\log\frac{e\sqrt{\kappa_U}\,d}{\bar\eps}
  \Bigr)^{9/2}\,.
\end{equation*}
Taking \(p=2\) also implements \ref{ass:smoothed-w2}.
Careful bookkeeping shows that the costs are as displayed.
\end{proof}

Under \eqref{eq:curvature}, the reference assumption of
\cref{thm:adaptive-gibbs-warm} implies
\begin{equation*}
  V(x_{\rm ref})-\inf V
  \le\frac{\norm{\grad V(x_{\rm ref})}^2}{2\alpha}
  \le\frac d2\,.
\end{equation*}
Thus \cref{thm:adaptive-gibbs-warm} follows from
\cref{thm:picard-warm-start}(ii) by taking \(\Delta=1/2\).  Part~\textup{(i)}
also gives a genuine KL warm start under the more general assumptions of
\cref{thm:picard-warm-start}.

\subsection{From the proxy warm start to high accuracy}
\label{sec:picard-high-accuracy}

We now record the guarantee for proximal BPS\@.

\begin{theorem}[Proximal BPS]
\label{thm:bps-first-order}
Assume \(V\in C^2\).  Let \(\Delta\ge1\) and \(0<\tau<1/4\), and set
\begin{equation*}
 \mathfrak L_{\BPS}
 \deq \Delta^2+\log\frac{Cd\kappa}{\tau}\,.
\end{equation*}
Then, the proximal BPS Markov kernel
\(K_{\tau,\Delta}\) satisfies the following guarantee:
for every \(\lambda_0\) with
\(\Ren_2(\lambda_0 \mmid \pi)\le\Delta^2\), it holds that $\TV(\lambda_0K_{\tau,\Delta},\pi)\le\tau$.
Moreover, if \(N_{\grad V}\) is the total number of gradient queries made by this
kernel, then
\begin{equation}
  \E N_{\grad V}
  \le
  C\sqrt\kappa\,
  \bigl(d\mathfrak L_{\BPS}+\mathfrak L_{\BPS}^2\bigr)^{1/4}\,
  \mathfrak L_{\BPS}^{3/2}
  \log(\kappa\mathfrak L_{\BPS})\,.
  \label{eq:bps-gradient-cost}
\end{equation}
\end{theorem}

\begin{proof}
This is Algorithm~4.2, Theorem~4.3, and Corollary~4.4 of
\cite{PBPS26}.
\end{proof}

\begin{proof}[Proof of \cref{thm:main-synthesis}]
Apply \cref{thm:picard-warm-start}(ii) with
\(\Delta=1/2\) and \(\delta=\varepsilon/2\).  This produces an
implementable law \(\widehat\pi_{\varepsilon/2}\) and a comparison law
\(\widehat\pi_{\varepsilon/2}^\dagger\) such that
\begin{equation*}
  \TV(\widehat\pi_{\varepsilon/2},
      \widehat\pi_{\varepsilon/2}^\dagger)
  \le\frac\varepsilon2\,,
  \qquad
  \Ren_2(\widehat\pi_{\varepsilon/2}^\dagger\mmid\pi)
  \le\frac14\,.
\end{equation*}
Run the proximal BPS kernel
\(K_{\varepsilon/2,1}\) from the implementable law.  By
data processing,
\begin{equation*}
  \TV\bigl(
    \widehat\pi_{\varepsilon/2}K_{\varepsilon/2,1},
    \widehat\pi_{\varepsilon/2}^\dagger
      K_{\varepsilon/2,1}
  \bigr)
  \le\frac\varepsilon2\,.
\end{equation*}
Since $\widehat \pi_{\varepsilon/2}^\dagger$ satisfies the hypothesis of
\cref{thm:bps-first-order} with \(\Delta=1\), its image under this kernel is
within \(\varepsilon/2\) of \(\pi\) in total variation.  The triangle
inequality therefore gives the desired error \(\varepsilon\).

Put \(\mathfrak L\deq\log(e\kappa d/\varepsilon)\).  The warm start
cost is at most
\(C\kappa^{7/6}d^{1/6}\mathfrak L^7\), since \(d,\kappa\ge1\).
For the subsequent proximal BPS call, \(\Delta=1\) and
\(\tau=\varepsilon/2\), so \(\mathfrak L_{\BPS}\le C\mathfrak L\).
Thus \eqref{eq:bps-gradient-cost} gives the cost bound
\begin{equation*}
 C\sqrt\kappa\,
 (d^{1/4}\mathfrak L^{7/4}+\mathfrak L^2)
 \log(\kappa\mathfrak L)
 \le C\sqrt\kappa\,
 (d^{1/4}\mathfrak L^{11/4}+\mathfrak L^3)\,.
\end{equation*}
Here we used \(\log(\kappa\mathfrak L)\le C\mathfrak L\).
The remaining term \(\sqrt\kappa\,\mathfrak L^3\) is absorbed by
\(\kappa^{7/6}d^{1/6}\mathfrak L^7\), proving the stated query bound.
\end{proof}

\newpage
\appendix
\section{Terminal RGO implementation by FORS}
\label{sec:fors-terminal}

This appendix records the FORS subroutine that we use in the paper. After reviewing the generic FORS mechanism,
we describe a variant that is tailored to sampling from RGO distributions.
Since the main arguments here were already introduced in prior works, we will be more terse and only emphasize the differences from the analyses of
\cite{CCDR26,CCRZ26}.

\subsection{Review of first-order rejection sampling (FORS)}
\label{sec:fors-review}

For \(q\geq1\), write
\begin{equation*}
 \overline{\mathsf D}_q(P,Q)
 \deq
 \max\Bigl\{
  \int\Bigl(\frac{\dd P}{\dd Q}\Bigr)^q\,\dd Q-1\,,
  \int\Bigl(\frac{\dd Q}{\dd P}\Bigr)^q\,\dd P-1
 \Bigr\}\,.
\end{equation*}
Either term is understood to be infinite when the corresponding
absolute continuity condition fails.

Let \(\mathbf Q\) be a proposal law for a path \(\mathbf Z\), and suppose the
desired path law satisfies
\[
 \frac{\dd\mathbf P_\star}{\dd\mathbf Q}(\mathbf Z)
 \propto \exp\{w_\star(\mathbf Z)\}\,.
\]
Given an auxiliary variable \(\xi\sim\Xi\), FORS uses an estimator
\(W(\xi;\mathbf Z)\) with
\(w(\mathbf Z)\deq\E_{\xi\sim\Xi}W(\xi;\mathbf Z)\).  At clipping level
\(B>0\), its implemented path law \(\widehat{\mathbf P}\) satisfies
\[
 \frac{\dd\widehat{\mathbf P}}{\dd\mathbf Q}(\mathbf Z)
 \propto
 \exp\bigl\{\E_{\xi\sim\Xi}
  \operatorname{clip}_{[-B,B]}\bigl(W(\xi;\mathbf Z)\bigr)\bigr\}\,.
\]

\begin{lemma}[FORS likelihood comparison]
\label{lem:fors-likelihood-comparison}
Let \(P_\star\) and \(\widehat P\) be the laws of the terminal endpoint of
\(\mathbf Z\) under \(\mathbf P_\star\) and \(\widehat{\mathbf P}\),
respectively.  For every \(q\geq2\),
\begin{align*}
 2\log\{1+\overline{\mathsf D}_q
  (\widehat P,P_\star)\}
 \leq{}&
 \log\E_{\mathbf P_\star}
 \exp\bigl\{4q\,|w_\star(\mathbf Z)-w(\mathbf Z)|\bigr\}
 +\log\E_{\substack{\mathbf Z\sim\mathbf P_\star\\\xi\sim\Xi}}
  \exp\bigl\{4q\,
   \bigl(|W(\xi;\mathbf Z)|-B\bigr)_+\bigr\}\,.
\end{align*}
\end{lemma}

\begin{proof}
This is \cite[Lemma~2.4]{CCRZ26}.
\end{proof}

To draw a sample from the law $\widehat{\mathbf P}$, we use the following mechanism.

\begin{lemma}[Poisson exponentiation]
\label{lem:fors-poisson}
Let \(X\sim Q\).  Suppose that, conditionally on \(X=x\), one can simulate a
random variable \(W_x\in[-B,B]\) with
\(w(x)\deq\E[W_x\mid X=x]\).  Draw
\(N\sim\mathsf{Poisson}(2B)\) and, conditionally on \(X=x\),
independent copies \(W_x^{(1)},\ldots,W_x^{(N)}\).  Accept \(x\) with
probability
\begin{equation*}
 \prod_{k=1}^N\frac{B+W_x^{(k)}}{2B}\,.
\end{equation*}
Then the accepted proposal has law proportional to
\(\exp(w)\,Q\).  The acceptance probability of one attempt is
at least \(\exp(-2B)\).
\end{lemma}

\begin{proof}
See~\cite[Algorithm~1 and Theorem~2.3]{CCRZ26}.
\end{proof}

\subsection{FORS for RGO distributions}
\label{sec:fors-quadratic-core}

Given a convex potential \(V\), a variance \(a>0\), and a center \(y\), let
\begin{equation*}
 R_{a,y}
 \propto
 \exp\Bigl\{-V(\cdot)-\frac{\norm{\cdot-y}^2}{2a}\Bigr\}\,.
\end{equation*}
Here, we describe a variant of FORS diffusion simulation, specialized to distributions of this form.
Compared to \cite{CCRZ26}, we exactly integrate the
linear part of the drift in the underdamped Langevin proposal.
Namely, from \((x,p)\), the proposal is the path measure of
\begin{align}
 \dd X_t&=P_t\,\dd t\,,
 \notag\\
 \dd P_t&=-\Bigl\{\frac{X_t-y}{a}+\grad V(x)+\gamma P_t\Bigr\}\,\dd t
          +\sqrt{2\gamma}\,\dd B_t\,,
 \qquad c_0a^{-1/2}\leq\gamma\leq C_0a^{-1/2}\,.
 \label{eq:fors-core-proposal}
\end{align}
The full algorithm is described below.

\begin{algorithm}[FORS for RGO]
\label{alg:fors-quadratic-core}
Fix a block length \(T>0\), a number \(N\) of blocks, a mesh width
\(h_{\rm mesh}\) dividing \(T\), a clipping level \(B_{\rm clip}>0\), and a
stopping set \(\mathcal E\subseteq\R^d\times\R^d\).  Compute
\(x^+\deq\prox_{aV}(y)\).  On each attempt, initialize
\begin{equation*}
 X^0\sim\mathsf N\Bigl(x^+,\frac{a}{1+a\beta}\Id\Bigr)
 \qquad\text{and}\qquad
 P^0\sim\mathsf N(0,\Id)
\end{equation*}
independently.  Restart the attempt if \((X^0,P^0)\notin\mathcal E\).

For \(n=0,\ldots,N-1\), repeat the following proposal until it is accepted.
Starting from \(x=X^n\) and \(p=P^n\), draw
\(M\sim\mathsf{Poisson}(2B_{\rm clip})\) and independent
\(t_j\sim\mathsf{Unif}[0,T)\), \(j\in[M]\).  Jointly sample the linear
proposal \eqref{eq:fors-core-proposal}, its driving Brownian motion, and all
values needed below.  Let \(r_j\in[T/h_{\rm mesh}]\) be the unique index such
that \((r_j-1)h_{\rm mesh}\leq t_j<r_jh_{\rm mesh}\), and set
\begin{align*}
 \mu_t
 &\deq
 \frac{\grad V(X_t)-\grad V(x)}{\sqrt{2\gamma}}\,,
 \\
 -\mathcal W(t_j)
 &\deq
 \frac{T}{\sqrt{2\gamma}\,h_{\rm mesh}}
 \,\ip{B_T-B_{r_jh_{\rm mesh}}}
 {\grad V(X_{r_jh_{\rm mesh}})
  -\grad V(X_{(r_j-1)h_{\rm mesh}})}
 +\frac{T}{2}\,\norm{\mu_{t_j}}^2\,.
\end{align*}
Accept the proposed endpoint \((X_T,P_T)\) with probability
\begin{equation*}
 \prod_{j\in[M]}
 \frac{B_{\rm clip}
 +\operatorname{clip}_{[-B_{\rm clip},B_{\rm clip}]}
   (\mathcal W(t_j))}{2B_{\rm clip}}\,.
\end{equation*}
Upon acceptance, set \((X^{n+1},P^{n+1})\deq(X_T,P_T)\).  If the accepted
state lies outside \(\mathcal E\), restart from the initialization step.
After \(N\) accepted blocks, return \(X^N\).
\end{algorithm}

\begin{theorem}[FORS for RGO distributions]
\label{thm:fors-implementation}
Let \(V:\R^d\to\R\) be convex and \(\beta\)-smooth, and let
\(R_{a,y}\) be the law above.
Fix \(q\geq2\) and \(0<\varepsilon,p\leq1/2\).
\begin{enumerate}
    \item If \(0<a\beta\leq1\), there is a randomized routine using evaluations of
\(\grad V\) and one evaluation of \(\prox_{aV}(y)\) whose output law
\(\widehat R_{a,y}\) satisfies
\begin{equation}
 \sup_{y\in\R^d}
 \Ren_q(\widehat R_{a,y} \mmid R_{a,y})
 \leq\varepsilon^2\,.
 \label{eq:fors-implementation-renyi}
\end{equation}
Writing
\begin{equation*}
 \mathfrak L
 \deq
 q+\log\frac{Cd}{a\beta\varepsilon p}\,,
\end{equation*}
its expected number of queries and, with probability at least \(1-p\), its
realized number of queries are at most
\begin{equation}
 C\mathfrak L^{5/3}\,
 \bigl\{1+\bigl((a\beta)^2d\bigr)^{1/3}\bigr\}\,.
 \label{eq:fors-implementation-cost}
\end{equation}
\item
Suppose that a
point \(x_+\) is supplied such that
\begin{equation*}
 \norm{y-a\grad V(x_+)-x_+}\leq\sqrt{da}\,.
\end{equation*}
If
\begin{equation*}
 a^{-1}
 \geq
 C\beta\,\bigl[
  \sqrt{d\,(q+\log(1/\varepsilon))}
 +q+\log(1/\varepsilon)
 \bigr]\,,
\end{equation*}
then a gradient-only routine returns a law satisfying
\eqref{eq:fors-implementation-renyi}.  It uses \(O(1)\) expected evaluations
of \(\grad V\), and at most \(C\log(2/p)\) evaluations with probability at
least \(1-p\).
\end{enumerate}
\end{theorem}

\begin{proof}
The second statement is the smooth case of
\cite[Theorem~D.1]{CCDR26}. We therefore focus on the first statement, which is based on \cref{alg:fors-quadratic-core}.

By rescaling, it suffices to consider \(\beta=1\).
Let
\(x^+\deq\prox_{aV}(y)\), set \(x=x^++\sqrt a\,z\), and define
\begin{equation*}
 F(z)
 \deq
 V(x^++\sqrt a\,z)-V(x^+)
 -\sqrt a\,\ip{\grad V(x^+)}z\,.
\end{equation*}
After a change of variables, we consider the density proportional to
\begin{equation*}
 z\mapsto \exp\Bigl\{-\frac12\,\norm z^2-F(z)\Bigr\}\,,
 \qquad
 F(0)=0\,,
 \qquad
 \grad F(0)=0\,,
 \qquad
 0\preceq\Hess F\preceq a\Id\,.
\end{equation*}
Writing \(g\deq\grad F\) and
\(\bar\gamma\deq\sqrt a\,\gamma\), the corresponding space--time rescaling
transforms \eqref{eq:fors-core-proposal}, conditionally on the blockwise
initial state \((Z_0,P_0)=(z,p)\), into
\begin{align}
 \dd Z_t&=P_t\,\dd t\,,
 \notag\\
 \dd P_t&=-\{Z_t+g(z)+\bar\gamma P_t\}\,\dd t
          +\sqrt{2\bar\gamma}\,\dd B_t\,,
 \qquad c_0\leq\bar\gamma\leq C_0\,.
 \label{eq:fors-normalized-core-proposal}
\end{align}

We explain the only changes needed in the FORS proof of \cite{CCRZ26}.  Let
\(S\) solve
\[
 S''+\bar\gamma S'+S=0\,,
 \qquad S(0)=0\,,
 \qquad S'(0)=1\,,
 \qquad b_t\deq\int_0^tS_r\,\dd r\,.
\]
For bounded \(\bar\gamma\) and sufficiently small \(t\),
\(|S_t|\leq Ct\), \(|b_t|\leq Ct^2\), and
\(|S_t'|+|S_t''|\leq C\).  Under
\eqref{eq:fors-normalized-core-proposal},
\begin{equation}
 Z_t-z
 =S_t\,p-b_t\,\{z+g(z)\}
 +\sqrt{2\bar\gamma}\int_0^tS_{t-s}\,\dd B_s\,.
 \label{eq:fors-core-proposal-formula}
\end{equation}
Consequently the Girsanov control contains only the residual force,
\begin{equation*}
 \mu_t\deq\frac{g(Z_t)-g(z)}{\sqrt{2\bar\gamma}}\,.
\end{equation*}
This is the main improvement due to exact integration of the linear term.

Now put
\begin{equation*}
 \mathcal R(z,p)
 \deq\norm p^2+\norm z^2+a^{-1}\,\norm{g(z)}^2+d\,.
\end{equation*}
The calculations in
\cite[Lemmas~B.1 and B.3]{CCRZ26}, applied to
\eqref{eq:fors-core-proposal-formula}, gives the same estimate
as in that paper with its full smoothness replaced by the residual
smoothness \(a\).  In particular, the random likelihood factor has
sub-Gaussian scale
\begin{equation}
 \log\E\exp\{\lambda\,|\mathcal W_t|\}
 \leq
 \log 2+
 C\lambda^2\,\frac{a^2T^3}{\bar\gamma}\,\mathcal R(z,p)\,,
 \qquad 1\leq\lambda\leq\frac{c}{aT^2}\,.
 \label{eq:fors-core-factor-mgf}
\end{equation}
If the control is frozen on a mesh of width \(h_{\rm mesh}\), its
discretization error \(\mathcal E_{\rm mesh}\) satisfies
\begin{equation}
 \log\E\exp\{\lambda\,|\mathcal E_{\rm mesh}|\}
 \leq
 C\lambda\,\frac{ah_{\rm mesh}}{\bar\gamma}\,
 (d^{-1/2}+aT^2)\,\mathcal R(z,p)\,.
 \label{eq:fors-core-mesh-mgf}
\end{equation}
The proofs are the same as
\cite[Lemma~C.1 and the proof of Proposition~3.1]{CCRZ26}; the bounds on
\(S_t\) above replace the exponential Euler bounds there.

Applying \cref{lem:fors-likelihood-comparison} using
\cref{eq:fors-core-factor-mgf,eq:fors-core-mesh-mgf}, and then choosing the
clipping level and mesh exactly as in that proof, gives the required
error on one block.  The stopped composition and query tail arguments
are unchanged from \cite[Theorem~B.11]{CCRZ26}.  One $\mathcal W_t$ factor
uses \(g\) at no more than three times.  Since the linear process
\((Z,P,B)\) can be sampled by Gaussian regression at only those times and the
two adjacent mesh points, refining the mesh makes no additional
gradient queries.

Initialize from \(\mathsf N(0,(1+a)^{-1}\,\Id)\otimes \mathsf N(0,\Id)\).  The bounds
\[
 \frac12\,\norm z^2
 \leq\frac12\,\norm z^2+F(z)
 \leq\frac{1+a}{2}\,\norm z^2
\]
imply that the initial
$\Ren_{2q}$ divergence is at most \(ad/2\).  For the choices below,
the logarithmic factor in the stopped composition argument satisfies $2q+\log\frac{CNq}{\varepsilon^2} \leq C\mathfrak L$.
Take
\begin{equation*}
 T
 \deq
 c\mathfrak L^{-2/3}
 \min\{1,(a^2d)^{-1/3}\}
 \qquad\text{and}\qquad
 N\deq\bigl\lceil C\mathfrak L/T\bigr\rceil\,.
\end{equation*}
Indeed, the cubic restriction from
\cref{eq:fors-core-factor-mgf} is
\(T^3\leq c/(a^2d\mathfrak L^2)\), while the quadratic restriction is
\(T^2\leq c/(a\mathfrak L)\); the displayed choice satisfies both.
The mesh is then refined until
\cref{eq:fors-core-mesh-mgf} gives the required path accuracy, without
additional gradient queries.  Since the target is \(1\)-strongly
log-concave, a total time \(NT\geq C\mathfrak L\) suffices in the
LSI mixing and R\'enyi composition argument of
\cite[Theorem~3.2(ii)]{CCRZ26}, run at R\'enyi order \(2q\).  This yields
\eqref{eq:fors-implementation-renyi}.  Moreover,
\begin{equation*}
 N
 \leq
 C\mathfrak L^{5/3}\,
 \{1+(a^2d)^{1/3}\}\,.
\end{equation*}
The stopped implementation uses \(O(N)\) expected queries and
\(O(N+\log(1/p))\) queries with probability at least \(1-p\).  Since
\(\log(1/p)\leq\mathfrak L\), this proves
\eqref{eq:fors-implementation-cost}.
\end{proof}

\section{Chebyshev--Lobatto interpolation and quadrature}
\label{app:chebyshev-lobatto}

This appendix collects the interpolation facts used by the Picard
integrator.
\subsection{Nodes, interpolating polynomials, and integrated weights}

For an integer \(J\ge2\) indicating the total number of nodes, define the increasing Chebyshev--Lobatto nodes on
\([0,h]\) by
\begin{equation*}
    t_j\deq\frac h2\,\bigl\{1-\cos\bigl({\textstyle \frac{j-1}{J-1}}\,\uppi\bigr)\bigr\}\,,
 \qquad 1\le j\le J\,.
\end{equation*}
Their interpolating polynomials are
\begin{equation*}
 \ell_j(t)\deq
 \prod_{\substack{1\le k\le J\\k\ne j}}
 \frac{t-t_k}{t_j-t_k}\,,
 \qquad
 \ell_j(t_k)=\one_{\{j=k\}}\,.
\end{equation*}
Thus, the operator \(\mathcal I_{J,h}\) in
\eqref{eq:interpolation-operator} is the unique polynomial interpolant of
degree at most \(J-1\).  Its Lebesgue constant is defined as
\begin{equation}
 \Lambda_J
 \deq\sup_{0\le t\le h}\sum_{j=1}^{J}|\ell_j(t)|\,.
 \label{eq:chebyshev-lebesgue-constant}
\end{equation}
For a normed space \(E\) and \( \{f(t_j)\}_{j=1}^J \subseteq E\), the triangle
inequality immediately gives
\begin{equation*}
 \sup_{0\le t\le h}\,\norm{(\mathcal I_{J,h}f)(t)}_E
 \le\Lambda_J\max_{j\in[J]}\,\norm{f(t_j)}_E\,.
\end{equation*}

Recall the quadrature coefficients \(\omega_{i,j}\) and \(\omega_j\) from
\eqref{eq:rhmc-chebyshev-data}.  By their
definition, for every polynomial \(f\) of degree at most \(J-1\),
\begin{align*}
 \int_0^{t_i}(t_i-t)f(t)\,\dd t
 &=\sum_{j=1}^{J}\omega_{i,j}f(t_j)\,, \qquad \int_0^h f(t)\,\dd t
 =\sum_{j=1}^{J}\omega_jf(t_j)\,.
\end{align*}

\begin{proposition}[Chebyshev--Lobatto coefficient bounds]
\label{prop:chebyshev-lobatto-bounds}
For every \(J\ge2\),
\begin{equation}
 \begin{aligned}
     \max_{i\in[J]}\sum_{j\in[J]}|\omega_{i,j}|&\le h^2\,,
  &\omega_j&\ge0\,,
  &\sum_{j\in[J]}\omega_j&=h\,,
  &\Lambda_J&\le \frac2\uppi\log J+1\,.
 \end{aligned}
 \label{eq:rhmc-chebyshev-bounds}
\end{equation}
\end{proposition}
\begin{proof}
    The classical estimate for $\Lambda_J$ can be found in \cite[Theorem~15.2]{T19}.  The first estimate follows from \cref{lem:integrated-weights-l2,rem:integrated-weights-l1} below.
The \(\omega_j\)'s are the
Clenshaw--Curtis weights, which are non-negative.  Since the rule integrates
constants exactly on \([0,h]\), they sum to \(h\).
\end{proof}

\begin{lemma}[Square estimates]
\label{lem:integrated-weights-l2}
For every $J\geq2$,
\begin{equation*}
 \max_{i\in[J]}\sum_{j\in[J]}\omega_{i,j}^2\leq\frac{5\uppi^2h^4}{128\,(J-1)}\,,
 \qquad
 \sum_{j\in[J]}\omega_j^2\leq\frac{\uppi^2h^2}{8\,(J-1)}\,.
\end{equation*}
\end{lemma}
\begin{proof}
For $z=(z_j)_{j\in[J]}\in\R^J$, let
$f\deq\mathcal I_{J,h}z=\sum_{j\in[J]}z_j\ell_j$.
The discrete norm comparison in \cite[Lemma 1.3.1]{ShenTang06}, with
the Chebyshev--Lobatto weights in \cite[(1.3.11)]{ShenTang06}, gives
\begin{equation*}
 \int_0^h\frac{\lvert f(s)\rvert^2}{\sqrt{s\,(h-s)}}\,\dd s
 \leq\frac{\uppi}{J-1}\,\Bigl\{
       \frac{\lvert z_1\rvert^2+\lvert z_J\rvert^2}{2}
       +\sum_{j=2}^{J-1}\lvert z_j\rvert^2\Bigr\}
 \leq\frac{\uppi}{J-1}\,\norm z^2\,.
\end{equation*}
Indeed, the change of variables $x=1-2s/h$ sends $t_j$ to
$\cos(\frac{j-1}{J-1}\,\uppi)$ and the measure
$\dd s/\sqrt{s\,(h-s)}$ to $\dd x/\sqrt{1-x^2}$, with the limits reversed.
The two endpoint weights are $\frac{\uppi}{2\,(J-1)}$ and every interior
weight is $\frac{\uppi}{J-1}$.

For a continuous function $k:[0,h]\to\R$, put
$a_j(k)\deq\int_0^h k(s)\,\ell_j(s)\,\dd s$, $j\in[J]$.
Cauchy--Schwarz with the reciprocal weights
$1/\sqrt{s\,(h-s)}$ and $\sqrt{s\,(h-s)}$ yields
\begin{align*}
 \Bigl\lvert\sum_{j\in[J]}a_j(k)\,z_j\Bigr\rvert^2
 &=\Bigl\lvert\int_0^h f(s)\,k(s)\,\dd s\Bigr\rvert^2
 \leq\Bigl(\int_0^h\frac{\lvert f(s)\rvert^2}{\sqrt{s\,(h-s)}}\,\dd s\Bigr)\,
       \Bigl(\int_0^h\lvert k(s)\rvert^2\sqrt{s\,(h-s)}\,\dd s\Bigr)\\
 &\leq\frac{\uppi\,\norm z^2}{J-1}
       \int_0^h\lvert k(s)\rvert^2\sqrt{s\,(h-s)}\,\dd s\,.
\end{align*}
Taking the supremum over $\norm z=1$ gives
\begin{equation*}
 \sum_{j\in[J]}a_j(k)^2
 \leq\frac{\uppi}{J-1}
       \int_0^h\lvert k(s)\rvert^2\sqrt{s\,(h-s)}\,\dd s\,.
\end{equation*}
For $k=(t_i-\cdot)_+$, these coefficients are $a_j(k)=\omega_{i,j}$.
Since $0\leq(t_i-s)_+\leq h-s$ for $s \in [0,h]$,
\begin{equation*}
 \int_0^h(t_i-s)_+^2\sqrt{s\,(h-s)}\,\dd s
 \leq\int_0^h(h-s)^2\sqrt{s\,(h-s)}\,\dd s
 =\frac{5\uppi h^4}{128}\,.
\end{equation*}
This proves the first estimate, uniformly in $i\in[J]$.
For $k=1$, we have $a_j(k)=\omega_j$ and
\begin{equation*}
 \int_0^h\sqrt{s\,(h-s)}\,\dd s=\frac{\uppi h^2}{8}\,,
\end{equation*}
which proves the second estimate. \end{proof}

\begin{remark}[Refined weight estimate]
\label{rem:integrated-weights-l1}
By Cauchy--Schwarz and $J/(J-1)\leq2$, \cref{lem:integrated-weights-l2} gives
\begin{equation*}
 \max_{i\in[J]}\sum_{j\in[J]} |\omega_{i,j}|
 \leq\sqrt{\frac{5\uppi^2J}{128\,(J-1)}}\,h^2
 \leq\frac{\uppi\sqrt5}{8}\,h^2
 \leq h^2\,.
\end{equation*}
The same bound holds with $\omega_{i,j}$ replaced by
$\int_0^t(t-s)\ell_j(s)\,\dd s$, uniformly over $t\in[0,h]$:
the proof above only uses $0\le(t-s)_+\le h-s$.
\end{remark}

\subsection{Interpolation error bound}

We record the standard interpolation error bound.

\begin{lemma}[Banach-valued interpolation error bound]
\label{lem:chebyshev-banach-remainder}
Let \(E\) be a Banach space and let \(f\in C^J([0,h];E)\).  Then
\begin{equation*}
 \norm{f(t)-(\mathcal I_{J,h}f)(t)}_E
 \le
 \frac1{J!}\sup_{0\le s\le h}\,\norm{f^{(J)}(s)}_E
 \prod_{j=1}^{J}|t-t_j|
 \le
 \frac{h^J}{J!}\sup_{0\le s\le h}\,\norm{f^{(J)}(s)}_E\,.
\end{equation*}
\end{lemma}

This is the Banach-valued form of the Genocchi--Hermite interpolation
remainder; see \cite[Section~9]{DB05}. Apply the scalar argument to $\ell \circ f$, where $\ell \in E^*$ has unit norm. Taking the supremum over all such $\ell$ concludes the proof.

\section{Gaussian analysis}
\label{app:gaussian-analysis}

This appendix collects the Gaussian analysis used throughout
the paper.

\subsection{Basic inequalities and heat smoothing}

We recall the definition of a Wiener chaos. The $n$-th Wiener chaos is the subspace of $L^2(\gamma)$
consisting of polynomials of degree at most $n$ that are
orthogonal to every polynomial of degree less than $n$.
The zeroth chaos consists of constants.

\paragraph{Derivative and divergence inequalities.}
Fix $d\ge 1$ and let \(\gamma\) be standard Gaussian measure on \(\R^d\). We also let
\(\EuScript H\) be a finite-dimensional\footnote{We expect that the results can be extended to the infinite-dimensional setting, but this will not be needed here.} Hilbert space, and let
\(V=(V_i)_{i\in[d]}:\R^d\to\EuScript H^d\) be smooth. Its Gaussian
divergence is
\[
 \delta_\gamma V
 \deq\sum_{i\in[d]}(x_iV_i-\partial_iV_i)\,.
\]
It is the \(L^2(\gamma)\)-adjoint of the gradient:
\begin{equation}
 \E_\gamma \ip{DF}{V}_{\HS}
 =\E_\gamma \ip{F}{\delta_\gamma V}_{\EuScript H}\,.
 \label{eq:gaussian-divergence-adjoint}
\end{equation}
Here, $D = \nabla$ is the usual gradient.

Componentwise Gaussian Poincar\'e and two applications of
\eqref{eq:gaussian-divergence-adjoint} give
\begin{equation}
 \norm{F-\E F}_{L^2(\gamma)}
 \le\norm{DF}_{L^2(\gamma;\HS)}\,,
 \qquad
 \norm{\delta_\gamma V}_{L^2(\gamma)}^2
 \le\norm V_{L^2(\gamma;\HS)}^2
     +\norm{DV}_{L^2(\gamma;\HS)}^2\,.
 \label{eq:rhmc-proof-l2-gaussian-tools}
\end{equation}
Indeed, the exact identity behind the second inequality is
\begin{equation*}
 \norm{\delta_\gamma V}_{L^2(\gamma)}^2
 =\norm V_{L^2(\gamma;\HS)}^2
  +\sum_{i,j\in[d]}
    \E\ip{\partial_jV_i}{\partial_iV_j}_{\EuScript H}\,;
\end{equation*}
the last sum is at most \(\norm{DV}_{L^2(\gamma;\HS)}^2\) by
Cauchy--Schwarz.  If \(\gamma_t\deq\cN(0,t\Id)\), then
\[
 \delta_{\gamma_t}V
 \deq\sum_{i\in[d]}(t^{-1}x_iV_i-\partial_iV_i)\,,
\]
and the change of variables \(x=\sqrt t\,z\) gives
\begin{equation}
 \norm{\delta_{\gamma_t}V}_{L^2(\gamma_t)}^2
 \le t^{-1}\,\norm V_{L^2(\gamma_t;\HS)}^2
     +\norm{DV}_{L^2(\gamma_t;\HS)}^2\,.
 \label{eq:gaussian-divergence-covariance-t}
\end{equation}
For the $L^q$ analogues, we use the following inequalities.

\begin{lemma}[Gaussian Sobolev inequalities]
\label{lem:gaussian-sobolev-inequalities}
Let \(q\ge2\).  If
\(F\in W^{1,q}(\gamma;\EuScript H)\), then
\begin{equation*}
 \norm{F-\E F}_{L^q(\gamma;\EuScript H)}
 \le C\sqrt q\,\norm{DF}_{L^q(\gamma;\HS)}\,.
\end{equation*}
If \(V\in W^{1,q}(\gamma;\EuScript H^d)\) and
\(\delta_\gamma V\) is its row divergence, then
\begin{equation}
 \norm{\delta_\gamma V}_{L^q(\gamma;\EuScript H)}
 \le C\sqrt q\,\norm{\E V}_{\HS}
   +Cq\,\norm{DV}_{L^q(\gamma;\HS)}
 \,.
 \label{eq:rhmc-proof-gaussian-div-q}
\end{equation}
If we replace $\gamma$ by $\gamma_t$, the centered Gaussian measure of covariance \(t\Id\), then the first term on
the right side of \eqref{eq:rhmc-proof-gaussian-div-q} is multiplied by
\(t^{-1/2}\).
\end{lemma}
\begin{proof}
The first estimate follows from Pisier's Gaussian convexity inequality
\cite[Theorem~2.2]{P86}, applied with the convex function $u\mapsto \norm u_{\EuScript H}^2$.  The divergence estimate is
\cite[Corollary~1.4]{CCLZ26Divergence}, applied to the Hilbert target
$\EuScript H$.  The last sentence follows from rescaling.
\end{proof}

\paragraph{Dual Poincar\'e inequality.}
Let $\mathcal L_{\msf{OU}}\deq\Delta-\ip{x}{\grad}=-\delta_\gamma D$
be the Ornstein--Uhlenbeck generator.  On the $n$-th Wiener chaos,
$\mathcal L_{\msf{OU}}F=-nF$; its inverse below is taken on centered
functions and extended by zero on constants.

For a centered, finite-dimensional Hilbert-valued \(F\in L^2(\gamma;\EuScript H)\), define
\begin{equation*}
 \norm F_{\dot H^{-1}(\gamma;\EuScript H)}
 \deq
 \sup\bigl\{
   \E\ip FH_{\EuScript H}:
   \E H=0,\quad
   \norm{DH}_{L^2(\gamma;\HS)}\le1
 \bigr\}\,.
\end{equation*}
We then have
\begin{equation*}
 \norm F_{\dot H^{-1}(\gamma;\EuScript H)}
 =\norm{D(-\mathcal L_{\msf{OU}})^{-1}F}_{L^2(\gamma;\HS)}\,.
\end{equation*}
The following is the dual form of the Gaussian Poincar\'e inequality, see~\cite{K13}.

\begin{lemma}[Dual Gaussian Poincar\'e inequality]
\label{lem:rhmc-proof-gaussian-poisson}
For all such $F$,
\begin{equation*}
 \norm F_{\dot H^{-1}(\gamma;\EuScript H)}
 \le\norm F_{L^2(\gamma;\EuScript H)}\,.
\end{equation*}
\end{lemma}

\paragraph{Heat equation estimates.}
For later reference, let \((P_t)_{t\ge0}\) denote the Euclidean heat
semigroup,
\begin{equation*}
 (P_tF)(x)\deq\E F(x+\sqrt t\,G)\,.
\end{equation*}
Thus \(\partial_tP_tF=\frac12\Delta P_tF\).  If \(F\) is weakly
differentiable, differentiation commutes with the heat semigroup, whereas
Gaussian integration by parts gives the derivative formula
\begin{equation}
    D P_tF(x)=P_t(DF)(x)
 =\frac1{\sqrt t}\,\E[F(x+\sqrt t\,G)\otimes G]\,.
 \label{eq:gaussian-heat-derivative}
\end{equation}
In particular, Jensen and Bessel's inequality give
\begin{align}
 \norm{D P_tF(x)}_{\op}
 &\le\Lip(F)\,,\notag\\
 \norm{D P_tF(x)}_{\HS}^2
 &\le\frac1t\inf_{a\in\EuScript H}
       \E\norm{F(x+\sqrt t\,G)-a}_{\EuScript H}^2\,.
 \label{eq:gaussian-heat-bessel}
\end{align}
To justify the second estimate, fix \(a\in\EuScript H\) and put
\[
 Y\deq F(x+\sqrt t\,G)-a\,.
\]
The orthogonal projection of \(Y\in L^2(\gamma;\EuScript H)\) onto the
\(\EuScript H\)-valued first Gaussian chaos is
\[
 \mathcal P_1Y
 \deq\sum_{i\in[d]}G_i\,\E[G_iY]\,.
\]
Since the coordinate functions \((G_i)_{i\in[d]}\) are orthonormal in
\(L^2(\gamma)\), Bessel's inequality gives
\[
 \sum_{i\in[d]}\norm{\E[G_iY]}_{\EuScript H}^2
 =\E\norm{\mathcal P_1Y}_{\EuScript H}^2
 \le\E\norm Y_{\EuScript H}^2\,.
\]
Moreover, \(\E[a\otimes G]=0\), so
\eqref{eq:gaussian-heat-derivative} and the preceding display imply
\[
 t\,\norm{DP_tF(x)}_{\HS}^2
 =\sum_{i\in[d]}\norm{\E[G_iY]}_{\EuScript H}^2
 \le\E\norm{F(x+\sqrt t\,G)-a}_{\EuScript H}^2\,.
\]
Taking the infimum over \(a\in\EuScript H\) proves
\eqref{eq:gaussian-heat-bessel}.

\subsection{Cumulants and derivatives of smoothed potentials}

For random variables \(F_1,\ldots,F_k\) whose joint moment-generating
function is finite near the origin, define their joint cumulant by
\begin{equation*}
 \operatorname{Cum}(F_1,\ldots,F_k)
 \deq
 \biggl.
 \partial_{t_1}\cdots\partial_{t_k}
 \log\E\exp\Bigl\{\sum_{j=1}^k t_jF_j\Bigr\}
 \biggr|_{t_1=\cdots=t_k=0}\,.
\end{equation*}
If \(X\) is vector-valued, \(\operatorname{Cum}_k(X)\) denotes the
symmetric \(k\)-tensor obtained by applying this definition to linear
functionals of \(X\).

We record the connection with derivatives of the smoothed potential.  Let
\(X_y\sim R_{\eta,y}\), and set
\[
 Z_\eta(y)
 \deq
 \int\exp\Bigl\{-V(x)-\frac{\norm x^2}{2\eta}
                 +\frac{\ip yx}{\eta}\Bigr\}\,\dd x\,.
\]
Up to an additive constant,
\[
 V_\eta(y)=\frac{\norm y^2}{2\eta}-\log Z_\eta(y)\,.
\]
Since \(\log Z_\eta\) is the cumulant-generating function of the exponential
family \(R_{\eta,y}\), differentiation gives, for every \(k\ge2\),
\begin{equation}
 D^kV_\eta(y)
 =\one_{\{k=2\}}\,\eta^{-1}\Id
  -\eta^{-k}\operatorname{Cum}_k(X_y)\,.
 \label{eq:smoothed-potential-cumulant-identity}
\end{equation}

We need the following recursive representation of cumulants.

\begin{lemma}[Iterated Gaussian covariance]
\label{lem:rhmc-proof-iterated-covariance}
Let \(F_1,\ldots,F_k\) be smooth scalar functions of a standard Gaussian,
with bounded derivatives of every positive order. Define
\begin{align}
 \Gamma_1(F_1)&\deq F_1-\E F_1\,,\notag\\
 \Gamma_j(F_1,\ldots,F_j)
 &\deq\ip{DF_j}{
   D(-\mathcal L_{\msf{OU}})^{-1}
   \Gamma_{j-1}(F_1,\ldots,F_{j-1})}\,,
 \quad j\ge2\,.
 \label{eq:rhmc-proof-Gamma-recursion}
\end{align}
Then
\begin{equation*}
 \operatorname{Cum}(F_1,\ldots,F_k)
 =\sum_{\pi\in\mathfrak S_{k-1}}
   \E\Gamma_k(F_1,F_{\pi(2)},\ldots,F_{\pi(k)})\,,
\end{equation*}
where each permutation is over $\{2,\dotsc,k\}$.
\end{lemma}
\begin{proof}
This is a symmetrized form of \cite[Theorem~4.4]{NN11}.
\end{proof}

\begin{lemma}[Cumulants of a near-identity Gaussian image]
\label{lem:rhmc-proof-near-identity-cumulant}
Suppose \(T:\R^d\to\R^d\) is locally Lipschitz and, almost everywhere,
\[
 \norm{DT}_{\op}\le1\,,
 \qquad
 \norm{DT-\Id}_{\op}\le\varepsilon_T\,,
 \qquad 0\le\varepsilon_T\le1\,.
\]
If \(G\) is standard Gaussian, then, for every \(k\ge3\),
\begin{align}
 \sup_{\norm v=1}\,
 \norm{\operatorname{Cum}_k(T(G))[v,\cdot^{k-1}]}_{\HS}
 &\le2\,(k-1)!\,\varepsilon_T\,,\notag\\
 \norm{\operatorname{Cum}_k(T(G))}_{\HS}
 &\le2\,(k-1)!\,\varepsilon_T\sqrt d\,.
 \label{eq:rhmc-proof-near-cumulant-full}
\end{align}
Here, $A[v,\cdot^{k-1}]$ denotes the $(k-1)$-tensor obtained by fixing
the first argument of a $k$-tensor $A$ to $v$, leaving the remaining
$k-1$ arguments free.  In particular,
\[
 \norm{A[v,\cdot^{k-1}]}_{\HS}^2
 =\sum_{j_2,\ldots,j_k\in[d]}
   \bigl|A[v,e_{j_2},\ldots,e_{j_k}]\bigr|^2\,,
\]
where $(e_j)_{j\in[d]}$ is any orthonormal basis of $\R^d$.
\end{lemma}

\begin{proof}
First take \(T\) smooth.
For fixed \(v\), define the $(j-1)$-tensor
\(\Gamma_j(v)\) by
\[
 [\Gamma_j(v)]_{i_2,\ldots,i_j}
 \deq\Gamma_j\bigl(\ip v{T(G)},T_{i_2}(G),\ldots,T_{i_j}(G)\bigr)\,,
 \qquad i_2,\ldots,i_j\in[d]\,,
\]
using the scalar recursion \eqref{eq:rhmc-proof-Gamma-recursion}.
For $j=2$, its $i$-th entry is
$\ip{DT_i(G)}{D(-\mathcal L_{\msf{OU}})^{-1}(\ip v{T(G)})}$,
so collecting the entries gives
$\Gamma_2(v)=DT(G)\,D(-\mathcal L_{\msf{OU}})^{-1}(\ip v{T(G)})$.
Moreover, $\ip vG$ belongs to the first Wiener chaos, so
$D(-\mathcal L_{\msf{OU}})^{-1}(\ip vG)=v$.
Subtracting this linear contribution yields
\begin{align*}
    \norm{\Gamma_2(v)-v}_{L^2(\gamma)}
    &\le \norm{(DT-I)\,v}_{L^2(\gamma)} + \bigl\lVert DT\,\bigl(D(-\mathcal L_{\msf{OU}})^{-1} (\ip v{T})-v\bigr)\bigr\rVert_{L^2(\gamma)} \\
    &\le \varepsilon_T\,\norm v + \norm{DT\,D(-\mathcal L_{\msf {OU}})^{-1}(\ip v{T-\Id})}_{L^2(\gamma)} \\
    &\le \varepsilon_T\,\norm v + \norm{D(-\mathcal L_{\msf {OU}})^{-1}(\ip v{T-\Id})}_{L^2(\gamma)} \\
    &= \varepsilon_T\,\norm v + \norm{\langle v, T-\Id \rangle}_{\dot H^{-1}(\gamma)}
    \le \varepsilon_T\,\norm v + \norm{D\langle v, T-\Id\rangle}_{L^2(\gamma)}
    \le 2\varepsilon_T\,\norm v\,,
\end{align*}
where we applied \cref{lem:rhmc-proof-gaussian-poisson} and the Gaussian Poincar\'e inequality.
This implies
\begin{equation}
 \norm{\Gamma_2(v)-\E\Gamma_2(v)}_{L^2(\gamma;\HS)}
 \le2\varepsilon_T\,\norm v\,.
 \label{eq:rhmc-proof-Gamma-two-bound}
\end{equation}
For \(j\ge2\), collect the free indices into a tensor.  Then, $\Gamma_{j+1}$ is a
contraction of \(DT\) with
\(D(-\mathcal L_{\msf{OU}})^{-1}\Gamma_j\); hence, by~$\norm{DT}_{\op}\le1$ and \cref{lem:rhmc-proof-gaussian-poisson},
\[
 \norm{\Gamma_{j+1}(v)}_{L^2(\gamma;\HS)}
 \le\norm{D(-\mathcal L_{\msf{OU}})^{-1}
          \Gamma_j(v)}_{L^2(\gamma;\HS)} =
          \norm{\Gamma_j(v)}_{\dot H^{-1}(\gamma;\HS)}
 \le\norm{\Gamma_j(v)-\E\Gamma_j(v)}_{L^2(\gamma;\HS)}\,.
\]
Starting from \eqref{eq:rhmc-proof-Gamma-two-bound} thus gives
\(\norm{\Gamma_k(v)}_{L^2(\gamma;\HS)}\le2\varepsilon_T\,\norm v\) for every ordering.
Thus \cref{lem:rhmc-proof-iterated-covariance}, Minkowski, and Jensen imply
\[
 \norm{\operatorname{Cum}_k(T(G))[v,\cdot^{k-1}]}_{\HS}
 \le\sum_{\pi\in\mathfrak S_{k-1}}
      \norm{\E\Gamma_k^\pi(v)}_{\HS}
 \le2\,(k-1)!\,\varepsilon_T\,\norm v\,.
\]
Squaring and summing over an orthonormal basis proves
\eqref{eq:rhmc-proof-near-cumulant-full}.  A standard approximation argument extends the bound to locally Lipschitz \(T\).
\end{proof}

\subsection{Proofs of the second-order Wasserstein estimates}
\label{app:second-order-wasserstein-estimates}
\label{app:second-order-w2-estimate}

\begin{proof}[Proof of \cref{lem:rhmc-proof-random-map-l2}]
By approximation, it is enough to treat smooth \(f\).  Throughout the
proof, for functions of $(z,w)$, super- or sub-scripts \(1\) and \(2\) indicate operations in the first and second
arguments, respectively.  Thus, \(P_s^{(2)}\) is the heat semigroup in the
second argument, \(D_1\) and \(D_2\) are the corresponding derivatives,
\(\delta_1\) is Gaussian row divergence in the first argument under
\(\cN(0,\Id_m)\), and \(\delta_{2,t}\) is Gaussian row divergence in the
second argument under \(\cN(0,t\Id_n)\).

Let
\((W_t)_{0\le t\le1}\) be Brownian motion in \(\R^n\), independent of
\(Z\), and define
\[
 u_t(z,w)\deq P_{1-t}^{(2)}f(z,\cdot)(w)\,,
 \qquad X_t\deq Z+u_t(Z,W_t)\,.
\]
Thus, \(X_0=Z+\bar f(Z)\), while \(X_1\) has the law of \(Z+f(Z,G)\).
Writing \( (\mathcal F_t^W)_{0\le t \le 1}\) for the Brownian filtration, the Markov property
gives
\[
 u_t(z,W_t)=\E[f(z,W_1)\mid\mathcal F_t^W]\,, \qquad (0\le t\le1)\,.
\]
Consequently, \( (X_t)_{0\le t \le 1}\) is a martingale with respect to the enlarged
filtration \((\sigma(Z)\vee\mathcal F_t^W)_{0\le t\le1}\).  Since \(u\)
solves the backward heat equation with respect to \(w\),
It\^o's formula identifies its martingale representation as
\[
 \dd X_t=D_2u_t(Z,W_t)\,\dd W_t\,.
\]
Define
\[
 \Sigma_t(z,w)\deq D_2u_t(z,w)D_2u_t(z,w)^\T\,.
\]
A second application of It\^o's formula gives, for any smooth scalar test
function \(\varphi\),
\[
 \partial_t\,\E\varphi(X_t)
 =\frac12\E\ip{\Sigma_t(Z,W_t)}{D^2\varphi(X_t)}_{\HS}\,.
\]
To estimate the Wasserstein distance, we recast this as a continuity equation.  The chain rule gives
\begin{align*}
 D_1\grad\varphi(z+u_t(z,w))
 &=D^2\varphi(z+u_t(z,w))
   \,(\Id+D_1u_t(z,w))\,,\\
 D_2\grad\varphi(z+u_t(z,w))
 &=D^2\varphi(z+u_t(z,w))\,D_2u_t(z,w)\,.
\end{align*}
Since \(\Sigma_t\) and \(D^2\varphi\) are symmetric, it follows that
\begin{align*}
 &\ip{\Sigma_t(z,w)}
       {D^2\varphi(z+u_t(z,w))}\\
 &\qquad={}
 \ip{\Sigma_t(z,w)}
    {D_1\grad\varphi(z+u_t(z,w))}
 -\ip{D_1u_t(z,w)\,D_2u_t(z,w)}
    {D_2\grad\varphi(z+u_t(z,w))}\,.
\end{align*}
Here and in the sequel, all derivatives of \(u_t\) and \(\Sigma_t\)
are understood as evaluated at $(Z,W_t)$.
Applying the divergence adjoint identity
\eqref{eq:gaussian-divergence-adjoint} in each Gaussian variable gives
\[
 \E\ip{\Sigma_t}{D^2\varphi(X_t)}_{\HS}
 =\E\ip{
   \delta_1\Sigma_t-\delta_{2,t}(D_1u_tD_2u_t)}
  {\grad\varphi(X_t)}\,.
\]
It follows that
\[
 \partial_t\E\varphi(X_t)
 =\E\ip{V_t}{\grad\varphi(X_t)}\,,
 \qquad
 V_t\deq\frac12\,\{
   \delta_1\Sigma_t-\delta_{2,t}(D_1u_tD_2u_t)
 \}\,.
\]
Consequently, the evolution of \(\law(LX_t)\) satisfies the continuity equation
\[
 \partial_t\law(LX_t)
 +\operatorname{div}\bigl(v_t\law(LX_t)\bigr)=0\,,
\]
where the velocity field is given by
\[
 v_t(y)\deq\E[LV_t\mid LX_t=y]\,,
 \qquad
 \norm{v_t}_{L^2}\le\norm{LV_t}_{L^2}\,.
\]

It remains to estimate this explicit velocity.  The heat semigroup preserves
the first derivative bounds, while \eqref{eq:gaussian-heat-bessel} applied
to the weak derivatives of \(f\) gives the following bounds:
\begin{align*}
 \norm{D_1u_t}_{\op}&\le A\,,&
 \norm{D_2u_t}_{\op}&\le B\,,\\
 \norm{D_1D_2u_t}_{\HS} \vee \norm{D_2 D_1 u_t}_{\HS} &\le\frac{A\sqrt m}{\sqrt{1-t}}\,,&
 \norm{D_2^2u_t}_{\HS}
 &\le\frac{B\sqrt m}{\sqrt{1-t}}\,.
\end{align*}
Here, we used
\(\norm{D_1f}_{\HS}\le\sqrt m\,A\) and
\(\norm{D_2f}_{\HS}\le\sqrt m\,B\).  The product rule and
\(\norm{RS}_{\HS}\le\norm R_{\op}\,\norm S_{\HS}\) therefore give
\begin{align*}
 \norm{\Sigma_t}_{\HS}&\le B^2\sqrt m\,,&
 \norm{D_1\Sigma_t}_{\HS}
 &\le\frac{2AB\sqrt m}{\sqrt{1-t}}\,,\\
 \norm{D_1u_tD_2u_t}_{\HS}&\le AB\sqrt m\,,&
 \norm{D_2(D_1u_tD_2u_t)}_{\HS}
 &\le\frac{2AB\sqrt m}{\sqrt{1-t}}\,.
\end{align*}
Every column of \(D_2u_t\) lies in \(E\), so
\(\operatorname{range}(\Sigma_t)\subseteq E\).  Since \(E\) is a fixed linear subspace,
differentiating the columns of \(\Sigma_t\) in \(z\) preserves membership
in \(E\), and hence \(\delta_1\Sigma_t\in E\).

We now apply the Gaussian divergence bounds conditionally in each
variable. First,
\eqref{eq:rhmc-proof-l2-gaussian-tools} in the \(z\)-variable gives
\begin{align*}
 \norm{L\delta_1\Sigma_t}_{L^2}
 &=\norm{\delta_1(L\Sigma_t)}_{L^2}
 \le \norm{L\Sigma_t}_{L^2(\HS)}
      +\norm{LD_1\Sigma_t}_{L^2(\HS)}\\
 &\le \sqrt m\,\Bigl\{
       \norm{L\Pi_E}_{\op}\,B^2
       +2\,\norm L_{\op}\,\frac{AB}{\sqrt{1-t}}
      \Bigr\}\,.
\end{align*}
Here we used \(L\Sigma_t=L\Pi_E\Sigma_t\) in the first term.  Next,
the estimate
\eqref{eq:gaussian-divergence-covariance-t} in the \(w\)-variable gives
\begin{align*}
 \norm{L\delta_{2,t}(D_1u_tD_2u_t)}_{L^2}
 &={}
   \norm{\delta_{2,t}(LD_1u_tD_2u_t)}_{L^2}\\
 &\le
   t^{-1/2}\,\norm{LD_1u_tD_2u_t}_{L^2(\HS)}
   +\norm{LD_2(D_1u_tD_2u_t)}_{L^2(\HS)}\\
 &\le \sqrt m\,\norm L_{\op}\,AB
   \,\{t^{-1/2}+2\,(1-t)^{-1/2}\}\,.
\end{align*}
Combining these two estimates in the definition of \(V_t\) yields
\begin{equation*}
 \norm{LV_t}_{L^2}
 \le \sqrt m\,\Bigl\{
 \frac12\,\norm{L\Pi_E}_{\op}\,B^2
 +\frac12\,\norm L_{\op}\,AB\,t^{-1/2}
 +2\,\norm L_{\op}\,AB\,(1-t)^{-1/2}
 \Bigr\}\,.
\end{equation*}
The first singularity comes from rescaling Gaussian divergence under
\(\cN(0,t\Id_n)\), and the second from differentiating the heat extension.

Since $\int_0^1t^{-1/2}\,\dd t
 =\int_0^1(1-t)^{-1/2}\,\dd t=2$, we have
 \begin{align*}
     W_2(\law(LX_\varepsilon), \law(LX_{1-\varepsilon}))
     &\le \int_\varepsilon^{1-\varepsilon} \norm{v_t(LX_t)}_{L^2}\,\dd t
 \le \sqrt m\,\Bigl\{
  \frac12\,\norm{L\Pi_E}_{\op}\,B^2
  +5\,\norm L_{\op}\,AB
 \Bigr\}\,.
 \end{align*}
We can then let $\varepsilon\searrow 0$.  This proves
\eqref{eq:rhmc-proof-random-map-l2-result}.
\end{proof}

\begin{proof}[Proof of \cref{lem:rhmc-proof-random-map}]
Use the notation and repeat the proof of
\cref{lem:rhmc-proof-random-map-l2} through the construction of
\(u_t,X_t,\Sigma_t,V_t\), and \(v_t\), including all the displayed
pointwise derivative estimates.  Conditional expectation is now used as
an \(L^q\) contraction.  Applying
\eqref{eq:rhmc-proof-gaussian-div-q} conditionally in the \(z\)-variable,
and retaining separately its mean and derivative terms, gives
\[
 \norm{L\delta_1\Sigma_t}_{L^q}
 \le C\sqrt m\,\bigl\{
   \sqrt q\,\norm{L\Pi_E}_{\op}\,B^2
   +q\,\norm L_{\op}\,AB\,(1-t)^{-1/2}
 \bigr\}\,.
\]
Its rescaled version in the \(w\)-variable similarly gives
\[
 \norm{L\delta_{2,t}(D_1u_tD_2u_t)}_{L^q}
 \le C\sqrt m\,\norm L_{\op}\,AB\,
 \bigl\{\sqrt q\,t^{-1/2}
          +q\,(1-t)^{-1/2}\bigr\}\,.
\]
Integrating these bounds proves
\eqref{eq:rhmc-proof-random-map-result}.
\end{proof}

\subsection{Gaussian polynomial bounds}

For this subsection, \(D\) denotes the gradient in \(p\in\R^d\), and the
Gaussian divergence is the operator defined above with \(x\) replaced by
\(p\).

\begin{lemma}[Gaussian polynomial estimates]
\label{lem:gaussian-polynomial-estimates}
Let \(r\ge0\) be an integer.  If \(F:\R^d\to\EuScript H\) is a polynomial of total
degree at most \(r\), then
\begin{equation*}
 \norm{DF}_{L^2(\gamma;\HS)}
 \le\sqrt r\,\norm F_{L^2(\gamma)}\,.
\end{equation*}
If \(V:\R^d\to \EuScript H^d\) is a polynomial vector field of total degree
at most \(r\), then
\begin{equation*}
 \norm{\delta_\gamma V}_{L^2(\gamma)}
 \le\sqrt{r+1}\,\norm V_{L^2(\gamma;\HS)}\,.
\end{equation*}
\end{lemma}

These estimates are standard; see \cite[Propositions~1.2.2 and~1.3.1]{N06}, restricting to Wiener chaoses of degree at
most \(r\).
We use the following consequences of the sharp Gaussian divergence and
Bernstein--Markov inequalities in \cite{CCLZ26Divergence}.

\begin{lemma}[$L^q$ Gaussian polynomial estimates]
\label{lem:rhmc-proof-gaussian-polynomial}
Let $q\ge2$ and let $r\ge0$ be an integer.  If
$F:\R^d\to\EuScript H$ is a polynomial of total degree at most $r$, then
\begin{equation*}
 \norm{DF}_{L^q(\gamma;\HS)}
 \le C\sqrt r\,\norm{F-\E F}_{L^q(\gamma;\EuScript H)}\,.
\end{equation*}
If $V:\R^d\to\EuScript H^d$ is a polynomial vector field of total degree
at most $r$, then
\begin{equation*}
 \norm{\delta_\gamma V}_{L^q(\gamma;\EuScript H)}
 \le C\sqrt q\,\norm{\E V}_{\HS}
      +Cq\sqrt r\,\norm{V-\E V}_{L^q(\gamma;\HS)}
 \le C\,(\sqrt q+q\sqrt r)\,\norm V_{L^q(\gamma;\HS)}\,.
\end{equation*}
\end{lemma}
\begin{proof}
The scalar Bernstein--Markov bound is
\cite[Theorem~1.2]{CCLZ26Divergence}.  Its Hilbert-valued consequence is
obtained by applying it to $\ip{G_F}{F-\E F}$, where $G_F$ is a standard
Gaussian in the finite-dimensional span $\EuScript H_F$ of the
coefficients of $F$.  Jensen's inequality bounds the $L^q$ norm of $DF$
by the joint $L^q$ norm of $D\ip{G_F}{F-\E F}$; the Gaussian moment
factor on the right cancels the scalar theorem's denominator.
The divergence estimate then follows from
\cite[Corollary~1.4]{CCLZ26Divergence}, exactly as in the proof of
\cite[Corollary~1.7]{CCLZ26Divergence}, now with Hilbert target
$\EuScript H$.  For $r=0$, the derivative vanishes and the divergence
is a centered Gaussian with $L^q$ norm at most
$C\sqrt q\,\norm{\E V}_{\HS}$.
\end{proof}

\section{Replacing proximal queries by gradient queries}
\label{app:gradient-only-prox}

\subsection{Approximate proximal queries}
\label{app:gradient-only-proof-modifications}

For \(0<\eta<1\) and \(\eps_{\mathsf{prox}}\ge0\), an
\(\eps_{\mathsf{prox}}\)-accurate proximal query at \(y\) returns a
measurable point \(\widetilde x_y\) satisfying
\begin{equation}
 \norm{\widetilde x_y-\prox_{\eta V}(y)}
 \le\eps_{\mathsf{prox}}\,.
 \label{eq:rhmc-prox-tolerance}
\end{equation}
Given such a query and \(G\sim\cN(0,\Id)\), define
\begin{equation*}
 \widehat g_{\eta,\eps_{\mathsf{prox}}}(y;G)
 \deq\grad V(\widetilde x_y+\sqrt\eta\,G)\,.
\end{equation*}
The approximate proximal variant of \cref{alg:rhmc-prefix} replaces every
occurrence of \(\widehat g_\eta\) by
\(\widehat g_{\eta,\eps_{\mathsf{prox}}}\).

The proof of \cref{thm:rhmc-module} in the main text assumes exact
proximal queries.  We now record the modifications needed when these calls
are replaced by the approximate proximal queries from
\eqref{eq:rhmc-prox-tolerance}.

For every queried center \(y\), let \(\widetilde x_y\) be the output of an
abstract \(\eps_{\mathsf{prox}}\)-accurate proximal oracle, and let
\(\widehat K^{\eps_{\mathsf{prox}}}\) denote the resulting kernel.

\begin{lemma}[Stability under approximate proximal centers]
\label{lem:rhmc-proof-prox-stability}
If \(h^2\le1\), then
\begin{equation}
 W_q(\delta_z\widehat K^{\eps_{\mathsf{prox}}}\,,
     \delta_z\widehat K)
 \le \mathcal E_{\rm prox}\deq
 Ch\eps_{\mathsf{prox}}\,,
 \qquad q\ge1\,.
 \label{eq:rhmc-proof-prox-local-error}
\end{equation}
Consequently, for $J\ge 2$ and $h\le c/(\kappa q)$, and every phase-space law $\widehat \mu$,
\begin{align}
    W_{q,M_\kappa}^2(\widehat\mu(\widehat K^{\varepsilon_{\mathsf{prox}}})^N,\Pi_\eta)
 &\le e^{-cNh/\kappa}\,W_{q,M_\kappa}^2(\widehat\mu,\Pi_\eta)
   +\frac{C\kappa q}{h}\,\Evar[q]^2
 +\frac{C\kappa^2}{h^2}\,
       (\Ebias^2+\Edisc[q]^2+\mathcal E_{\rm prox}^2)\,.
 \label{eq:rhmc-proof-global-recurrence-inexact}
\end{align}
\end{lemma}
\begin{proof}
For the same Gaussian \(G\) in the stochastic gradient, \(1\)-smoothness and
\eqref{eq:rhmc-prox-tolerance} imply
\[
 \norm{\grad V(\widetilde x_y+\sqrt\eta G)
       -\grad V(x_y^++\sqrt\eta G)}
 \le\eps_{\mathsf{prox}}\,.
\]
The first stochastic gradient block therefore moves every Picard node by at most
\(h^2\eps_{\mathsf{prox}}\).  The error in the second block
stochastic gradient is at most
\(\eps_{\mathsf{prox}}+h^2\eps_{\mathsf{prox}}\).  This yields position error at most
\(Ch^2\eps_{\mathsf{prox}}\) and momentum error at most
\(Ch\eps_{\mathsf{prox}}\), proving
\eqref{eq:rhmc-proof-prox-local-error}.
Then, \eqref{eq:rhmc-proof-global-recurrence-inexact} by
applying the triangle inequality and Young's inequality in the one-step recurrence
\eqref{eq:rhmc-proof-one-step-recurrence}.
\end{proof}

For the gradient-only implementation, set
\begin{equation}
 \eps_{\mathsf{prox}}
 \deq c\mathfrak L_q^{-4}\sqrt d\,h^3\,.
 \label{eq:rhmc-finite-parameters}
\end{equation}
Since
\(\mathcal E_{\rm prox}\le Ch\eps_{\mathsf{prox}}\), we have
\[
 \frac\kappa h\, \mathcal E_{\rm prox}
 \le C\kappa\eps_{\mathsf{prox}}
 \le C\mathfrak L_q^{9/2}\kappa\sqrt{d+q}\,h^3\,.
\]
Thus \eqref{eq:rhmc-proof-global-recurrence-inexact} gives the same
conclusion as the exact proximal proof of \cref{thm:rhmc-module}.

\subsection{Gradient-only implementation}
\label{app:gradient-only-prox-implementation}

The Picard HMC algorithm in \cref{alg:rhmc-prefix} is stated using exact
proximal queries, and \cref{sec:rhmc-higher-moment-statement} introduces its
\(\eps_{\mathsf{prox}}\)-accurate variant.  This subsection implements every
such approximate query using only evaluations of the gradient.

Fix a potential \(V\) satisfying
\(\kappa^{-1}\Id\preceq\Hess V\preceq\Id\), a smoothing variance
\(0<\eta\le c_0<1\), and a center \(y\).  Write $x_y^+\deq\prox_{\eta V}(y)$.

\begin{algorithm}[Gradient descent]
\label{alg:rhmc-proof-prox}
Given \(y\in\R^d\) and \(\eps_{\mathsf{prox}}>0\), set \(x^{(0)}\deq y\) and
iterate
\begin{equation*}
 x^{(k+1)}\deq y-\eta\grad V(x^{(k)})
\end{equation*}
until
\begin{equation*}
 \norm{x^{(k+1)}-x^{(k)}}
 \le(1-\eta)\,\eps_{\mathsf{prox}}\,.
\end{equation*}
Return \(\widetilde x_y\deq x^{(k)}\).
\end{algorithm}

\begin{lemma}[Certificate and pointwise work]
\label{lem:rhmc-proof-prox-work}
The iteration in \cref{alg:rhmc-proof-prox} stops after finitely many steps
and satisfies
\begin{equation}
 \norm{\widetilde x_y-x_y^+}\le\eps_{\mathsf{prox}}\,.
 \label{eq:rhmc-proof-prox-certified}
\end{equation}
Its number \(N_{\mathsf{prox}}(y)\) of gradient evaluations obeys
\begin{equation}
 N_{\mathsf{prox}}(y)\le C_{c_0}\,\Bigl[1+
 \log\Bigl(1+\frac{\norm{\grad V(y)}}{\eps_{\mathsf{prox}}}\Bigr)\Bigr]\,.
 \label{eq:rhmc-proof-prox-work-pointwise}
\end{equation}
\end{lemma}

\begin{proof}
    The map \(T_y\deq y-\eta\grad V\) is an \(\eta\)-contraction, and its
unique fixed point is \(x_y^+\).  Therefore
\[
 \norm{x^{(k)}-x_y^+}
 \le\sum_{j=k}^{\infty}\,\norm{x^{(j+1)}-x^{(j)}}
 \le\frac{\norm{x^{(k+1)}-x^{(k)}}}{1-\eta}\,,
\]
which proves the certificate is valid.  Moreover, $\norm{x^{(k+1)}-x^{(k)}}
 \le\eta^k\,\norm{x^{(1)}-x^{(0)}}
 =\eta^{k+1}\,\norm{\grad V(y)}$.
Using \(\eta\le c_0<1\)
gives \eqref{eq:rhmc-proof-prox-work-pointwise}.
\end{proof}

The pointwise bound depends on the queried center.  To obtain the expected
complexity claimed by the gradient-only sampler, we record that the centers
encountered along the realized run have controlled second moments.

\begin{lemma}[Expected work along the Picard HMC run]
\label{lem:rhmc-proof-expected-prox-work}
Under the parameter choices in
\eqref{eq:rhmc-proof-parameter-choices} and
\eqref{eq:rhmc-finite-parameters}, assume
also that \(h\mathfrak L_q\le c/\kappa\) and \(\eta\le c_0\), and implement every
\(\eps_{\mathsf{prox}}\)-accurate proximal query in
\cref{alg:rhmc-prefix} using
\cref{alg:rhmc-proof-prox}.  Every argument \(Y\) supplied to this routine
obeys
\begin{equation}
 \E\norm{\grad V(Y)}^2\le C\kappa d\,\mathfrak L_q^{18}\,.
 \label{eq:rhmc-proof-queried-gradient-moment}
\end{equation}
Consequently the expected total number of gradient evaluations is at most
\begin{equation}
 \E N_{\grad V}\le C\,\frac\kappa h\,\mathfrak L_q^3\,.
 \label{eq:rhmc-proof-inexact-total-work}
\end{equation}
\end{lemma}

\begin{proof}
Let \(x_\star\) minimize \(V\).  By
\eqref{eq:rhmc-proof-global-recurrence-inexact} with \(q=2\) and
\cref{lem:rhmc-proof-initial-radius},
\[
 \sup_{j\le N}\E\bigl[\norm{X_j-x_\star}^2+\norm{P_j}^2\bigr]
 \le C\kappa d\,\mathfrak L_q^{18}\,.
\]
This also controls every layer \(1\) node.  Since
\(x_{x_\star}^+=x_\star\), non-expansiveness and
\eqref{eq:rhmc-proof-prox-certified} give
\(\norm{\widetilde x_Y-x_\star}\le
\norm{Y-x_\star}+\eps_{\mathsf{prox}}\); hence \(1\)-smoothness and Gaussian
second moments control the gradients used to form layer \(2\).  The bound
\(\max_{i\in[J]}\sum_{j\in[J]}|\omega_{i,j}|\le h^2\) then controls
every layer \(2\) node.  Finally,
\(\norm{\grad V(Y)}\le\norm{Y-x_\star}\) at either layer, proving
\eqref{eq:rhmc-proof-queried-gradient-moment}.

By \cref{lem:rhmc-proof-prox-work} and Jensen's inequality,
\[
 \E N_{\mathsf{prox}}(Y)
 \le C\,\Bigl[1+
 \log\Bigl(
   1+\frac{(\E\norm{\grad V(Y)}^2)^{1/2}}
           {\eps_{\mathsf{prox}}}
 \Bigr)\Bigr]
 \le C\mathfrak L_q\,.
\]
There are \(2J\) proximal and \(2J\) direct gradient calls per phase.
Using \(J\le C\mathfrak L_q\) and
\(N\le C\kappa\mathfrak L_q/h\) proves
\eqref{eq:rhmc-proof-inexact-total-work}.
\end{proof}

\begin{proof}[Proof of \cref{thm:rhmc-module}]
Replace every proximal query in \cref{alg:rhmc-prefix} by
\cref{alg:rhmc-proof-prox} at the tolerance
\eqref{eq:rhmc-finite-parameters}.
The certificate in \cref{lem:rhmc-proof-prox-work} makes every call a valid
\(\eps_{\mathsf{prox}}\)-accurate proximal query.  Then, \cref{app:gradient-only-proof-modifications} gives the
claimed \(W_q\) bound, and the query bound follows from
\cref{lem:rhmc-proof-expected-prox-work} after substituting the step size
chosen in the main text proof.
\end{proof}

\section{Extension to arbitrary Picard depth}
\label{app:depth-K-picard}

The algorithm used in the main results takes Picard depth $2$.  In this appendix, we record the discretization error of Picard HMC using the exact smoothed gradient, at arbitrary depth.

Retain the Chebyshev--Lobatto nodes and coefficients from
\eqref{eq:rhmc-chebyshev-data}; their properties are summarized in
\cref{app:chebyshev-lobatto}.  Let \(g_\eta\deq\grad V_\eta\), let
\(\Pi_\eta\deq\pi_\eta\otimes\cN(0,\Id)\), set
\(\Omega\deq(\omega_{i,j})_{i,j=1}^J\). Recall that
\begin{equation*}
 \max_{i\in[J]}\sum_{j\in[J]}\lvert\omega_{i,j}\rvert
 \le h^2\,.
\end{equation*}

\begin{algorithm}[Depth-\(K\) exact Picard HMC phase]
\label{alg:depth-K-reference}
Given \((X_{\rm init},P_{\rm init})\), draw independent
\(\zeta_0,\zeta_1\sim\cN(0,\Id)\), set
\(a_h\deq e^{-h/2}\), \(\sigma_h\deq(1-e^{-h})^{1/2}\), and define
\begin{align}
 P_0&\deq a_hP_{\rm init}+\sigma_h\zeta_0\,,
 &X_{t_i}^{[0]}&\deq X_{\rm init}+t_iP_0\,,
 \notag\\
 X_{t_i}^{[\ell+1]}
 &\deq X_{\rm init}+t_iP_0
   -\sum_{j\in[J]}\omega_{i,j}g_\eta(X_{t_j}^{[\ell]})\,,
 &&0\le\ell<K-1\,.
 \label{eq:depth-K-iteration}
\end{align}
Set
\begin{align*}
 P_h^{[K]}
 &\deq P_0-\sum_{j\in[J]}\omega_jg_\eta(X_{t_j}^{[K-1]})\,,
 &X_h^{[K]}
 &\deq X_{\rm init}+hP_0
   -\sum_{j\in[J]}\omega_{J,j}g_\eta(X_{t_j}^{[K-1]})\,,
\end{align*}
and write the resulting next-phase initialization as
\begin{align*}
 X_{\rm init}^+&\deq X_h^{[K]}\,,
 &P_{\rm init}^+&\deq a_hP_h^{[K]}+\sigma_h\zeta_1\,.
\end{align*}
For \(z=(X_{\rm init},P_{\rm init})\) and
\(\zeta=(\zeta_0,\zeta_1)\), write
\(\Phi_z^{[J,K]}(\zeta) \deq (X_{\rm init}^+, P_{\rm init}^+) \) and
\(K_{\eta,h}^{[J,K]}\) for its kernel.
\end{algorithm}

\begin{theorem}[Depth-\(K\) Picard HMC]
\label{thm:depth-K-reference}
Suppose \(\kappa\ge1\), \(0<\eta\le1\), \(V\in C^2(\R^d)\), and
\(\kappa^{-1}\Id\preceq\Hess V\preceq\Id\).  There are
universal \(c,C>0\) such that, if
\begin{equation}
 h^2\le\frac14\,,
 \qquad
 h\le\frac c\kappa\,,
 \label{eq:depth-K-restrictions}
\end{equation}
then, for every integer \(K\ge1\), every \(q\ge2\), every
\(z,z'\in\R^{2d}\), and every \(\zeta\in\R^{2d}\),
\begin{align}
 W_{q,M_\kappa}\bigl(
   \delta_zK_{\eta,h}^{[J,K]}\,,
   \delta_{z'}K_{\eta,h}^{[J,K]}
 \bigr)
 &\le
 (1-ch/\kappa)\,\norm{z-z'}_{M_\kappa}\,,
 \label{eq:depth-K-contraction}\\
 \norm{
   \{\Phi_z^{[J,K]}(\zeta)-z\}
   -\{\Phi_{z'}^{[J,K]}(\zeta)-z'\}
 }_{M_\kappa}
 &\le Ch\,\norm{z-z'}_{M_\kappa}\,,
 \label{eq:depth-K-increment-coupling}\\
 W_{q,M_\kappa}\bigl(
   \Pi_\eta K_{\eta,h}^{[J,K]}\,,
 \Pi_\eta
 \bigr)
 &\le
 C\sqrt{d+q}\,
 \Bigl\{
  B_{J,q} h\,\Bigl(\frac h{\sqrt\eta}\Bigr)^J
  +h^{2K+1}
 \Bigr\}\,,
 \label{eq:depth-K-defect}
\end{align}
where $B_{J,q}\deq C_0^{J+1}q^J\sqrt{J!}$, and \(C_0\) is universal.
In particular, \(K=2\) recovers \cref{prop:rhmc-proof-deterministic-defect}.
\end{theorem}
\begin{proof}
\textbf{Contraction and coupling.}
As in the proof of \cref{prop:rhmc-proof-contraction}, couple inputs
\(z=(x_{\rm init},p_{\rm init})\) and
\(z'=(x'_{\rm init},p'_{\rm init})\) using the same Gaussian variables, and
write $\Delta_0\deq z-z'=(\Delta x,\Delta p)$, \(\Delta_1\deq
\Phi_z^{[J,K]}(\zeta)-\Phi_{z'}^{[J,K]}(\zeta)\).
Then,
\begin{align*}
 \Delta X_{t_i}^{[0]}&=\Delta x+a_ht_i\Delta p\,,
 \qquad \Delta X_{t_i}^{[\ell+1]}
 =\Delta X_{t_i}^{[0]}-
   \sum_{j\in[J]}\omega_{i,j}H_j^{[\ell]}
   \Delta X_{t_j}^{[\ell]}\,,
\end{align*}
where
\[
 H_j^{[\ell]}
 \deq H\bigl(X_{t_j}^{[\ell]},X_{t_j}^{\prime[\ell]}\bigr)\,,
 \qquad
 (2\kappa)^{-1}\Id
 \preceq H_j^{[\ell]}\preceq\Id\,.
\]
Since $\max_{i\in[J]}\sum_{j\in[J]}\lvert\omega_{i,j}\rvert \le h^2\le 1/4$,
induction in \(\ell\) gives
\[
 \max_{j\in[J]}\,
 \norm{\Delta X_{t_j}^{[K-1]}-\Delta X_{t_j}^{[0]}}
 \le Ch^2\,\norm{\Delta_0}\,.
\]
Now set
\[
 H_j\deq H\bigl(X_{t_j}^{[K-1]},X_{t_j}^{\prime[K-1]}\bigr)\,,
 \qquad
 H\deq h^{-1}\sum_{j\in[J]}\omega_jH_j\,.
\]
The endpoint calculation in
\cref{lem:rhmc-proof-difference-expansion} therefore applies unchanged and
yields
\[
 \Delta_1=(\Id+hA_H)\Delta_0+\mathcal R_h\Delta_0\,,
 \qquad
 \norm{\mathcal R_h}_{\op}\le Ch^2\,.
\]
The matrix inequality in the proof of
\cref{prop:rhmc-proof-contraction}, followed by
\(h\le c/\kappa\), proves
\eqref{eq:depth-K-contraction}.
Moreover, subtracting \(\Delta_0\) from the expansion and using
\(\norm{A_H}_{\op}\le C\) gives
\[
 \norm{\Delta_1-\Delta_0}_{M_\kappa}
 \le C\,(h+h^2)\,\norm{\Delta_0}_{M_\kappa}
 \le Ch\,\norm{\Delta_0}_{M_\kappa}\,,
\]
which is \eqref{eq:depth-K-increment-coupling}.

\textbf{Discretization error.}
Start the phase from \(\Pi_\eta\).
The first OU half-refresh preserves this law.  Let \(X_t\),
\(0\le t\le h\), be the position along the exact Hamiltonian trajectory,
and put $g(t) \deq g_\eta(X_t)$.
The interpolation estimate already proved in
\eqref{eq:rhmc-proof-interpolation-error} gives
\begin{equation}
 \mathcal E_{J,q}
 \deq \sup_{0\le t\le h}\,
 \norm{g(t)-\mathcal I_{J,h}g(t)}_{L^q}
 \le B_{J,q}\sqrt{d+q}\,
       \Bigl(\frac h{\sqrt\eta}\Bigr)^J\,.
 \label{eq:depth-K-interpolation}
\end{equation}

Let \(\mathsf X\deq(X_{t_j})_{j\in[J]}\) and
\(\mathsf X^{[0]}\deq(X_{t_j}^{[0]})_{j\in[J]}\), and use the norm given by
the maximum over nodes of the \(L^q\)-norm.  The exact Volterra equation
and polynomial interpolation give
\begin{equation}
 \mathsf X
 =\mathsf X^{[0]}-\Omega g_\eta(\mathsf X)+\varrho\,,
 \qquad
 \norm{\varrho}\le\frac{h^2}{2}\,\mathcal E_{J,q}\,.
 \label{eq:depth-K-exact-node-equation}
\end{equation}
The map $\mathcal T :\mathsf y \mapsto \mathsf X^{[0]}-\Omega g_\eta(\mathsf y)$ has Lipschitz constant at most \(h^2\), because \(g_\eta\) is
\(1\)-Lipschitz.  Let \(\overline{\mathsf X}\) be its fixed point.  From
\eqref{eq:depth-K-exact-node-equation},
\begin{equation}
 \norm{\overline{\mathsf X}-\mathsf X}
 \le\frac{h^2}{2\,(1-h^2)}\,\mathcal E_{J,q}\,.
 \label{eq:depth-K-fixed-exact}
\end{equation}
The proof of \cref{prop:rhmc-proof-deterministic-defect} gives
\(\sup_{0\le t\le h}\,\norm{g(t)}_{L^q}\le C\sqrt{d+q}\).  Combining this with
\eqref{eq:depth-K-fixed-exact} and the fixed-point equation yields
\begin{equation*}
 \norm{\mathsf X^{[0]}-\overline{\mathsf X}}
 \le Ch^2\sqrt{d+q}
      +Ch^2\mathcal E_{J,q}\,.
\end{equation*}
The recursion \eqref{eq:depth-K-iteration} is precisely the iteration of
\(\mathcal T\), starting from \(\mathsf X^{[0]}\).  Hence
\begin{equation}
 \norm{\mathsf X^{[K-1]}-\overline{\mathsf X}}
 \le Ch^{2K}\,
       \{\sqrt{d+q}+\mathcal E_{J,q}\}\,.
 \label{eq:depth-K-Picard-remainder}
\end{equation}

The exact Hamiltonian endpoint satisfies
\begin{align*}
 X_h&=X_{\rm init}+hP_0-\int_0^h(h-t)g(t)\,\dd t\,, \qquad
 P_h=P_0-\int_0^h g(t)\,\dd t\,.
\end{align*}
The quadrature rules integrate \(\mathcal I_{J,h}g\) exactly.  Use
\eqref{eq:depth-K-fixed-exact}--\eqref{eq:depth-K-Picard-remainder},
\(1\)-Lipschitzness of \(g_\eta\), and the coefficient bounds in
\eqref{eq:rhmc-chebyshev-bounds}.  The momentum estimate dominates and
gives
\begin{equation}
 \norm{(X_{\rm init}^+,P_{\rm init}^+)-(X_h,P_h^+)}_{L^q,M_\kappa}
 \le C\,
 \bigl\{h\mathcal E_{J,q}
 +h^{2K+1}\sqrt{d+q}\bigr\}\,.
 \label{eq:depth-K-endpoint-error}
\end{equation}
Here \(P_h^+\) includes the common final OU half-refresh.  The terms in
\eqref{eq:depth-K-Picard-remainder} containing \(\mathcal E_{J,q}\) are
absorbed into the first term because \(h^2\le1/4\) and
\(h\le1\).
Substituting
\eqref{eq:depth-K-interpolation} into
\eqref{eq:depth-K-endpoint-error} proves
\eqref{eq:depth-K-defect}.
\end{proof}

\bibliographystyle{alpha}
\bibliography{ref}

\end{document}